\documentclass{article}

\usepackage[english]{babel}
\usepackage{wrapfig}

\usepackage[letterpaper,top=2cm,bottom=2cm,left=3cm,right=3cm,marginparwidth=1.75cm]{geometry}
\usepackage[toc,page]{appendix}

\usepackage[algo2e,linesnumbered,vlined,ruled]{algorithm2e}
\usepackage{amsmath,amssymb,amsthm,color}
\usepackage{graphicx}
\usepackage{subcaption}
\usepackage[colorlinks=true, allcolors=blue]{hyperref}
\usepackage{mdframed}
\usepackage{enumitem}

\newtheorem{theorem}{Theorem}[section]
\newtheorem{lemma}[theorem]{Lemma}
\newtheorem{corollary}[theorem]{Corollary}
\newtheorem{proposition}[theorem]{Proposition}
\newtheorem{definition}[theorem]{Definition}
\newtheorem{assumption}[theorem]{Assumption}

\newtheorem{example}[theorem]{Example}
\usepackage{algorithm}
\usepackage{algpseudocode}
\usepackage{verbatim}
\newcommand{\cX}{\mathcal{X}}
\newcommand{\cA}{\mathcal{A}}
\newcommand{\cB}{\mathcal{B}}
\newcommand{\cM}{\mathcal{M}}

\newcommand{\cG}{\mathcal{G}}

\newcommand{\cGphl}{\mathcal{G}_{\phi,h}^\lambda}

\newcommand{\cDphl}{\mathcal{D}_{\phi,h}^\lambda}
\newcommand{\cDphlk}{\mathcal{D}_{\phi,h}^{\lambda_k}}
\newcommand{\Tphl}{\mathbf{T}_{\phi,h}^\lambda}
\newcommand{\Tphlk}{\mathbf{T}_{\phi,h}^{\lambda_k}}

\newcommand{\Tshe}{\mathbf{T}_{\phi+\mathrm{id}_{\cX_s}\!,h}^\eta}
\newcommand{\cGphe}{\mathcal{G}_{\phi,h}^\eta}
\newcommand{\cGshe}{\mathcal{G}_{\phi+\mathrm{id}_{\cX_s}\!,h}^\eta}

\newcommand{\RR}{\mathbb{R}}
\newcommand{\EE}{\mathbb{E}}

\newcommand{\prox}{\mathbf{prox}}

\newcommand{\Tshesk}{\mathbf{T}_{\phi,\,h}^{\eta_{s,k}}}

\newcommand{\tcDshesk}{\widetilde{\mathcal{D}}_{\phi,\,h}^{\eta_{s,k}}}

\newcommand{\cDshesk}{\mathcal{D}_{\phi,\,h}^{\eta_{s,k}}}

\newcommand{\trinorm}[1]{\left\lvert\left\lvert\left\lvert #1 \right\rvert\right\rvert\right\rvert}
\usepackage{xcolor}

\newcommand{\argmin}{\mathop{\mathrm{argmin}}}

\title{Stochastic Bregman Proximal Gradient Method Revisited: \\Kernel Conditioning and Painless Variance Reduction }
\author{Junyu Zhang\thanks{National University of Singapore, junyuz@nus.edu.sg}}

\date{}

\begin{document}

\maketitle

\begin{abstract}
We investigate stochastic Bregman proximal gradient (SBPG) methods for minimizing a finite-sum nonconvex function $\Psi(x):=\frac{1}{n}\sum_{i=1}^nf_i(x)+\phi(x)$, where $\phi$ is convex and nonsmooth, while $f_i$, instead of gradient global Lipschitz continuity, satisfies a smooth-adaptability condition w.r.t. some kernel $h$. Standard acceleration techniques for stochastic algorithms (momentum, shuffling, variance reduction) depend on bounding stochastic errors by gradient differences that are further controlled via Lipschitz property. Lacking this, existing SBPG results are mostly limited to vanilla stochastic approximation schemes that cannot obtain the optimal $O(\sqrt{n})$ complexity dependence on $n$. Moreover, existing works report complexities under various nonstandard stationarity measures that largely deviate from the standard minimal limiting Fr\'echet subdifferential $\mathrm{dist}(0,\partial\Psi(\cdot))$. Our analysis 
reveals that these popular nonstandard stationarity measures are often much smaller than $\mathrm{dist}(0,\partial\Psi(\cdot))$ by a large or even unbounded instance-dependent mismatch factor, leading to overstated solution quality and producing non-stationary output. This also implies that current complexities based on nonstandard measures are actually asymptotic and instance-dependent if translated to $\mathrm{dist}(0,\partial\Psi(\cdot))$. To resolve these issues, we design a new gradient mapping $\mathcal{D}_{\phi,h}^\lambda (\cdot)$ by BPG residuals in dual space and a new kernel-conditioning (KC) regularity, under which the mismatch between $\|\mathcal{D}_{\phi,h}^\lambda (\cdot)\|$ and  $\mathrm{dist}(0,\partial\Psi(\cdot))$ is provably $O(1)$ and instance-free. Moreover, KC-regularity guarantees Lipschitz-like bounds for gradient differences, providing general analysis tools for momentum, shuffling, and variance reduction under smooth-adaptability. We illustrate this point on variance reduced SBPG methods and establish an $O(\sqrt{n})$ complexity dependence for $\|\mathcal{D}_{\phi,h}^\lambda (\cdot)\|$, providing instance-free (worst-case) complexity under  $\mathrm{dist}(0,\partial\Psi(\cdot))$. 
\end{abstract}

\section{Introduction}
In this paper, we consider the composite  nonconvex optimization problem
\begin{equation}
    \label{prob:main-det}
    \min_{x\in\RR^d} \Psi(x) = f(x) + \phi(x) \qquad\mbox{with}\qquad f(x) = \frac{1}{n}\sum_{i=1}^n f_i(x),
\end{equation} 
where $\phi(x)$ is a convex but possibly non-differentiable function, while $f$ and each $f_i$ are nonconvex and continuously differentiable.
In particular, we consider the problem class where the gradient $\nabla f$ (or $\nabla f_i$) is not globally Lipschitz continuous. With various applications to optimizing log-determinant of
Fisher information matrix \cite{hanzely2021accelerated}, D-optimal design and generalized volumetric optimization \cite{lu2018relatively}, quadratic inverse problem \cite{bolte2018first}, 
multi-layer neural networks \cite{ding2023nonconvex,mukkamala2019bregman}, etc., this problem setting has  drawn increasing interest recently. 

Under classic nonconvex composite finite-sum setting where each $f_i$ has globally Lipschitz continuous gradients, complexities of stochastic first-order methods are well-understood for problem \eqref{prob:main-det}. Define the proximal operator and the gradient mapping as 
\begin{equation}
    \label{defn:GradMap-standard}
    \prox_{\lambda\phi}(v) := \argmin_{x\in\RR^d} \phi(x) + \frac{1}{2\lambda}\|x-v\|^2\quad\mbox{and}\quad G^\lambda_\phi(x):=\frac{x-\prox_{\lambda\phi}(x-\lambda\nabla f(x))}{\lambda}\,.
\end{equation}
To obtain an expected $\epsilon$-stationary point  $\Bar{x}$ s.t. $\mathbb{E}\big[\|G^\lambda_\phi(\Bar{x})\|^2\big]\leq \epsilon$, the vanilla mini-batch  stochastic approximation (SA) scheme requires $O(\epsilon^{-2})$ samples \cite{ghadimi2016mini}. Common techniques to accelerate SA include momentum, random shuffling, and stochastic variance reduction, etc. Among these techniques, momentum often speeds up practical performance while not improving theoretical complexity \cite{liu2020improved}; 
Random shuffling achieves an improved sample complexity of $O(\sqrt{n}\epsilon^{-1.5})$ in case $\phi=0$ \cite{mishchenko2020random} or an $O(n\epsilon^{-1.5})$ sample complexity for general nonsmooth convex $\phi$ \cite{mishchenko2022proximal}; Various stochastic variance reduction techniques can further improve the sample complexity to  $O(\sqrt{n}\epsilon^{-1})$ \cite{nguyen2019optimal,pham2020proxsarah,wang2018spiderboost,cutkosky2019momentum} regardless of the nonsmooth component $\phi$,  matching the information theoretic lower bounds \cite{arjevani2023lower,zhou2019lower}.

In the absence of global gradient  Lipschitz continuity, Bolte and Nesterov proposed two equivalent concepts called smooth-adaptability (smad) \cite{bauschke2017descent} and relative smoothness \cite{lu2018relatively}, respectively.  This condition allows $f$ to behave smoothly relative to the Bregman divergence of some general kernel function $h$. Then the (deterministic) Bregman proximal gradient (BPG) method was proposed as:
\begin{equation}
    \label{defn:BPG-standard}
    x_{k+1} = \mathbf{T}_{\phi,h}^\lambda\big(x_k, \nabla f(x_k)\big)\qquad\mbox{with}\qquad \mathbf{T}_{\phi,h}^\lambda (x,v) := \argmin_{y\in\RR^d}\,\, y^\top v + \phi(y) + \lambda^{-1}D_h(y,x),\vspace{-0.1cm}
\end{equation} 
where $D_h(y,x):=h(y)-\nabla h(x)^\top(y-x)-h(x)$ stands for the Bregman divergence induced by $h$. 

On the one hand, the smooth-adaptability condition ensures a generalized descent lemma  \cite{bauschke2017descent}, leading to an $O(1/\epsilon)$ complexity for a wide range of deterministic BPG variants \cite{bolte2018first,teboulle2018simplified,gao2020randomized,gao2021convergence,mukkamala2020convex,leinertial} under several different nonstandard stationarity measures, a popular example is the  (squared) size of the following  Bregman proximal gradient mapping suggested by \cite[Section 4.1]{bolte2018first}:
\begin{equation}
    \label{defn:GradMap-BPG}
    \mathcal{G}_{\phi,h}^\lambda(x):= \frac{x-\mathbf{T}_{\phi,h}^\lambda\big(x,\nabla f(x)\big)}{\lambda}.  
\end{equation} 
As $\cGphl(\cdot)$ is defined by the residual of (primal) BPG iterates, we call it primal gradient mapping for simplicity. It is worth noting that these nonstandard measures may significantly deviate from the standard minimal squared limiting Fr\'echet subdifferential measure $\mathrm{dist}^2(0,\partial\Psi(\cdot))$, which is abbreviated as Fr\'echet measure in later discussion, see our detailed analysis in Section \ref{section:preliminary}. 

On the other hand, existing results on stochastic BPG algorithms are still limited to the vanilla SA schemes, with only $O(\epsilon^{-2})$ sample complexities \cite{davis2018stochastic,ding2023nonconvex,xiao2021unified,zhang2018convergence,fatkhullin2024taming}, under a variety of nonstandard measures. In terms of the attempts to exploit stochastic variance reduction techniques, \cite{latafat2022bregman} adopted a Finito/MISO scheme for stochastic BPG method. Only asymptotic convergence and $O(n)$-dependence has been obtained for nonconvex problems. In a more recent attempt  \cite{wang2024bregman}, the authors directly assume the knowledge of a uniform upper bound of local Lipschitz constants for all (stochastic) iterations and explicitly use it in algorithm design, leading to chicken-and-egg issues and essentially limiting their analysis to classic mirror descent with globally Lipschitz continuous gradient. Overall, it is still not clear how to improve the complexity of stochastic BPG beyond the vanilla SA results. 

Based on above discussion, there are  two clear questions, about the nonstandard stationarity measures and the general tool to accelerate stochastic BPG. Let us discuss them in detail one by one.\vspace{0.2 cm}

\noindent\textbf{Instance dependent or independent? }$\,$
Since Nemirovski's landmark book \cite{nemirovskij1983problem} and Nesterov's early discussion on optimal algorithms \cite{nesterov1983method,nesterov2018lectures} in optimization community, the complexity concept has become formalized, mostly referring to complexity bound of the \emph{worst-case} instance in the considered \emph{problem class}. Such a bound is regardless of which instance in the problem class is picked and is hence \emph{instance-free}. Though for various reasons, people still derive \emph{instance-dependent} complexity that relies on the property of each specific input instance, the instance-free (worst-case) complexity is equally important due to its robustness against hard instances. Examples include the instance dependent or independent bounds for bandit optimization \cite[etc.]{agrawal2012analysis,kaufmann2012thompson,bubeck2013prior}, reinforcement learning \cite[etc.]{wang2022gap}, two-person Markov games \cite{dou2022gap,xie2020learning}, see detailed discussion in Section \ref{subsec:issues}. Though this issue has long been  overlooked in BPG literature, our analysis shows that the current complexity results all exhibit an instance-dependent nature hidden under the local scaling of the nonstandard  measures, leaving the instance-free (worst-case) complexity an open question for deterministic and stochastic BPG.  

In details, given any kernel $h$ and $L>0$, the interested problem class $\mathcal{P}$ consists of all problem instances of \eqref{prob:main-det} with convex $\phi$ and $L$-smooth adaptable $f$ w.r.t. kernel $h$. For example, the classic $L$-smooth problem class is included by the quadratic kernel $h(x)=\|x\|^2/2$.  Due to technical difficulties in analyzing complexity for bounding $\mathrm{dist}^2(0,\partial\Psi(\cdot))$, existing BPG literature typically circumvents this challenge by adopting alternative nonstandard stationarity measures from the algorithmic residuals of primal iterates, including the (squared) primal gradient mapping  \cite{bolte2018first,gao2020randomized,gao2021convergence,teboulle2018simplified,ding2023nonconvex}, the (squared) local norm of some envelope's gradient \cite{davis2018stochastic}, the scaled Bregman divergence of consecutive iterations \cite{bolte2018first,gao2020randomized,gao2021convergence,teboulle2018simplified} and its symmetrized variant \cite{zhang2018convergence}. Such an inconsistency makes it hard to interpret and compare the complexities in different works. To resolve such ambiguity, we propose to calibrate an arbitrary stationarity measure $\cM(x)$ against the standard Fr\'echet measure by a mismatch factor $C_\cM(x):=\frac{\mathrm{dist}^2(0,\partial\Psi(x'))}{\cM(x)}$, where $x$ is typically an algorithm's output and $x'$ may differ from but depend on $x$ in the nonsmooth instances. The introduction of this factor removes the local scaling effect in different stationarity measures, facilitating a clearer understanding of existing complexity results under a unified perspective. Then for the interested problem class $\mathcal{P}$, a natural question is:

\begin{mdframed}[backgroundcolor=gray!20,leftmargin=15pt, rightmargin=15pt]
  \textbf{Q-1.} Is the mismatch factor $C_\cM$ uniformly bounded over $\mathcal{P}$ for the popular nonstandard measures in current literature?  If not, how to interpret the existing complexity results? Can we find a new measure with constant mismatch and analyze complexity under it?
\end{mdframed}
This question is closely related to the instance-dependent nature of existing works. Take the squared primal gradient mapping $\cM(\cdot) = \|\cGphl(\cdot)\|^2$ for example, existing results output a point $\bar{x}$ such that $\|\cGphl(\bar{x})\|^2\leq \epsilon$ with   $O(\epsilon^{-2})$   samples \cite{davis2018stochastic,ding2023nonconvex,xiao2021unified}. If the mismatch $C_{\cM}$ is uniformly bounded for the output over all instances in $\mathcal{P}$, then existing results immediately translates to the same  $O(\epsilon^{-2})$ sample complexity for finding some $\bar{x}$ such that $\mathrm{dist}^2(0,\partial\Psi(\bar{x}))\leq \epsilon$, except for an additional constant factor (mismatch upper bound) hidden in big-O. Such a result will be independent of the instances in $\mathcal{P}$, and is thus an instance-free  (worst-case) complexity. Unfortunately, this does not hold.  

In general, our analysis shows a strong instance-dependent nature for the above discussed nonstandard measures. Their mismatch to the standard Fr\'echet measure can vary drastically among the different instances in $\mathcal{P}$ and can potentially be unbounded in the worst-case. This indicates that all the reviewed \emph{seemingly} instance-free $O(\epsilon^{-2})$ complexities based on nonstandard measures are  actually instance-dependent, with dependency hidden in the local scaling of stationarity measures. And they are not able to infer the worst-case complexity under the Fr\'echet measure. \vspace{0.2cm}

\noindent\textbf{How to provably accelerate stochastic BPG? }$\,\,$  Current  framework for analyzing stochastic BPG methods is mostly the classic vanilla SA approach based on the generalized descent lemma \cite{bauschke2017descent} and the bounded variance assumption, see e.g. \cite{davis2018stochastic,ding2023nonconvex,xiao2021unified,zhang2018convergence,fatkhullin2024taming}. However, more recent acceleration techniques for stochastic first-order algorithms, like momentum \cite{liu2020improved}, shuffling \cite{mishchenko2020random,mishchenko2022proximal}, and variance reduction \cite{nguyen2019optimal,pham2020proxsarah,wang2018spiderboost,cutkosky2019momentum}, 
all share a same analysis workflow that bounds the stochastic errors by gradient differences, and then controls gradient differences by iterative descents through gradient Lipschitz property. Unfortunately, the smooth-adaptability condition alone is not able to ensure an appropriate Lipschitz-like bound for gradient differences. Due to this issue, even under the nonstandard (instance-dependent) stationarity measures, only the typical $O(\epsilon^{-2})$ complexity for vanilla SA approach was obtained. Therefore, the second question in this paper arises:
 
\begin{mdframed}[backgroundcolor=gray!20,leftmargin=15pt, rightmargin=15pt]
  \textbf{Q-2.} Is there a natural kernel regularity condition that is able to ensure an appropriate Lipschitz-like bound for gradient differences under smooth-adaptability? 
\end{mdframed}
If such a Lipschitz-like bound can be established, then it may function as a general tool for analyzing momentum, shuffling or variance reduced stochastic BPG methods under smooth-adaptability.  \vspace{0.2cm}

\noindent\textbf{Main contribution. } To resolve the  two questions, we propose to investigate the BPG method from the dual space. For Q-1, we propose a new Bregman proximal gradient mapping  
\begin{equation}
    \label{defn:GradMap-New}
    \cDphl(x) := \frac{\nabla h(x)-\nabla h \big(\mathbf{T}_{\phi,h}^\lambda \big(x,\nabla f(x)\big)\big)}{\lambda}
\end{equation}
defined by the dual residual of the BPG step, and we will call it the dual gradient mapping in this paper. In addition, we also introduce and thoroughly discuss a novel \emph{kernel conditioning} (KC) regularity condition on the kernel function $h$ that has yet been studied in existing results. We prove that $\cDphl(\cdot) = \nabla \Psi(\cdot)$ always hold when  $\phi = 0$. For the general nonsmooth case where $\phi\neq0$, we use KC-regularity condition to establish a  uniformly bounded mismatch factor between the squared dual gradient mapping and Fr\'echet measure for all instances in the problem class $\mathcal{P}$. As this gradient mapping emerges from the dual perspective of BPG iterations, it  naturally involves in the descent terms and is very convenient in the analysis. For Q-2, we start from the KC-regularity and establish a new uniform local Lipschitz-like bound, linking gradient differences and Bregman divergence (or another non-Euclidean distance). It is local in the sense that the Lipschitz-constant-like coefficient is only bounded for $\delta$-close points given any predetermined $\delta>0$. On the other hand, it is also uniform, or global, in the sense that this upper bound remains the same for all $\delta$-close points throughout the whole space. Therefore, there will always be a Lipschitz-like bound for gradient differences as long as one prevents too aggressive update. This provides a general tool for analyzing the SA acceleration techniques that are based on gradient difference bound, like momentum, shuffling, and variance reduction. Given the above development of a general analysis toolkit, we choose the variance reduction technique to illustrate how they can be applied to speed up stochastic BPG.

First, due to the technical simplicity and wide popularity of the instance-dependent complexity analysis in current BPG research, we introduce a simple mechanism that facilitates the analysis of many existing variance reduction techniques for providing instance-dependent bound based on the most popular nonstandard stationarity measure, the squared primal gradient mapping. Via the gradient difference bound ensured by KC-regularity and a novel probabilistic analysis, we prove that there exists a high probability event $\mathcal{A}$, conditioning on the success of which the proposed methods output a point $\Bar{x}$ such that $\EE\big[\|\cG_{\phi,h}^\lambda(\Bar{x})\|^2 | \mathcal{A}\big]\leq\epsilon$ with an $O(\sqrt{n}\epsilon^{-1})$ sample complexity. As the total iteration number $T\to+\infty$, the success probability $\mathrm{Prob}(\cA)\to1$ automatically with a sublinear rate. In particular, the $O(\sqrt{n})$ dependence on $n$ is optimal because it matches the complexity lower bound for stochastic finite-sum optimization \cite[Theorem 4.7]{zhou2019lower}. 

Second, to obtain a robust and stable complexity bound that works for all possible instances, we slightly modify the algorithmic parameters of the above variance reduced BPG and establish an $\widetilde{O}\big(\sqrt{n}L_\epsilon\epsilon^{-1}\big)$ sample complexity finding some $\bar{x}$ such that $\|\cDphl(\bar{x})\|^2\leq\epsilon$, for some $L_\epsilon$ factor. For mild instances, the $L_\epsilon$ factor can be viewed a constant. But in the worst case,   $L_\epsilon$ potentially contains extra $\epsilon$-dependence and may scale as the maximum kernel Hessian norm in an $O(1/\sqrt{\epsilon})$-radius region. For example, for problem class associated with an $r$-degree polynomial kernel, $L_\epsilon$ can scale as $O(\epsilon^{1-r/2})$ and is proved to tight by constructing a worst-case instance. Due to the guaranteed $O(1)$ mismatch under KC-regularity, this result directly translates to the an $\widetilde{O}\big(\sqrt{n}L_\epsilon\epsilon^{-1}\big)$ instance-free (worst-case) sample complexity based on Fr\'echet measure, closing a gap in the complexity theory of  BPG method. \vspace{0.1cm}

\noindent\textbf{Other related works.} In this paragraph, we review a few works on convex optimization without globally Lipschitz gradients, which are related but not closely related to our paper. First, within the scope of BPG type methods,  \cite{lu2018relatively,bauschke2017descent} were   concurrently the first to propose the notion of relative smoothness (or smooth adaptability). They derived an $O(1/T)$ sublinear convergence for general convex case and a linear convergence for strongly convex case. If the objective function satisfy a so-called triangle scaling property, \cite{hanzely2021accelerated} further proposed an accelerated BPG method with improved rates. In \cite{lu2019relative,hanzely2021fastest}, the authors discussed the sample complexity of stochastic BPG and its coordinate descent variant under (strong) convexity, while \cite{dragomir2021fast} studied the stochastic variance reduced BPG method for optimizing the average of $n$ smooth functions and an optimal $O(\sqrt{n})$ dependence on $n$ has been obtained. However, \cite{dragomir2021fast} relies on an abstract technical assumption that is hard to verify and interpret. 
\vspace{0.1cm}

\noindent\textbf{Organization.} In Section \ref{section:preliminary}, we start with some basic definitions and properties of the smooth adaptable functions, and then provide a thorough discussion on the instance-dependent nature of existing results, our kernel-conditioning regularity assumption, and the new dual gradient mapping. In Section \ref{section:VR}, we discuss how the kernel-conditioning regularity combined with a simple epoch bound mechanism can enable almost all the existing stochastic variance reduction schemes and provide the improved instance-dependent sample complexity under the squared primal gradient mapping. In Section \ref{section:GrdMap-new}, we propose novel adaptive step size control mechanisms for variance reduced method and provide instance-free sample complexities under the squared dual gradient mapping. We conclude this paper in Section \ref{section:comparison}. 
\vspace{0.1cm}

\noindent\textbf{Notations.} For $\forall x\in\RR^d$, We default $\|x\|:=\sqrt{x^\top x}$  the $\ell_2$-norm, and we denote $\|x\|_1:=\sum_i|x_i|$ as the $\ell_1$-norm. For a matrix $X\in\RR^{d\times d}$, we use $\lambda_{\max}(X)$ and $\lambda_{\min}(X)$ to denote the maximal and minimal eigenvalues of $X$, respectively. And we use $\|X\|$ to denote the $\ell_2$-operator norm of $X$. When $X$ is positive semidefinite, we write $X\succeq0$, and we have $\|X\| = \lambda_{\max}(X)$ in this case. For any set $\cX$, 
we denote $\mathrm{id}_{\cX}(\cdot)$ as the indicator function of the set.  Namely, $\mathrm{id}_{\cX}(x)=0$ if $x\in\cX$ and $\mathrm{id}_{\cX}(x) = +\infty$ if $x\notin\cX$. We denote the interior of $\cX$ as $\mathrm{int}(\cX)$ and we denote the boundary of $\cX$ as $\partial \cX$. We denote $[n]:=\{1,2,\cdots,n\}$. For any vector $x\in\RR^d$, the function $\mathrm{sign}: \RR^d\to\RR^d$ returns a sign vector of $x$. That is, for any $s=\mathrm{sign}(x)$, its $i$-th element satisfies $s_i=1$ if $x_i\geq0$ and $s_i=-1$ if $x_i<0$, for any $1\leq i\leq d$. Because many literature use the terminology $L$-smooth to denote $L$-Lipschitz continuity of the gradient, we will use ``continuously differentiable'' instead of ``smooth''  to avoid confusion. 

\section{Kernel-conditioning and  stationarity measures}
\label{section:preliminary}
\subsection{Preliminary results}
Before presenting the newly introduced kernel-conditioning regularity assumption and dual gradient mapping, let us provide a brief introduction to the basic concepts and properties of smooth adaptability and Bregman proximal gradient methods. 
\begin{assumption}[Smooth adaptability, \cite{bolte2018first}]
    \label{assumption:L-smad}
    Let $f$ and $h$ be twice continuously differentiable in $\mathbb{R}^d$, and let $h$ be strictly convex. Then we assume $f$ is $L$-smooth adaptable to $h$ for some positive constant $L>0$. In other words, both  $Lh+f$ and $Lh-f$ are convex functions. 
\end{assumption} 
\noindent Given the twice continuous differentiability of $f$ and $h$, Assumption \ref{assumption:L-smad} can be equivalently written  as 
\begin{equation}
    \label{asmp:L-smad-alt}
    -L\nabla^2h(x)\preceq \nabla^2 f(x)\preceq L\nabla^2h(x),\,\, \forall x\in\RR^d.
\end{equation} 
As we consider the problem class where $\nabla f$ is not globally Lipschitz continuous, then naturally, one would expect $\|\nabla^2 f(x)\|$ and $\|\nabla^2h(x)\|$  to grow unbounded in a $\limsup$ sense as $\|x\|\to+\infty$. A particularly interesting example that satisfies the smooth adaptability assumption is the function class with polynomially growing Hessian, as described below. 
\begin{proposition}[Proposition 2.1, \cite{lu2018relatively}]
    \label{proposition:poly-kernel}
    Suppose $f$ is a twice continuously differentiable function that satisfies $\|\nabla^2 f(x)\|\leq p_r(\|x\|)$ for some $r$-degree polynomial $p_r(\cdot)$. Let $L$ be such that $p_r(\alpha)\leq L(1+\alpha^r)$ for $\alpha\geq0$. Then the function $f$ is $L$-smooth adaptable to $h(x) := \frac{1}{2}\|x\|^2 + \frac{1}{r+2}\|x\|^{r+2}$.
\end{proposition}
\noindent The polynomial kernel $h$ is in fact 1-strongly convex over $\mathbb{R}^d$, and hence the Bregman proximal operator $\mathbf{T}_{\phi,h}^\lambda(\cdot)$ introduced in \eqref{defn:BPG-standard} is unique and well-defined. Under smooth adaptability, a generalized descent lemma was derived in \cite{bolte2018first}, which is a key property for analyzing the BPG type algorithms.  
\begin{lemma}[Extended descent lemma, \cite{bolte2018first}]
    \label{lemma:descent-lemma}
    Suppose $f$ and $h$ satisfy Assumption \ref{assumption:L-smad} for some constant $L>0$, then for any $\forall x,y\in\RR^d$, it holds that 
    $$|f(x)-f(y)-\langle\nabla f(y),x-y\rangle|\leq LD_h(x,y).$$
\end{lemma}
\noindent Similar versions of Assumption \ref{assumption:L-smad} and Lemma \ref{lemma:descent-lemma} are also established in \cite{bauschke2017descent,lu2018relatively}. Based on this lemma, the BPG method is proposed as a majorization minimization scheme:
\begin{equation}
    \label{defn:majorization}
    x_{k+1} = \mathbf{T}_{\phi,h}^\lambda\big(x_k, \nabla f(x_k)\big) = \argmin_{x\in\RR^d}\,\, f(x_k) + \langle\nabla f(x_k),x-x_k\rangle + \phi(x) + \lambda^{-1}D_h(x,x_k), 
\end{equation}
where we iteratively minimize an upper bound model of the objective function. Setting  $\lambda < 1/L$ and denoting $\Delta_\Psi:=\Psi(x_0)-\inf_x \Psi(x)$, standard analysis gives 
\begin{equation} 
    \label{eqn:BPG-summability}
    \sum_{k=0}^{T-1}D_h(x_{k+1},x_{k})\leq \frac{\Delta_\Psi}{1/\lambda-L},
\end{equation} 
indicating that $\min_{k\leq T} D_h(x_{k+1},x_{k})\leq O(1/T)$, 
see e.g. \cite[Proposition 4.1]{bolte2018first}. Such a summability result plays a central role in the complexity analysis of the BPG method for nonconvex problems.

\subsection{Instance-dependent nature of existing complexity results}
\label{subsec:issues}
Based on the summability property \eqref{eqn:BPG-summability}, many papers have developed their global convergence and complexity results for BPG and its variants, under various different stationarity measures that significantly diverge from the standard Fr\'echet measure. To obtain a thorough understanding of existing results and provide a unified interpretation for their complexity, we briefly discuss the stationarity measures in the existing works and then calibrate them against the standard Fr\'echet measure. For the ease of discussion, we limit our self to the deterministic BPG in the current subsection. 

One popular stationarity measure for BPG is the Bregman residual $B_{\lambda}(x):=\lambda^{-2}D_h\big(x_\lambda^+,x\big)$  with $x_\lambda^+:=\mathbf{T}_{\phi,h}^\lambda(x, \nabla f(x))$, which has been discussed in \cite[etc.]{bolte2018first,gao2020randomized,gao2021convergence,teboulle2018simplified,davis2018stochastic}. That is, given $x_{k}$ and $x_{k+1}$ generated by \eqref{defn:BPG-standard}, the Bregman residual uses $\lambda^{-2}D_h\big(x_{k+1},x_k\big)$ to measure stationarity and convergence. In particular, \cite{zhang2018convergence} also proposed the symmetrized Bregman residual as a stationarity measure: $M_\lambda(x):=\frac{1}{\lambda^2}\big(D_h(x_\lambda^+,x)+D_h(x,x_\lambda^+)\big)$. According to \eqref{eqn:BPG-summability}, finding some solution $x\in\RR^d$ such that $B_{\lambda}(x)\leq \epsilon$ or $M_\lambda(x)\leq \epsilon$ will take at most $O(\frac{L\Delta_\Psi}{\epsilon})$  iterations. 

Another widely adopted stationarity measure is the squared primal gradient mapping size $\|\cGphl(x)\|^2$, see definition in \eqref{defn:GradMap-BPG}. This measure is often discussed together with $D_h\big(x_\lambda^+,x\big)$ or $M_\lambda(x)$ while assuming the kernel $h$ to be globally $\mu$-strongly convex, which yields
$$\|\cGphl(x)\|^2 \leq \frac{2}{\mu}B_\lambda(x)\qquad\mbox{and}\qquad\|\cGphl(x)\|^2 \leq \frac{4}{\mu}M_\lambda(x).$$
As a result, finding  $x\in\RR^d$ such that $\|\cGphl(x)\|^2\leq \epsilon$ also takes $O(\frac{L\Delta_\Psi}{\epsilon})$ iterations. This type of results and their variants can be widely observed in the literature, see \cite[etc.]{bolte2018first,gao2020randomized,gao2021convergence,teboulle2018simplified,ding2023nonconvex,davis2018stochastic}. In particular, when $\phi = 0$ and $\Psi$ is differentiable, \cite[Section 4]{davis2018stochastic} provided an alternative justification for this stationarity measure. Define the envelope function 
$$e_{\lambda\Psi}^{h}(x) := \min_{y} \Psi(y) + \frac{1}{\lambda}D_h(y,x)$$
as a surrogate of the objective function. Then \cite{davis2018stochastic} proposed to measure stationarity by a squared local norm  $\|\nabla e_{\lambda\Psi}^{h}(x)\|^2_x$ where $\|v\|_x:=\|[\nabla^2h(x)]^{-1}v\|$ for any $v\in\RR^d$. By \cite[Theorem 4.1]{davis2018stochastic}, we can derive\vspace{0.1cm}
$$\nabla e_{\lambda\Psi}^{h}(x)=\nabla^2h(x)\cdot\cGphl(x),$$
hence indicating $\|\cGphl(x)\|^2 = \|\nabla e_{\lambda\Psi}^{h}(x)\|^2_x$ and one can interpret the squared primal gradient mapping size as a scaled and squared gradient of certain surrogate envelope function.

Because of the existence of various different stationarity measures, to better understand and compare the existing results, it is necessary to compare them with a same standard benchmark. 

\begin{definition}[Limiting Fr\'{e}chet subdifferential \cite{Kruger2003}] Let $\Psi$ be a lower semicontinuous function that is potentially non-convex. A vector $u$ is said to be a Fr\'{e}chet subgradient of $\Psi$ at $x \in dom(\Psi)$ if 
$$\Psi(x+\Delta x) \geq \Psi(x)+ u^\top\Delta x  + o\left(\|\Delta x\|\right).$$
The set of Fr\'{e}chet subgradient of $\Psi$ at $x$ is called the Fr\'{e}chet subdifferential and is denoted as $\hat{\partial}\Psi(x)$. Then the limiting Fr\'{e}chet subdifferential denoted by $\partial \Psi(x)$ is defined as 
$$\partial \Psi(x)=\{v:\mathrm{there~}\exists~ x_k \rightarrow x ~\mathrm{and}~ v_k \in \hat{\partial}\Psi(x_k)~ \mathrm{s.t.}~ v_k \rightarrow v\}.$$
\end{definition}
It is known that $\partial \Psi(\cdot) = \{\nabla \Psi(\cdot)\}$ when $\Psi$ is continuously differentiable,  and $\partial \Psi(\cdot)$ equals the set of convex subgradients when $\Psi$ is convex. For our additive composite setting where $\Psi = f+\phi$, it is known that $\partial\Psi(\cdot) = \nabla f(\cdot) + \partial\phi(\cdot)$. Therefore,  the standard benchmark stationarity measure should be the Fr\'echet measure $\mathrm{dist}^2(0,\partial\Psi(\cdot))$, which reduces to $\|\nabla\Psi(\cdot)\|^2$ when $\Psi$ is differentiable,  see e.g. \cite{drusvyatskiy2018error}.  To connect the above nonstandard measures like Bregman residual with the standard Fr\'echet measure, we assume the kernel $h$ to be twice continuously differentiable and introduce a few notations. Let $\mathcal{X}\subseteq\RR^d$ be a compact set, define 
\begin{equation}
    \label{defn:condition}
    \mu_h(\mathcal{X}) \,=\, \min_{x\in\cX}\,\lambda_{\min}\big(\nabla^2 h(x)\big),\quad L_h(\mathcal{X}) \,=\, \max_{x\in\cX}\,\lambda_{\max}\big(\nabla^2 h(x)\big),\quad\mbox{and}\quad \kappa_h(\cX) := \frac{L_h(\cX)}{\mu_h(\cX)}.
\end{equation}
Because the kernel $h$ is strictly convex and twice continuously differentiable over $\RR^d$, the ratio $\kappa_h(\cX)$ is always well-defined for any compact $\cX$. For any compact sets $\cX_1\subseteq\cX_2$, it is clear that $\mu_h(\cX_1)\geq\mu_h(\cX_2)$, $L_h(\cX_1)\leq L_h(\cX_2)$, and $\kappa_h(\cX_1)\leq\kappa_h(\cX_2)$. Based on this notation, we provide a technical lemma that is useful throughout the paper. 
\begin{lemma}
    \label{lemma:grd-vs-grdmap}
    For any $x,v\in\RR^d$ and any strictly convex kernel $h$, denote $x_\lambda^+(v) := \mathbf{T}_{\phi,h}^\lambda(x,v)$.  Let $[x,x_\lambda^+(v)]$ be the line segment between $x_\lambda^+(v)$ and $x$, then there exists $u\in\partial\phi(x_\lambda^+(v))$ such that  
    \begin{equation}
        \mu_h\big([x,x_\lambda^+(v)]\big)\cdot \|x-x_\lambda^+(v)\|\leq\lambda\|v+u\|\leq L_h\big([x,x_\lambda^+(v)]\big)\cdot \|x-x_\lambda^+(v)\|\,,\nonumber
    \end{equation} 
    \begin{equation}
        \sqrt{2\mu_h\big([x,x_\lambda^+(v)]\big)\cdot D_h(x_\lambda^+(v),x)}\leq \lambda\|v+u\|\leq \sqrt{2L_h\big([x,x_\lambda^+(v)]\big)\cdot D_h(x_\lambda^+(v),x)}\,\,.\nonumber
    \end{equation}   
\end{lemma}
\begin{proof}
    By the optimality condition of the subproblem $x_\lambda^+(v)=\argmin_{y\in\RR^d}\, y^\top v + \phi(y) + \frac{1}{\lambda}D_h(y,x)$, we have $0\in \partial\phi(x_\lambda^+(v))+v+\frac{1}{\lambda}\nabla_y D_h(y,x)|_{y=x_\lambda^+(v)}$. Namely, there exists  $u\in\partial\phi(x_\lambda^+(v))$ such that 
    \begin{equation}
        \label{lm:grd-vs-grdmap-1}
        \lambda(v + u) +  \big(\nabla h(x_\lambda^+(v))-\nabla h(x)\big)=0.
    \end{equation}
    Then \cite[Theorem 2.1.9]{nesterov2018lectures}, together with the definition of $\mu_h(\cdot)$ and $L_h(\cdot)$, indicates that 
    $$\mu_h\big([x,x_\lambda^+(v)]\big)\cdot \|x-x_\lambda^+(v)\|\leq\|\nabla h(x_\lambda^+(v))-\nabla h(x)\|\leq L_h\big([x,x_\lambda^+(v)]\big)\cdot \|x-x_\lambda^+(v)\|.$$
    Combining this bound with equation \eqref{lm:grd-vs-grdmap-1} proves the first inequality of Lemma \ref{lemma:grd-vs-grdmap}. Also observe that $\nabla_y^2 D_h(y,x) = \nabla^2h(y)$, we have 
    $$\mu_h([x,x_\lambda^+(v)])\cdot I \preceq \nabla^2_y D_h(y,x) \preceq L_h([x,x_\lambda^+(v)])\cdot I\quad\mbox{for}\quad \forall y\in[x,x_\lambda^+(v)].$$ 
    Combined with the fact that $\nabla_y D_h(y,x)|_{y=x} = 0$, then   \cite[Theorem 2.1.5, Eq.(2.1.10)]{nesterov2018lectures} and \cite[Theorem 2.1.10, Eq.(2.1.24)]{nesterov2018lectures} immediately indicates 
    $$\frac{\left\|\nabla_y D_h(y,x)|_{y=x_\lambda^+(v)}\right\|^2}{2L_h([x,x_\lambda^+(v)])}\leq D_h(x_\lambda^+(v),x)-D_h(x,x)\leq \frac{\left\|\nabla_y D_h(y,x)|_{y=x_\lambda^+(v)}\right\|^2}{2\mu_h([x,x_\lambda^+(v)])}.$$
    Then substituting $D_h(x,x)=0$ and $\lambda(v+u) = \nabla h(x)-\nabla h(x_\lambda^+(v)) = -\nabla_z D_h(y,x)\mid_{y=x_\lambda^+(v)}$ to the above bound proves the second inequality of Lemma \ref{lemma:grd-vs-grdmap}.  
\end{proof}
As a direct corollary of Lemma \ref{lemma:grd-vs-grdmap}, we have the following bounds on the mismatch between the popular stationarity measures and the Fr\'echet measure, whose proof is omitted.  
\begin{corollary}
\label{corollary:mismatch}
    For any $x\in\RR^d$ and strictly convex kernel $h$, let $x_\lambda^+:=\mathbf{T}_{\phi,h}^\lambda(x, \nabla f(x))$ with step size $\lambda < 1/L$. Then it holds that     $$\frac{\mathrm{dist}^2(0,\partial\Psi(x_\lambda^+))}{B_\lambda(x)}\leq  8\kappa_h(\cX)L_h(\cX),\qquad \quad \frac{\mathrm{dist}^2(0,\partial\Psi(x_\lambda^+))}{M_\lambda(x)}\leq 4\kappa_h(\cX)L_h(\cX),$$
    $$\frac{\mathrm{dist}^2(0,\partial\Psi(x_\lambda^+))}{\|\cGphl(x)\|^2}\leq 4L_h^2(\cX),$$
    where $\cX = [x,x_\lambda^+]$. When $\phi = 0$ and hence $\Psi$ is continuously differentiable, we have 
    $$\frac{\|\nabla\Psi(x)\|^2}{B_\lambda(x)}\leq  2\kappa_h(\cX)L_h(\cX),\qquad \frac{\|\nabla\Psi(x)\|^2}{M_\lambda(x)}\leq \kappa_h(\cX)L_h(\cX),\quad\mbox{and}\quad\,\,\frac{\|\nabla\Psi(x)\|^2}{\|\cGphl(x)\|^2}\leq L_h^2(\cX).$$ 
\end{corollary}

As discussed in the introduction, the mismatch factors in Corollary \ref{corollary:mismatch} stand for the ability for a stationarity measure to upper bound the Fr\'echet measure. Take the case $\phi=0$ for example, with $\cX = [x,x_\lambda^+]$, Corollary \ref{corollary:mismatch} indicates that having $\|\cGphl(x)\|^2\leq\epsilon$ only implies $\|\nabla\Psi(x)\|^2\leq L_h^2(\cX)\epsilon$. Therefore, smaller and uniformly bounded mismatch factors are always more desirable, at least for the output.
However, for general smooth-adaptable setting, the  $L_h(\cX)$ factor in Corollary \ref{corollary:mismatch} can be unbounded over $\RR^d$ for many popular non-Lipschitz-smooth kernels. Depending on the kernel Hessian around the output, if the instance is simple and the $L_h(\cX)$ factor is mild, then the existing $O(\epsilon^{-1})$ complexity in terms of $B_\lambda(\cdot)$, $M_\lambda(\cdot)$, or $\|\cGphl(\cdot)\|^2$ can be directly translated to that of the standard Fr\'echet measure. On the other hand, for hard instances where $L_h(\cX)$ is huge or even unbounded, these results will fail to provide meaningful complexity for Fr\'echet measure or finite-step predictions for solution quality. The same instance-dependence happens for the stochastic setting, where the only difference is that for stochastic algorithms whose theoretical complexity is often established for a randomly selected output from for all iterations, the mismatch should also consider all iterations.  

Though overlooked under the topic of smooth-adaptable optimization and BPG method, we would like to point out that the issue of instance-dependent and instance-free bounds is broadly recognized in many other machine learning and optimization topics, and it is worth slightly deviating from our main topic. For example, in bandit problem or policy optimization in reinforcement learning, the instance-dependency often denotes the dependency on certain gap $\epsilon_{\text{gap}}>0$ that varies from instance to instance. For bandit problem, people first prove that the well-known Thompson sampling algorithm exhibit an $O(\frac{\ln T}{\epsilon_{\text{gap}}})$ instance-dependent logarithmic regret \cite{agrawal2012analysis,kaufmann2012thompson}. However, this bound may fail to provide meaningful information as $\epsilon_{\text{gap}}$ can be arbitrarily close to 0. A few years later, people prove an $O(\sqrt{T})$ instance-free worst-case regret for Thompson sampling \cite{bubeck2013prior}. Similar regret dependency on $\epsilon_{\text{gap}}$ can also be observed for two-person Markov games \cite{dou2022gap,xie2020learning}. Finally,  for optimizing offline Markov decision process,  gap-dependent $O(\frac{1}{\epsilon\cdot\epsilon_{\text{gap}}})$ and gap-independent $O(\frac{1}{\epsilon^2})$ complexities  \cite{wang2022gap} are also observed. In terms of our smooth-adaptable problem setting, the $L_h(\cX)$ factor just plays the role of the inverse gap $\epsilon_{\text{gap}}^{-1}$, both are unknown a priori and both can go to infinity in the worst case. Therefore, by removing the hidden  instance-dependent local kernel Hessian scaling in the popular nonstandard stationarity measures, we reveal the fact that the existing BPG complexity results reviewed above, are actually all instance-dependent results that well capture the mild instances while failing to characterize hard or worst-case instances, leaving the instance-free complexity an open gap for BPG type methods.  

It is worth noting that we have omitted the $\kappa_h(\cX)$ in the above discussion of instance-dependency. On one hand, most BPG literature, including our paper, requires the kernel to be globally $\mu$-strongly convex for some $\mu>0$, this immediately indicates $\kappa_h(\cX) \leq L_h(\cX)/\mu$. On the other hand, the KC-regularity to be introduced in Section \ref{subsec:kernel-conditioning} further ensures an instance-free constant bound for $\kappa_h(\cX)$.

Finally, to conclude this subsection, we provide an example of how the hidden $L_h(\cX)$ factor affects the mismatch against the Fr\'echet measure and how it affect convergence rate for hard instances.

\begin{example}
\label{example:counterexample-1}
Consider a bivariate instance of formulation \eqref{prob:main-det} with $\phi=0$:
\begin{equation} 
    \min_{x\in\RR^2} \Psi(x) = \frac{1}{\sqrt{2}+\ln(1+x_1^2)} + x_1^\alpha x_2^2,\nonumber
\end{equation}
where $\alpha \geq4$ is an even integer. 
\end{example}
Consider the case $\alpha =4$, by Proposition \ref{proposition: Poly-kernel-condition}, direct computation gives $\|\nabla^2 \Psi(x)\| \leq 2+6\|x\|^4$ and $\Psi$ is $8$-smooth adaptable to $h(x) = \frac{\|x\|^2}{2}+\frac{\|x\|^{r+2}}{r+2}$, for $\forall r\geq4$. With initial point $x=[1,0]$, we implement the standard BPG method \eqref{defn:BPG-standard} to solve Example \ref{example:counterexample-1} with $\alpha=4$, as shown in Figure \ref{fig:stationarity-comparison-smooth}.

\begin{figure}[H]
    \centering
    \includegraphics[width=0.2325\linewidth]{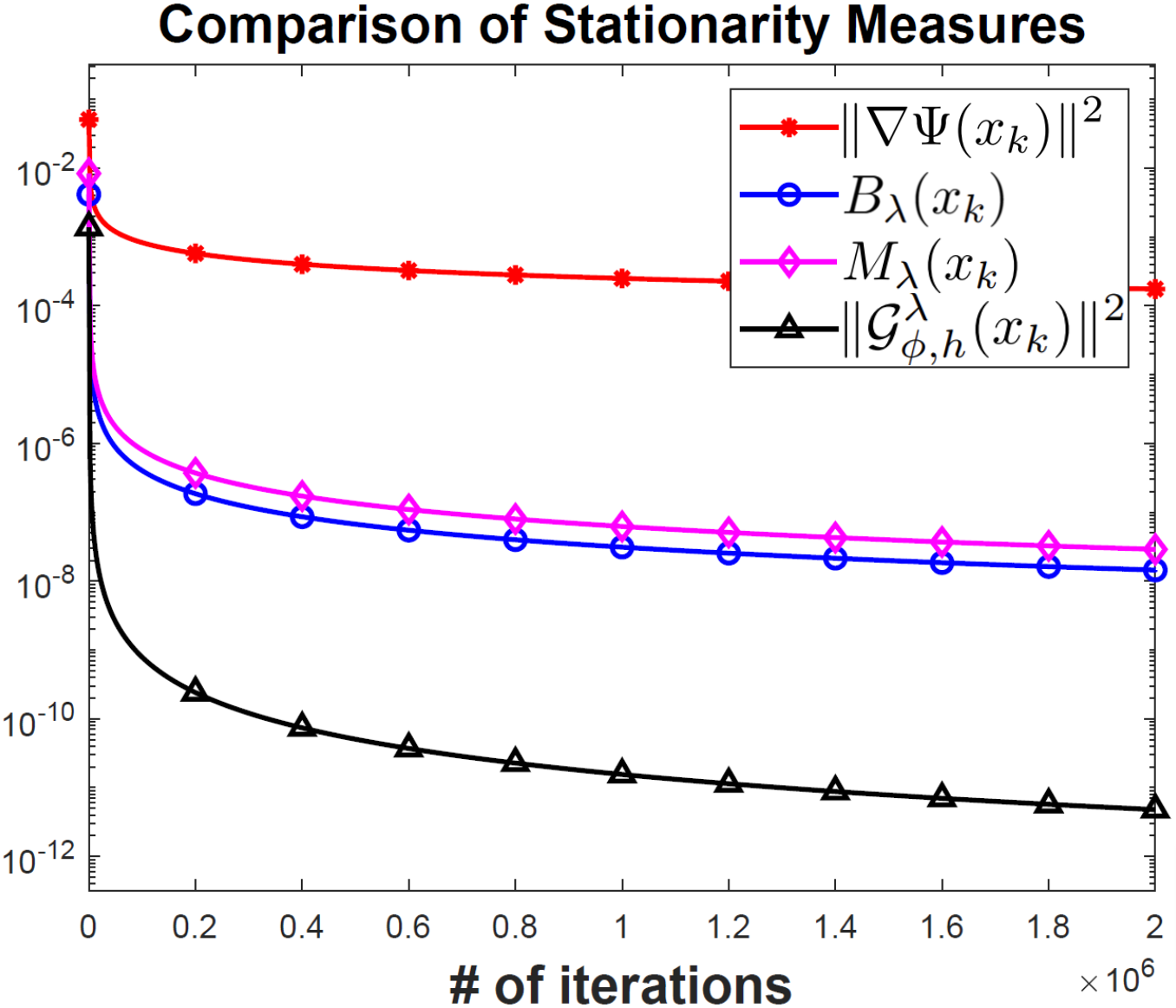} 
    \includegraphics[width=0.2445\linewidth]{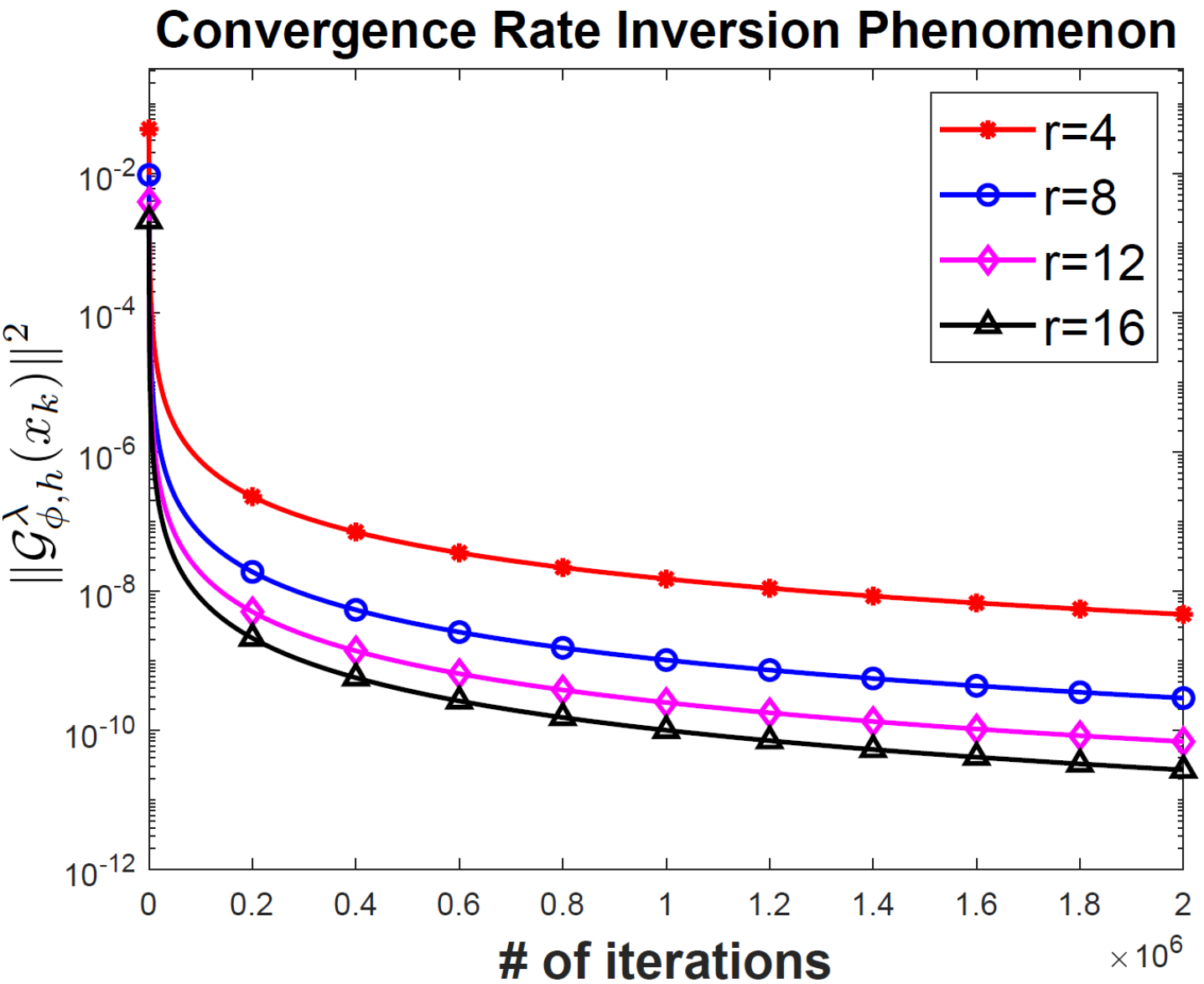}
    \includegraphics[width=0.2395\linewidth]{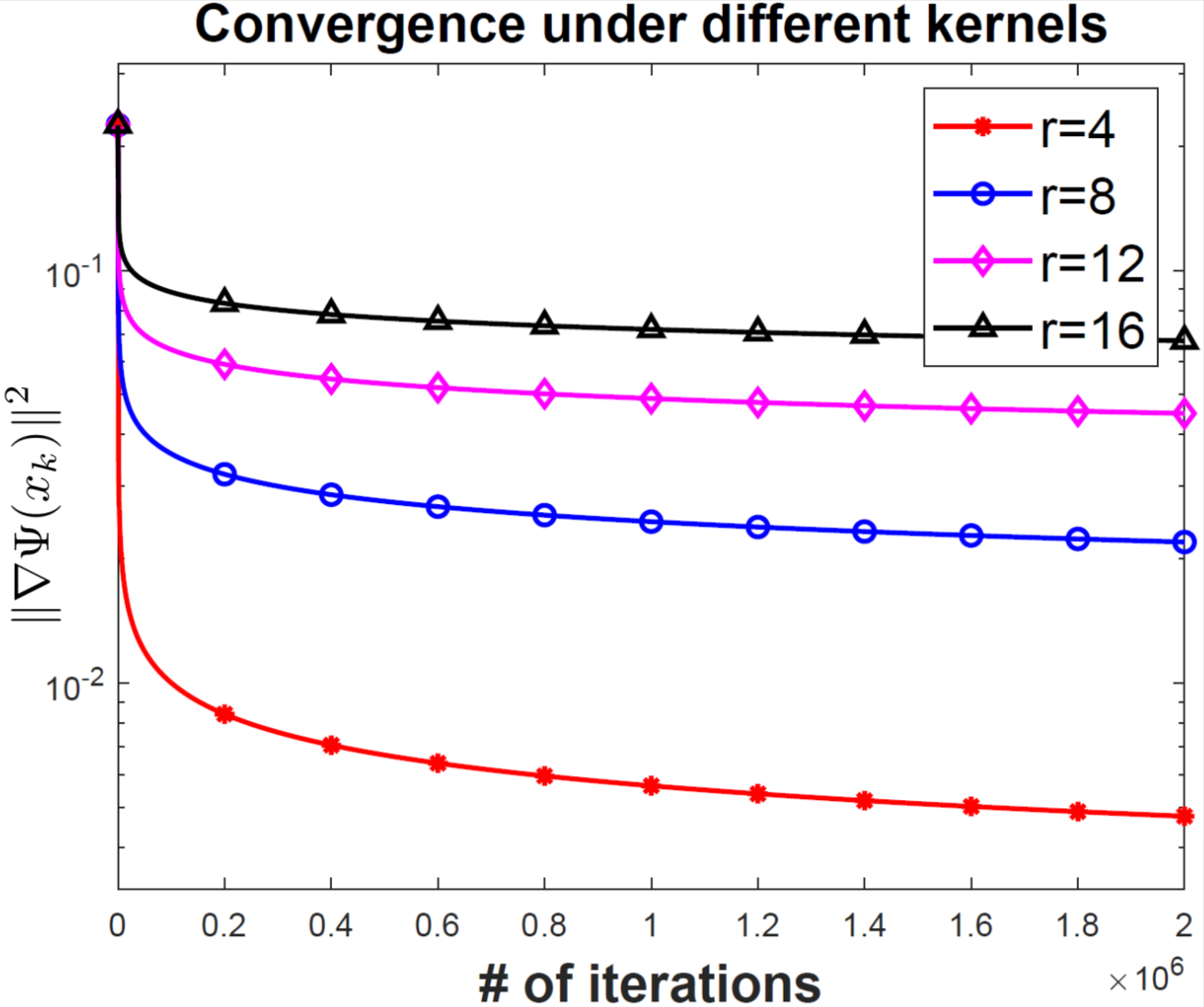}
    \includegraphics[width=0.243\linewidth]{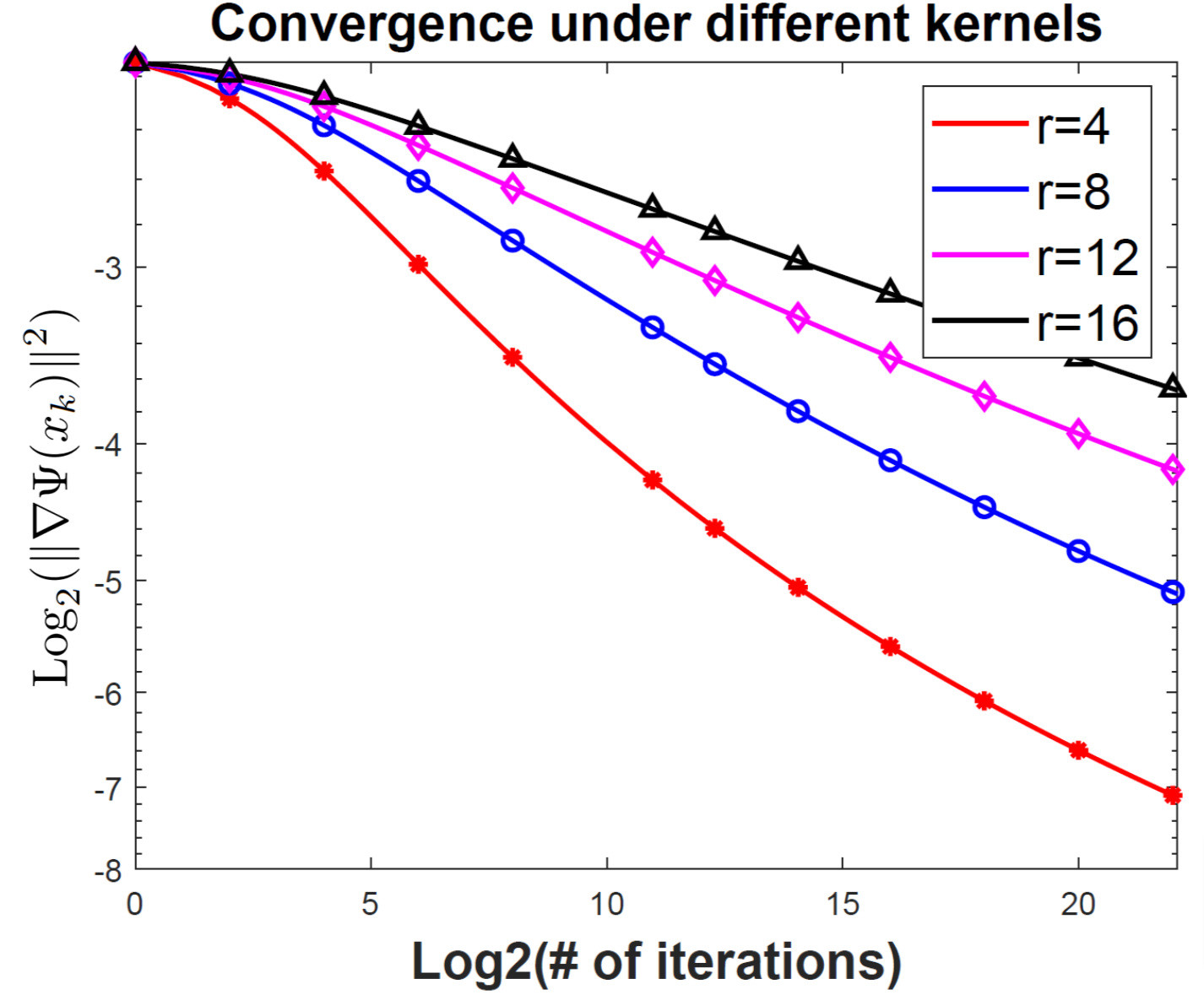}
    \caption{The first figure illustrates the mismatch between the existing stationarity measures and the squared gradient. All curves are plotted from the same sequence $\{x_k\}$ generated with kernel $r=4$. The other three figures illustrate the dependence of convergence rates on different kernels.   }
    \label{fig:stationarity-comparison-smooth}
\end{figure}

It can be observed in the first figure that even when the Bregman residuals and primal gradient mapping are small, the Fr\'echet measure, which reduces to gradient size when $\phi=0$, still remains large. For the mismatch factor discussed in Corollary \ref{corollary:mismatch}, take the squared primal gradient mapping for example, it grows to around $10^8$ while still not reaching the limit, which illustrates the mismatch issue for hard instances in our discussion. For the second and third figures, we illustrate an interesting ``rate inversion'' phenomenon, where we apply polynomial kernels of different degree $r$ to the same problem and report the convergence rate w.r.t. Fr\'echet measure and  primal gradient mapping, respectively. From the second figure, it is suggested that higher degree kernels converges ``faster'' in terms of primal gradient mapping. However, if we look at
Fr\'echet measure in the third figure, the quantity that we truly care about, kernels with lower degrees are more appropriate. Though we only plot primal gradient mapping, the same phenomenon also happens for the two Bregman residual measures. For the last figure, we plot $\log_2 \|\nabla\Psi(x_k)\|^2$ versus $\log_2 k$, whose slope represents the exponent of an $O(T^{-\gamma})$ sublinear rate. From the experiment, it is clear that the polynomial kernels of different degree $r$ result in different exponent $\gamma$ in the sublinear convergence rate, while none of them exhibits a $-1$ slope that corresponds to an $O(1/T)$ rate in the existing BPG literature. This illustrates the inability of instance-dependent bounds to characterize hard instances, while the actual worst-case complexity or convergence rate still remains unknown for BPG methods.

\subsection{A new dual gradient mapping}
Given the above discussion of several popular stationarity measures, it is crucial to discuss the convergence and complexity of BPG methods in terms of the standard Fr\'echet measure, which, unfortunately, does not directly relate to the BPG iterations. It is important to find an appropriate quantity to bridge them in the analysis. To identify such a quantity, let us take $v=\nabla f(x)$ in \eqref{lm:grd-vs-grdmap-1} and take $x_\lambda^+ = \Tphl\big(x,\nabla f(x)\big)$, then slightly rearranging the terms of \eqref{lm:grd-vs-grdmap-1} gives
\begin{equation}
    \label{eqn:grad-map-1}
    \frac{\nabla h(x)-\nabla h(x_\lambda^+)}{\lambda}\in\nabla f(x) + \partial \phi(x_\lambda^+).
\end{equation}
When $\phi = 0$ and $\Psi = f$ is continuously differentiable, the above equality reduces to 
\begin{equation}
    \label{eqn:grad-map-2}
    \frac{\nabla h(x)-\nabla h(x_\lambda^+)}{\lambda}=\nabla f(x) = \nabla\Psi(x).
\end{equation}
Therefore, we introduce a new  gradient mapping via the dual residuals of a BPG step: 
\begin{equation}
    \label{defn:grd-map-new}
    \cDphl(x):=\frac{\nabla h(x)-\nabla h \big(\Tphl\big(x,\nabla f(x)\big)\big)}{\lambda}.
\end{equation}
To differentiate the new gradient mapping from the primal gradient mapping defined by \eqref{defn:GradMap-BPG}, we call it dual gradient mapping. Next, we explain the reason why we call it ``dual''.

Note that the BPG methods are actually mirror descent algorithms with specifically designed kernels, an alternative interpretation of our new  gradient mapping $\cDphl(\cdot)$ can be obtained from the dual space explanation of the mirror descent method, which was originally presented by Nemirovski and Yudin \cite{nemirovskij1983problem}. Suppose $\phi = 0$ and $\nabla \Psi = \nabla f$.  According to their observation, the gradient $\nabla f(x_k)$ is actually  a linear functional on $\RR^d$ and hence is naturally a covector in the dual space of $\RR^d$. When using an 
\begin{wrapfigure}{r}{0.40\textwidth}  
    \centering
    \includegraphics[width=\linewidth]{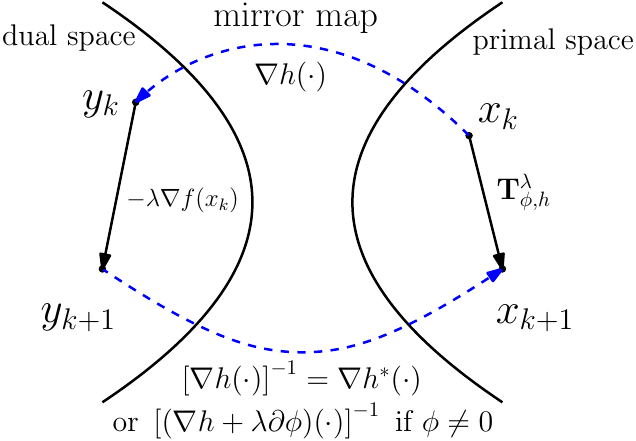}  
    \caption{Dual space interpretation}
    \label{fig:example}
\end{wrapfigure}
$\ell_2$-norm proximal term $D_h(x,x_k)$ with $h(x) = \frac{1}{2}\|x\|^2$,   we are naturally working on   $\mathbb{R}^d$
endowed with $\ell_2$-norm, which is \emph{self-dual}. Therefore, the resulting gradient descent update $x_{k+1} = x_k-\lambda\nabla f(x_k)$ can combine a vector $x_k$ in primal space with the covector $\nabla f(x_k)$ in the dual space. However, when working on non-$\ell_2$-normed spaces that are not self-dual, directly combining them could be problematic. Hence, Nemirovski and Yudin  proposed to map $x_k$ to a point $y_k=\nabla h(x_k)$ in the dual space via the mirror map $\nabla h(\cdot)$, then perform a gradient step in the dual space to obtain $y_{k+1}=y_k - \lambda\nabla f(x_k)$ and map it back to the primal space by inverting the mirror map: $x_{k+1} = \nabla h^{-1}(y_{k+1})$. According to \eqref{eqn:grad-map-2}, such a scheme is equivalent to the BPG iteration \eqref{defn:BPG-standard} with $\phi=0$, if the mirror map is chosen as the gradient of some kernel $h$, and then the inverse mirror map 

\noindent$\nabla h^{-1}(\cdot) = \nabla h^*(\cdot)$ equals the gradient of the convex conjugate of $h$. From this observation, instead of the primal gradient mapping defined on the primal iterates:
$$\cGphl(x_k) = \frac{x_k-x_{k+1}}{\lambda} = \frac{x_k-\nabla h^{*}(\nabla h(x_k)-\lambda\nabla f(x_k))}{\lambda}\neq\nabla f(x_k),$$
which suffers several nonlinear scaling issues incurred by $\nabla h$ and $\nabla h^*$, it is more natural to consider our new gradient mapping from the dual space: 
$$\cDphl(x_k) =\frac{\nabla h(x_k)-\nabla h(x_{k+1})}{\lambda}= \frac{y_k-y_{k+1}}{\lambda} = \nabla f(x_k),$$
which is \emph{invariant} w.r.t. the choice of kernel $h$ and the problem instance.

According to the above discussion, when $\phi = 0$, our new gradient mapping $\cDphl(\cdot) = \nabla\Psi(\cdot)$ exactly recovers the gradient of the objective function regardless of the kernel $h$. When $\nabla f$ is globally Lipschitz continuous s.t. a quadratic kernel $h(x) = \frac{1}{2}\|x\|^2$ is taken, then  $\nabla h(\cdot)$ reduces to the identity mapping and $\cDphl(\cdot) = G_\phi^\lambda(\cdot)$ also exactly recovers the standard proximal gradient mapping  defined in \eqref{defn:GradMap-standard}. Besides these special cases, we provide an exact characterization of the mismatch between the new dual gradient mapping and the Fr\'echet measure based on the $\kappa_h(\cdot)$ quantity defined in \eqref{defn:condition}.  

\begin{lemma}
    \label{lemma:GradMap-nonsmooth}
    Suppose $f$ and $h$ satisfy Assumption \ref{assumption:L-smad}, then for  any $x\in\RR^d$ and $\lambda>0$, we have 
    $${\mathrm{dist}^2(0, \partial \Psi(x_\lambda^+))}/{\|\cDphl(x)\|^2} \leq (1+L\lambda \kappa_h([x,x^+_\lambda]))^2,$$
    where $x_\lambda^+=\Tphl(x,\nabla f(x))$. In particular, when $\phi=0$, we have  
    $\|\nabla\Psi(x)\|^2/\|\cDphl(x)\|^2=1.$
\end{lemma}
\begin{proof}
    By \eqref{eqn:grad-map-1}, we have $\partial\Psi(x_\lambda^+) \ni  \cDphl(x) - \nabla f(x) + \nabla f(x_\lambda^+)$, which further indicates that 
    \begin{eqnarray}
         \mathrm{dist}(0,\partial\Psi(x_\lambda^+)) &\leq& 1 + \|\nabla f(x) - \nabla f(x_\lambda^+)\| \\
        & \leq & \|\cDphl(x)\| + L\cdot L_h([x,x_\lambda^+])\cdot\|x-x_\lambda^+\| \nonumber\\
        & \leq & \|\cDphl(x)\| + \frac{L\cdot L_h([x,x_\lambda^+])}{\mu_h([x,x_\lambda^+])}\cdot\|\nabla h(x) - \nabla h(x_\lambda^+)\| \nonumber\\
        & = & (1+L\lambda \kappa_h([x,x^+_\lambda]))\|\cDphl(x)\|,\nonumber
    \end{eqnarray} 
    where the second inequality is because Assumption \ref{assumption:L-smad}, which indicates $L_f([x,x_\lambda^+])\leq L\cdot L_h([x,x_\lambda^+])$, see Lemma \ref{lemma:B.1}. Then rearranging the terms and squaring both sides proves this lemma.
\end{proof}

\subsection{The kernel conditioning regularity}
\label{subsec:kernel-conditioning}
Compared to the existing stationarity measures discussed in Corollary \ref{corollary:mismatch}, the mismatch between dual gradient mapping and the standard Fr\'echet measure in Lemma \ref{lemma:GradMap-nonsmooth} no longer relies on the potentially unbounded $L_h(\cdot)$ factor. Instead, it relies on $\kappa_h(\cdot)$, which is the ratio between $L_h(\cdot)$ and $\mu_h(\cdot)$. If such a ratio can be globally upper bounded under mild conditions, then our selection of the dual gradient mapping is well justified even if $L_h(\cdot)\to+\infty$. Next, we formalize this discussion as a new kernel-conditioning (KC) regularity assumption, which has been overlooked by the existing BPG literature.

\begin{assumption}[Kernel-conditioning]
    \label{assumption:kernel-conditioning}
    For any $\delta>0$, there exists a constant $\kappa_h^\delta>0$ such that  $$\sup_{\mathcal{X}\subseteq\mathbb{R}^d}\Big\{\kappa_h(\mathcal{X}):\mathrm{diam}(\mathcal{X})\leq \delta\Big\}\leq \kappa_h^\delta,$$
where $\mathrm{diam}(\cX):=\sup\{\|x-y\|: x,y\in\cX\}$ denotes the diameter of the set $\cX$.
\end{assumption} 

Basically, Assumption \ref{assumption:kernel-conditioning} states that when the compact set $\cX$ is not very large, the localized condition number $\kappa_h(\cX)$ over $\cX$ will be uniformly bounded, even when both $\lambda_{\min}(\nabla^2h(x))$ and $\lambda_{\max}(\nabla^2h(x))$ go to $+\infty$. To the best of our knowledge, such a kernel regularity condition has not been considered in the existing works. In later discussion, if a kernel $h$ satisfies Assumption \ref{assumption:kernel-conditioning}, then we will say $h$ is KC-regular. And we will abbreviate kernel conditioning regularity as KC-regularity. As a result, Lemma \ref{lemma:GradMap-nonsmooth} immediately indicates that when $\|x-x_\lambda^+\|\leq \delta$ for some predetermined $\delta>0$, the mismatch ${\mathrm{dist}^2(0, \partial \Psi(x_\lambda^+))}/{\|\cDphl(x)\|^2} \leq (1+L\lambda \kappa_h^\delta)^2$ is at most a constant.  In addition to KC-regularity,  we also  inherits the commonly adopted global $\mu$-strong convexity regularity assumption in this paper, from the existing works \cite[etc.]{bolte2018first,davis2018stochastic,ding2023nonconvex,gao2020randomized,gao2021convergence,hanzely2021fastest,leinertial,li2019provable,mukkamala2020convex,teboulle2018simplified,zhang2018convergence}.

\begin{assumption}[Strong convexity]
    \label{assumption:SC}
    The kernel $h$ is $\mu$-strongly convex for some $\mu>0$.
\end{assumption}

In fact, the KC-regularity is a very robust property that remains stable under various common operations, we summarize this as the following closedness result.

\begin{theorem}[Closedness property]
    \label{theorem:KC-close}
    The KC-regularity is closed under scaling, positive linear combination, and  non-degenerate affine composition, in the sense that:\vspace{0.1cm} \\
    \textbf{\emph{(i).}} If a kernel $h$ is KC-regular, and the matrix $A$ has full column rank, then for any vector $b$, the kernel $h(A\cdot + b)$ is still KC-regular with constant 
    $$\kappa_{h(A\cdot+b)}^\delta \leq \kappa_A^2\cdot\kappa_h^{\|A\|\delta}$$
    where $\kappa_A$ denotes the condition number of the matrix $A$.\vspace{0.1cm}\\
    \textbf{\emph{(ii).}} If a kernel $h$ is KC-regular, then $\alpha h$ is also KC-regular with constant $\kappa_{\alpha h}^\delta = \kappa_h^\delta$ for any $\alpha>0$.\vspace{0.1cm}\\
    \textbf{\emph{(iii).}} If kernel $h$ and $g$ are both KC-regular, then their positive linear combination $\alpha h+\beta g$ is still KC-regular with constant 
    $\alpha h$ is also KC-regular with constant  
    $$\kappa_{\alpha h+\beta g}^\delta \leq \max\{\kappa_h^\delta,\kappa_g^\delta\}$$
    for any positive constants $\alpha,\beta>0$. 
\end{theorem}

\begin{proof}
    To prove (i), let us denote the new kernel as $\omega(x):= h(Ax+b)$. Then direct computation gives $\nabla^2\omega(x) = A^\top\nabla^2h(Ax+b)A$. Denote $\sigma_{\min}(A)$ and $\sigma_{\max}(A)$ the minimum and maximum singular value of $A$, respectively. As $A$ has full column rank, we know $\sigma_{\min}(A)>0$ and the matrix condition number $\kappa_A:=\sigma_{\max}(A)/\sigma_{\min}(A)<+\infty$ if finite. Then it is known that   
    $$\begin{cases}
        \lambda_{\max}\big(\nabla^2\omega(x)\big)\leq \sigma^2_{\max}(A)\cdot\lambda_{\max}\big(\nabla^2h(Ax+b)\big),\\
        \lambda_{\min}\,\big(\nabla^2\omega(x)\big)\geq \sigma^2_{\min}\,(A)\cdot\lambda_{\min}\,\big(\nabla^2h(Ax+b)\big).
    \end{cases}$$  
    Then for  $\forall\delta>0$ and any set $\cX$ with diameter $\mathrm{diam}(\cX)\leq\delta$, we define $\mathcal{Y}:=\{Ax+b:x\in\cX\}$, we have   
    \begin{eqnarray*}
        \kappa_\omega(\mathcal{X}) & = & \frac{\sup\{\lambda_{\max}(\nabla^2\omega(x)):x\in\cX\}}{\inf\{\lambda_{\min}(\nabla^2\omega(x')):x'\in\cX\}}\\
        & \leq & \frac{\sigma^2_{\max}(A)}{\sigma^2_{\min}(A)}\cdot\frac{\sup\{\lambda_{\max}(\nabla^2h(x)):y\in\mathcal{Y}\}}{\inf\{\lambda_{\min}(\nabla^2h(y')):y'\in\mathcal{Y}\}}\\
        & = & \kappa_A^2\cdot \kappa_h(\mathcal{Y}).
    \end{eqnarray*}  
    Note that $\mathrm{diam}(\mathcal{Y})\leq\|A\|\delta$ and $h$ is KC-regular, we have $\kappa_h(\mathcal{Y})\leq \kappa_h^{\|A\|\delta}$. Substituting this bound to the above inequality and taking supremum over all $\cX$ with $\mathrm{diam}(\cX)\leq \delta$ proves (i).

    The proof of (ii) is straightforward. For kernel $\omega(x) := \alpha h(x)$, we have $\nabla^2\omega(x) = \alpha\cdot\nabla^2h(x)$. Such a constant scaling of $\alpha>0$ is automatically canceled out when taking division and hence (ii)  holds.

    To prove (iii), it is sufficient to consider   $\alpha = \beta = 1$. For any set $\cX$ s.t. $\mathrm{diam}(\cX)\leq \delta$, we have 
    \begin{eqnarray*}
        \kappa_{h+g}(\cX) &:=&  \frac{\sup\{\lambda_{\max}(\nabla^2h(x)+\nabla^2g(x)):x\in\cX\}}{\inf\{\lambda_{\min}(\nabla^2h(x)+\nabla^2g(x)):x'\in\cX\}}\\
        & \leq & \frac{\sup\{\lambda_{\max}(\nabla^2h(x))):x\in\cX\}+\sup\{\lambda_{\max}(\nabla^2g(x)):x\in\cX\}}{\inf\{\lambda_{\min}(\nabla^2h(x)):x'\in\cX\}+\inf\{\lambda_{\min}(\nabla^2g(x)):x'\in\cX\}}\\
        & = & \frac{L_h(\cX)+L_g(\cX)}{\mu_h(\cX)+\mu_g(\cX)}. 
    \end{eqnarray*}
    In the above inequalities,  the second line is due to the fact that  for any positive definite matrices $A,B\succ0$, it holds that $\lambda_{\max}(A+B)\leq\lambda_{\max}(A)+\lambda_{\max}(B)$ and $\lambda_{\min}(A+B)\geq \lambda_{\min}(A)+\lambda_{\min}(B).$ 
    Then note that for any $a,b,c,d>0$, let us assume w.l.o.g. that $\frac{a}{b}\geq\frac{c}{d}$. Then direct computation gives
    $\frac{a}{b}-\frac{a+c}{b+d} = \frac{d}{b+d} \left(\frac{a}{b}-\frac{c}{d}\right)\geq0.$ That is, $\frac{a+c}{b+d}\leq \max\big\{\frac{a}{b},\frac{c}{d} \big\}$ always hold. Consequently
    $$\kappa_{h+g}(\cX) \leq \frac{L_h(\cX)+L_g(\cX)}{\mu_h(\cX)+\mu_g(\cX)} \leq \max\left\{\frac{L_h(\cX)}{\mu_h(\cX)},\frac{L_g(\cX)}{\mu_g(\cX)}\right\} \leq \max\Big\{\kappa_{h}^\delta,\kappa_{g}^\delta\Big\}.$$ 
    Then taking supremum over all $\cX$ with $\mathrm{diam}(\cX)\leq \delta$ proves (iii) when $\alpha=\beta=1$. For general $\alpha,\beta>0$, it is sufficient to combine this result with (ii) and obtain 
    $$\kappa_{\alpha h+\beta g}^\delta\leq\max\Big\{\kappa_{\alpha h}^\delta,\kappa_{\beta g}^\delta\Big\} = \max\Big\{\kappa_{h}^\delta,\kappa_{g}^\delta\Big\}.$$
    Hence we complete the proof of Theorem \ref{theorem:KC-close}.
\end{proof}

Theorem \ref{theorem:KC-close} indicates that starting from simple KC-regular kernels, one can construct appropriate new KC-regular kernels or verify KC-regularity for related kernels, by rotation, distortion, translation, or taking combinations. As a detailed example, we show that the power kernels satisfy this condition. 

\begin{proposition}
    \label{proposition: Poly-kernel-condition}
    Let $h(x)=\frac{\alpha}{2}\|x\|^2+\frac{1}{r+2}\|x\|^{r+2}$ be a power kernel for some real number $r\geq0$ and $\alpha>0$. Then this kernel satisfies the following properties: \vspace{0.1cm}\\
    \noindent\textbf{\emph{(i)}.} For any set $\mathcal{X}\subseteq\RR^d$, the local condition number satisfies $\kappa_h(\cX)\leq3r+4$ as long as 
    $$\mathrm{diam}(\cX)\leq \frac{1}{r}\cdot\max\Big\{\alpha^{1/r}\,,\,\min_{u\in\cX}\|u\|\Big\}.$$ 
    \textbf{\emph{(ii)}.} For any positive $\delta>0$, the kernel $h$ satisfies
    $$\kappa_h^\delta\leq \begin{cases}
        (r+1)\max\big\{1,\frac{\delta^r}{\alpha}\big\}+1,&\mbox{ if } r\leq 1\\
        (r+1)\big(1+\big(\frac{\delta^r}{\alpha}\big)^{\frac{1}{r-1}}\big)^{r-1} + 1, &\mbox{ if } r>1
    \end{cases}$$
    In particular, we have $\kappa_h^\delta\leq r+2$ for any $\delta\leq \alpha^\frac{1}{r}$ when $r\in(0,1]$, and we have $\kappa_h^\delta \leq 3r+4$ for all $\delta\leq \alpha^\frac{1}{r}/r$ when $r>1$. They provide uniform upper bounds for $\kappa_h(\cX)$ whenever $\mathrm{diam}(\cX)\leq\delta$. 
\end{proposition}
The verification of this proposition only consists elementary computation, and is moved to Appendix \ref{appdx:proposition: Poly-kernel-condition} for succinctness. For general composition of norm kernel $h(x): = H(\|x\|)$, see \cite{bauschke2017descent}, KC-regularity still applies if $H(\cdot)$ has desirable properties. Due to the closedness of KC-regularity, the kernel $h(x) = {\|x\|^2_A}/{2} + {\|x\|^{\alpha}_B}/{\alpha}$ is KC-regular, where $A,B$ are positive definite matrices and $\alpha>2$. Such a kernel has been applied to solving the subproblems of higher-order methods with H\"{o}lder continuity \cite{grapiglia2020tensor}.

We also note that many popular kernels possess a block-separable structure. That is,  variable $x$ can be partitioned to multiple blocks $x_1,\cdots,x_m$ s.t. $h(x) = \sum_{i=1}^mh_i(x_i)$. For example, \cite{ding2023nonconvex} considered a multi-block  polynomial kernel for neural networks, where each block corresponds to the network parameters in one layer. In most cases, each block $x_i$ is a single variable and the kernel $h$ is element-wisely separable, which is satisfied by most examples in \cite{bauschke2017descent}, where $h_i$ can be exponential, various types of entropy, Hellinger, as well as their regularized variants. For such block-separable kernels, it is natural to consider a block-separable variant of KC-regularity, we shall discuss this extension in Appendix \ref{section:multi-blocks}. At this moment, we focus on the basic single-block discussion to avoid the notational sophistication.

\subsection{A Lipschitz-like bound for gradient difference}
As a final preparation in this section, we would like to derive a Lipschitz-like bound for gradient differences guaranteed under KC-regularity. As discussed in the introduction, popular acceleration techniques for stochastic approximation methods like momentum, shuffling, and variance reduction, 
all rely on bounding stochastic errors by gradient differences, and then controls gradient differences by iterative descents through gradient Lipschitz property. 

In particular, for variance reduction, the fundamental logic is based on a simple insight that \emph{the gradient difference is easier to estimate than the gradient itself} for Lipschitz-smooth function. Roughly speaking, suppose $f(x) = \mathbb{E}_\xi[f_\xi(x)]$, where $f_\xi$ is $L$-smooth. Then the mean squared error (MSE) for a stochastic gradient estimator will be 
$\mathbb{E}[\|\nabla f_\xi(x) - \nabla f(x)\|^2]$, which is often upper bounded by some positive \emph{constant} through a bounded variance assumption. On the other hand, given a reference point $x_{\text{ref}}$, the MSE for estimating $\Delta:= \nabla f(x)-\nabla f(x_{\text{ref}})$ by $\Delta_{\xi} := \nabla f_\xi(x)-\nabla f_\xi(x_{\text{ref}})$ satisfies
$$\EE\left[\|\Delta-\Delta_{\xi}\|^2\right] \leq \EE\left[\|\Delta_{\xi}\|^2\right] = \EE\left[\|\nabla f_\xi(x)-\nabla f_\xi(x_{\text{ref}})\|^2\right] \leq L^2\|x-x_{\text{ref}}\|^2.$$
Even with a single sample $\xi$, the above MSE will automatically go to $0$ if $\|x-x_{\text{ref}}\|\to0$.  
For the variance reduced methods whose design ensures this to happen, 
given an accurate enough estimator $\tilde{\nabla}f(x_{\text{ref}})\approx \nabla f(x_{\text{ref}})$,  then $\tilde{\nabla}f(x_{\text{ref}}) +\Delta_{\xi}$ will give a much more accurate estimation of $\nabla f(x)$. This forms the basic insight why variance reduction accelerates SGD under classic $L$-smoothness condition.

Although for stochastic BPG method, the smooth-adaptability alone is not enough to ensure a global Lipschitz-like bound for gradient differences, fortunately,  KC-regularity  provides a remedy.

\begin{proposition}[Lipschitz-like bound] 
    \label{proposition:Lipschitz} 
    Suppose $f$ is $L$-smooth adaptable to some KC-regular kernel $h$. Let $\delta>0$ and let $\cX$ be any convex set with $\mathrm{diam}(\cX)\leq \delta$. Then for any $x,y\in\cX$, and an arbitrary interpolation point $z_\theta := \theta x + (1-\theta)y$ with $ \theta\in[0,1]$, we have  
    \begin{equation}
        \label{prop:Lipschitz}
        \frac{\|\nabla f(z_\theta)-\nabla f(y)\|^2}{2L^2\mu_h(\cX)}\leq  \theta^2\kappa_h^2(\cX) D_h(x,y),
    \end{equation}
    where by KC-regularity, we know the constant upper bound $\kappa_h(\cX)\leq \kappa_h^\delta$ always hold. 
\end{proposition}
The proof of this proposition is very simple, and is moved to Appendix \ref{appdx:proposition:Lipschitz}. We introduce the interpolation point $z_\theta$ in the proposition mainly for the ease of reference in the latter analysis as we study a variance reduction scheme with interpolation. However, when we set $\theta = 1$ so that $z_\theta = x$, it actually implies a more  interesting insight. Note that $D_h(x,y)\leq{\|\nabla h(x)-\nabla h(y)\|^2}/{2\mu_h(\cX)}$ always holds, substituting it to \eqref{prop:Lipschitz} yields 
\begin{equation}
    \label{key:dual-Lipschitz}
    \|\nabla f(x)-\nabla f(y)\| \leq L\kappa_h^\delta \cdot \|\nabla h(x)-\nabla h(y)\|\qquad\mbox{for}\qquad \forall x,y\in\cX.
\end{equation}
Because the mirror map $\nabla h$ is strictly monotone, it is straightforward to verify that the function defined by $\rho(x,y):=\|\nabla h(x)-\nabla h(y)\|$ satisfies
$$\mbox{(positive definiteness)}\qquad \rho(x,y)\geq 0 \mbox{ and } \rho(x,y)=0 \Longleftrightarrow x=y,\qquad\qquad\qquad\qquad$$
$$\mbox{(symmetry)}\qquad \rho(x,y) = \rho(y,x),\qquad\qquad\qquad\qquad\qquad\qquad\quad$$
$$\mbox{(triangle inequality)}\qquad \rho(x,y) \leq  \rho(x,z) + \rho(z,y),\qquad\qquad\qquad\qquad\qquad\!\!\quad\qquad$$
and is hence a distance metric. Therefore, \eqref{key:dual-Lipschitz} states that although $\nabla f$ is not Lipschitz continuous under the standard $\ell_2$ distance. It is locally but uniformly $L\kappa_h^\delta$-Lipschitz under the distance induced by the mirror map $\nabla h$ in the dual space. By local we mean the Lipschitz bound holds only locally for points inside a  $\delta$-bounded area. However, this bound is also uniform, or global, in the sense that the local Lipschitz constant ($L\kappa_h^\delta$) remains the same throughout the whole space under KC-regularity. This provides a brand-new geometric insight on what $L$-smooth adaptable to a kernel means. Nevertheless, as \eqref{prop:Lipschitz} directly relates gradient differences with Bregman divergence, which further relates to the iterative descents of the algorithms, we will mostly use the bound in Proposition \ref{proposition:Lipschitz}.

Finally, we would like to discuss \cite[Assumption 3]{ding2023nonconvex}, where the authors directly assume the existence of some constant $c$ s.t. $\|\nabla f(x)-\nabla f(y)\|^2\leq c\cdot D_h(y,x)$, which is, to some degree, similar to \eqref{prop:Lipschitz}. Note that  
$\|\nabla f(x)-\nabla f(y)\|^2\leq L_f^2(\cX)\|x-y\|^2$, while $D_h(y,x)\geq\frac{\mu_h(\cX)}{2}\|x-y\|^2$. One may require $\mu_h(\cX)\geq \mathrm{const}\cdot L_f^2(\cX)$ to guarantee the validity of \cite[Assumption 3]{ding2023nonconvex}. This is much stronger than the usual smooth-adaptability assumption the only implies $L_h(\cX)\geq\mathrm{const}\cdot L_f(\cX)$. In case $f$ is smooth adaptable to an $r$-degree polynomial kernel,  then this assumption may require one use a $2r$-degree polynomial kernel.   Indeed, \cite{ding2023nonconvex} justifies this assumption by considering $f(x) = \frac{x^4}{4}$ and $h(x) = \frac{x^2}{2}+\frac{x^8}{8}$, while $f$ is already 1-smooth adaptable to $h(x) = \frac{x^2}{2}+\frac{x^4}{4}.$ However, according to our observation in Example \ref{example:counterexample-1}, unnecessarily doubling the degree can significantly slow down the actual convergence rate of the algorithm, which can also be clearly explained by our instance-free worst-case complexity in latter sections.

\section{Improving the instance-dependent complexity}\label{section:VR}

In this section, we illustrate how the KC-regularity and the resulting Lipschitz-like bound improves the instance-dependent complexity of stochastic BPG method from $O(\epsilon^{-2})$ to $O(\sqrt{n}\epsilon^{-1})$. In particular, combined with a novel probabilistic argument, we provide a simple epoch bound mechanism that can facilitate most episodic stochastic variance reduction techniques such as SVRG \cite{johnson2013accelerating}, SPIDER \cite{fang2018spider}, SARAH and ProxSARAH  \cite{pham2020proxsarah}, etc. To avoid repetition, we only discuss the sample complexity for reducing squared primal gradient mapping, while the proposed technique can be easily extended to bounding Bregman residual and its symmetrized variant discussed in Corollary \ref{corollary:mismatch}. 

\subsection{The general algorithm and analysis framework}
Consider problem \eqref{prob:main-det} with $f(x) = \frac{1}{n}\sum_{i=1}^{n}f_i(x).$ We adopt the following variant of Assumption \ref{assumption:L-smad}. 
\begin{assumption}
    \label{assumption:L-smad-finite-sum}
    For each $i\in[n]$, $f_i$ is $L_i$-smooth adaptable to $h$ for some positive constant $L_{i}>0$. Denote $L:=\sqrt{\frac{1}{n}\sum_i^nL_i^2}$, then $f$ is $L$-smooth adaptable to $h$. 
\end{assumption}
Throughout Section \ref{section:VR}, we will use Assumption \ref{assumption:L-smad-finite-sum}, Assumption \ref{assumption:kernel-conditioning}, and Assumption \ref{assumption:SC}, and we propose a stochastic variance reduced BPG method with epoch-wise bounds  in Algorithm \ref{algorithm:proxSARAH}. 
 
\begin{algorithm2e}
\caption{Stochastic variance reduced BPG method with epoch bounds} 
\label{algorithm:proxSARAH}
\textbf{Input:} Initial point $x_{1,0}$, constant $\delta$, epoch length $\tau$, step size $\eta$,  interpolation factor $\gamma\in(0,1]$.   \\
\For{$s=1,2,3,\cdots,S$}{
    Construct a convex set $\mathcal{X}_s\supseteq {B}(x_{s,0},\delta/2)$ such that $\kappa_h(\mathcal{X}_s)\leq \kappa_h^\delta$\,.  \\\textcolor{gray}{//**Proposition \ref{proposition: Poly-kernel-condition} suggests $\mathcal{X}_s= B\left(x_{s,0},\max\left\{\frac{1}{2r},\frac{\|x_{s,0}\|}{2r+1}\right\}\right)$ for polynomial kernel**//}
    
    \For{$k = 0,1,2,\cdots,\tau-1$}{  
    If $k==0$, compute $v_{s,0}= \nabla f(x_{s,0}) = \frac{1}{n}\sum_{i=1}^n\nabla f_i(x_{s,0})$.\\
    If $k\geq 1$, uniformly sample a mini-batch $\mathcal{B}_{s,k}\subseteq[n]$  with replacement, compute  
    \begin{equation}
        \label{alg:sarah-small}
        v_{s,k} = v_{s,k-1} + \frac{1}{|\mathcal{B}_{s,k}|}\sum_{\xi\in\mathcal{B}_{s,k}} \Big(\nabla f_\xi(x_{s,k})-\nabla f_\xi(x_{s,k-1})\Big)\,.
    \end{equation}\\
    
    Denote $\mathrm{id}_{\cX_s}$ the indicator function of $\cX_s$. Compute the BPG update with

    \begin{equation}
        \label{alg:proxSARAH-update}
        \bar{x}_{s,k+1} =  \Tshe(x_{s,k},v_{s,k})\qquad\mbox{and}\qquad x_{s,k+1} = (1-\gamma)x_{s,k}+\gamma\bar{x}_{s,k+1}\vspace{-0.1cm}
    \end{equation}  \\ 
    \textbf{if}\,\,\,$\mathrm{dist}(x_{s,k+1},\partial \cX_s) \leq \delta/4$\,\,\,\textbf{then}\,\,\,\textbf{break} the inner forloop.
    } 
    Set $\tau_s=k+1$ and $x_{s+1,0} = x_{s,\tau_s}$. 
    } 
\end{algorithm2e}
 
In each epoch of this algorithm, based on a predetermined radius $\delta$ defined in KC-regularity (Assumption \ref{assumption:kernel-conditioning}), we impose an additional convex set constraint $x\in\cX_s$ in which the kernel $h$ has limited condition number.  With this simple mechanism, one can input any episodic variance reduced gradient estimator in place of \eqref{alg:sarah-small}. In this paper, we use the SARAH/SPIDER estimator. In particular, the update of $\bar{x}_{s,k+1}$ in \eqref{alg:proxSARAH-update} incorporates an indicator function $\mathrm{id}_{\cX_s}$, that is, 
\begin{equation} 
        \label{alg:sarah-update}
        \bar{x}_{s,k+1} = \argmin_{x\in\mathcal{X}_s}\,\, \langle v_{s,k}, x\rangle + \phi(x) +  \frac{1}{\eta}D_h(x,x_{s,k}) . 
    \end{equation} 
The purpose for adopting such an indicator function is to guarantee that the whole epoch $\{x_{s,k}\}_{k=0}^{\tau_s}$ stays inside $\cX_s$ so that KC-regularity can help us to bound the gradient estimation errors. However, we should also notice that if ${x}_{s,k}$ is too close to the boundary $\partial \cX_s$ and $\bar{x}_{s,k+1}$ hits $\partial \cX_s$, then the next point $x_{s,k+1}$ will have to take a very conservative step, which could have been a more aggressive step. Then Line 9 provides an early stop scheme for each epoch to prevent such cases. If one removes Line 9 and let every epoch to run full $\tau$ iterations, the algorithm   still works. But if some ${x}_{s,k}$ is close to $\partial \cX_s$ and is forced to take a conservative step in early stage of an epoch, it is quite possible that the future iterations in epoch $s$ will also suffer the same issue, causing a waste of computation. 

Define the true \emph{restricted} primal gradient mapping in epoch $s$ as  
\begin{equation}
    \label{defn:restricted-grd-mapping}
    \cGshe(x_{s,k}):=\frac{x_{s,k}-\hat{x}_{s,k+1}}{\eta}\quad\mbox{with}\quad\hat{x}_{s,k+1}:=\Tshe\big(x_{s,k},\nabla f(x_{s,k})\big),
\end{equation} 
where by ``restricted'' we means this primal gradient mapping incorporates the indicator function of the set constraint $x\in\cX_s$.
Compared to the $\bar{x}_{s,k+1}$ in Algorithm \ref{algorithm:proxSARAH}, $\hat{x}_{s,k+1}$ is constructed with the true gradient $\nabla f(x_{s,k})$.  For this restricted primal gradient mapping, the following lemma holds true. 
\begin{lemma}
    \label{lemma:restricted-grd-mapping} 
    Let $\cGshe(x_{s,k})$, $\hat{x}_{s,k+1}$ and $\bar{x}_{s,k+1}$ be defined by \eqref{defn:restricted-grd-mapping} and \eqref{alg:sarah-update}, respectively. Then 
    \begin{eqnarray}
    \label{lm:grad-mapping} 
        \big\|\cGshe(x_{s,k})\big\|^2 \leq \frac{2\|x_{s,k}-\bar{x}_{s,k+1}\|^2}{\eta^2} + \frac{2\|\mathcal{E}_{s,k}\|^2}{\mu_h^2(\mathcal{X}_s)},
    \end{eqnarray} 
    where $\mathcal{E}_{s,k} = \nabla f(x_{s,k})-v_{s,k}$ denotes the gradient estimation error at $x_{s,k}$.
\end{lemma}

\noindent As the proof of this lemma is very standard, it is relegated to Appendix \ref{appendix:sec:VR}. However, we should also bear in mind  that $\cGshe(\cdot)\neq\cGphe(\cdot)$ due to constraint $x\in\cX_s$ in each epoch. It is not the primal gradient mapping that we aim to bound eventually. Given this lemma, we can obtain the following descent result. Different from the standard descent result for stochastic BPG methods such as \cite{ding2023nonconvex}, we need to keep the descent both in terms of the true restricted primal gradient mapping $\|\cGshe(x_{s,k})\|^2$ and the Bregman divergence term $D_h(x_{s,k+1},x_{s,k})$.

\begin{lemma}
\label{lemma:descent-proxSARAH-finite}
    Let $\{x_{s,k}\}_{k=0}^{\tau_s}$ be the $s$-th epoch of Algorithm \ref{algorithm:proxSARAH}, then  we have 
    \begin{eqnarray}
    \Psi(x_{s,k+1})\leq \Psi(x_{s,k})-\frac{\gamma\eta\mu_h(\mathcal{X}_s)}{8}\big\|\cGshe(x_{s,k})\big\|^2\!-\!\bigg(\frac{\gamma}{\eta}- L\kappa_h^\delta\gamma^2\!\bigg)\!D_h(\bar{x}_{s,k+1},x_{s,k}) \!+\!\frac{5\gamma\eta\|\mathcal{E}_{s,k}\|^2}{4\mu_h(\mathcal{X}_s)}.\,\nonumber
    \end{eqnarray} 
\end{lemma}

\noindent The proof of Lemma \ref{lemma:descent-proxSARAH-finite} is moved to Appendix \ref{appendix:sec:VR}. Next, we bound the error term $\|\mathcal{E}_{s,k}\|^2$, whose proof is kept in the main paper to illustrate how KC-regularity affects the variance bounds.

\begin{lemma}
    \label{lemma:err-sarah} 
    Let $\{x_{s,k}\}_{k=0}^{\tau_s}$ be the $s$-th epoch of Algorithm \ref{algorithm:proxSARAH}. Given any batch size $b>0$, if we set $|\mathcal{B}_{s,k}|=b$ for $k=1,\cdots,\tau_s-1$. Then conditioning on the initial point $x_{s,0}$ of the epoch, we have
    \begin{eqnarray}
    \label{lm:err-sarah-0}
    \mathbb{E}\left[\frac{\|\mathcal{E}_{s,k}\|^2}{\mu_h(\mathcal{X}_s)}\,\Big|\, x_{s,0}\right] \leq  \frac{2\gamma^2L^2(\kappa_h^\delta)^2}{b}\mathbb{E}\bigg[\sum_{j=0}^{k-1} D_h(\bar{x}_{s,j+1},x_{s,j})\,\big|\,x_{s,0}\bigg]\,, 
\end{eqnarray}
as long as $h$ satisfies the kernel conditioning regularity assumption in $\cX_s$. 
\end{lemma}
\begin{proof}
    By \cite[Lemma 2]{pham2020proxsarah}, it is not hard to obtain that 
    \begin{equation}
        \label{lm:err-sarah-0.5}
        \mathbb{E}\Big[\|\mathcal{E}_{s,k}\|^2\,\big|\,x_{s,0}\Big]  \leq  \sum_{j=0}^{k-1} \mathbb{E}\bigg[\frac{1}{bn}\sum_{i=1}^n\|\nabla f_i(x_{s,j+1})-\nabla f_i(x_{s,j})\|^2\,\big|\,x_{s,0}\bigg].
    \end{equation}
    Because each $f_i$ is $L_i$-smooth adaptable to the KC-regular kernel $h$, applying the Lipschitz-like bound \eqref{prop:Lipschitz} to each $f_i$ on   $\bar{x}_{s,j+1}, x_{s,j}$ and their interpolation ${x}_{s,j+1} = \gamma\bar{x}_{s,j+1} + (1-\gamma)x_{s,j}$ immediately yields 
    \begin{eqnarray}
    \label{lm:err-sarah-1}
    \mathbb{E}\left[\frac{\|\mathcal{E}_{s,k}\|^2}{\mu_h(\mathcal{X}_s)}\,\Big|\, x_{s,0}\right] 
    &\leq&  \sum_{j=0}^{k-1}\mathbb{E}\bigg[ \frac{\sum_{i=1}^nL_i^2}{n}\cdot\frac{2\gamma^2  \kappa_h^2(\cX_s)}{b}D_h(\bar{x}_{s,j+1},x_{s,j})\,\big|\,x_{s,0}\bigg]\nonumber.
\end{eqnarray}
Using the fact that $L^2 = \frac{1}{n}\sum_{i=1}^nL_i^2$ and the fact that $\kappa_h(\cX_s)\leq \kappa_h^\delta$ leads to \eqref{lm:err-sarah-0}. 
\end{proof}

After properly bounding the error term $\mathcal{E}_{s,k}$, we obtain Lemma \ref{lemma:proxSARAH-grad-map-mid} for the restricted primal gradient mapping, whose proof is placed in Appendix \ref{appendix:sec:VR}. 
\begin{lemma}
    \label{lemma:proxSARAH-grad-map-mid}
    For any $b,\tau\in\mathbb{Z}_+$, set $\eta = \frac{\sqrt{2\tau}}{\sqrt{7\tau}+\sqrt{2b}}$, $\gamma = \frac{\sqrt{b}}{L\kappa_h^\delta\sqrt{\tau}}$, and $|\mathcal{B}_{s,k}|=b$, $\forall s,k\geq1$, then
    $$\mathbb{E}\left[\sum_{s=1}^S\sum_{k=0}^{\tau_s-1}\frac{\mu_h(\mathcal{X}_s)}{8}\big\|\cGshe(x_{s,k})\big\|^2 +D_h(\bar{x}_{s,k+1},x_{s,k}) \right]\leq \frac{\Delta_\Psi}{\gamma\eta},$$
    where $\Delta_\Psi:=\Psi(x_{1,0})-\inf\{\Psi(x):x\in\RR^d\}$ denotes the initial function value gap. 
\end{lemma}

\noindent There are several significant difficulties in the interpretation of Lemma \ref{lemma:proxSARAH-grad-map-mid}.

First, in Lemma \ref{lemma:proxSARAH-grad-map-mid}, the length $\tau_s$ of each epoch are random variables. Therefore, it is incorrect to simply divide $\sum_{s=1}^S\tau_s$ on both sides and argue $\mathbb{E}[\|\cGshe(x_{s,k})\|^2]\leq O(1/\sum_{s=1}^S\tau_s)$ for some randomly selected $x_{s,k}$. Moreover, if $\tau_s$ are too small compared to $\tau$, then one should frequently restart and take full batch to initialize new epochs, which may cause a bad sample complexity. Therefore, careful probabilistic analyses are required to exclude such event.
\begin{figure}[h]
\centering
    \includegraphics[width=0.99\linewidth]{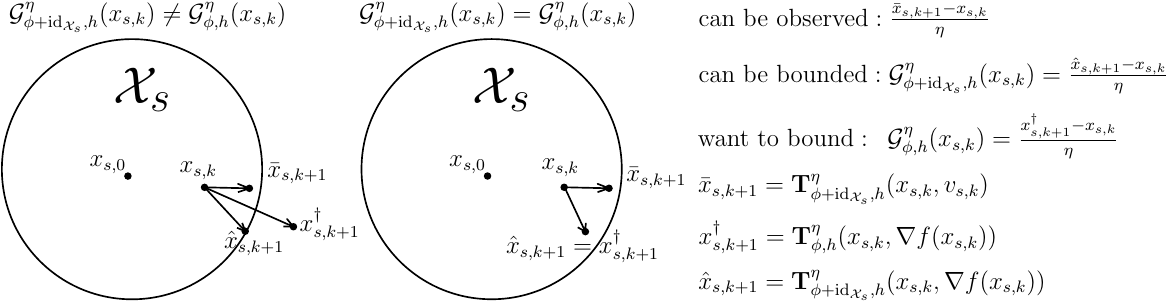}
    \caption{Differences between what we can observe, what we can bound, and what we want to bound.}
  \label{fig:mismatch}
\end{figure}
Second, even if the first issue is resolved, a direct consequence of Lemma \ref{lemma:proxSARAH-grad-map-mid} will be a small restricted primal gradient mapping $\|\cGshe(\cdot)\|^2$.  If the second case of Figure \ref{fig:mismatch} happens, $\hat{x}_{s,k+1}\in\mathrm{int}(\cX_s)$ and the constraint $x\in\cX_s$ is inactive. Then $\hat{x}_{s,k+1}=x_{s,k+1}^\dagger$ and $\|\cGshe(x_{s,k})\|^2 = \|\cGphe(x_{s,k})\|^2$.  However if  $\hat{x}_{s,k+1}\in\partial\cX_s$, one may have $x^\dagger_{s,k+1}\notin\cX_s$ and $\|\cGphe(x_{s,k})\|^2>\|\cGshe(x_{s,k})\|^2$. In this case, having a small restricted primal gradient mapping may not necessarily indicate a small primal gradient mapping. Moreover, observing $\bar{x}_{s,k+1}\in\mathrm{int}(\cX_s)$ also may not necessarily indicate $\hat{x}_{s,k+1}\in\mathrm{int}(\cX_s)$ due to the gradient estimation errors. Therefore, we also need careful probabilistic analyses to show that the bad event $\hat{x}_{s,k+1}\in\partial\cX_s$ may  only happen for limited times with high probability. To resolve the above two issues, let us bound the probability for the following events. 

\begin{lemma}
\label{lemma:event-halway} 
For any $S$ epochs  generated by Algorithm \ref{algorithm:proxSARAH}, define the set $\mathcal{I}_1$ and event $\mathcal{A}_1(m_1)$ as  
$$\mathcal{I}_1:=\big\{s\in[S] :  \tau_s < \tau\big\}\qquad\mbox{and}\qquad\mathcal{A}_1(m_1):=\left\{\omega: |\mathcal{I}_1|\geq m_1\right\}.$$
where $m_1>0$ is an arbitrary positive number. Then it holds that 
$$\mathrm{Prob}\left(\mathcal{A}_1(m_1)\right)\leq \frac{32\gamma\tau\Delta_\Psi}{\eta\mu\delta^2\cdot m_1}\,.$$
\end{lemma}
\begin{proof}
By lemma \ref{lemma:proxSARAH-grad-map-mid}, ignoring the restricted primal gradient mapping terms yields   
\begin{eqnarray}
    \label{lm:event-halfway-1}
    \frac{\Delta_\Psi}{\gamma\eta} &\geq& \mathbb{E}\left[\sum_{s=1}^S\sum_{k=0}^{\tau_s-1}D_h(\bar{x}_{s,k+1},x_{s,k}) \right]\nonumber\\ 
    & \geq & \mathbb{E}\left[\frac{\mu}{2\gamma^2}\sum_{s\in\mathcal{I}_1}\sum_{k=0}^{\tau_s-1}\|{x}_{s,k+1}-x_{s,k}\|^2 \right] \\
    & \geq & \frac{\mu}{2\gamma^2}\cdot \mathrm{Prob}\left(\mathcal{A}_1(m_1)\right) \cdot\mathbb{E}\left[\sum_{s\in\mathcal{I}_1}\sum_{k=0}^{\tau_s-1}\|{x}_{s,k+1}-x_{s,k}\|^2\,\Big|\, \mathcal{A}_1(m_1) \right] \nonumber,
\end{eqnarray}
where the last inequality is due to the fact that   
$$\EE[X] = \mathrm{Prob}(\cA)\cdot\EE\big[X\mid\cA\big] + \mathrm{Prob}(\cA^c)\cdot\EE\big[X\mid\cA^c\big]\geq \mathrm{Prob}(\cA)\cdot\EE\big[X\mid\cA\big]$$
for any non-negative random variable $X\geq0$ and any event $\cA$. Therefore, conditioning on the event $\mathcal{A}_1(m_1)$, for any epoch $s\in\mathcal{I}_1$, we will have $\|x_{s,\tau_s}-x_{s,0}\|\geq\frac{\delta}{4}$ because of Line 3 and Line 9 of Algorithm \ref{algorithm:proxSARAH}. Consequently, for $\forall s\in\mathcal{I}_1$, the triangle inequality and the arithmetic inequality indicate that
\begin{equation}
    \label{lm:event-halfway-2}
    \frac{\delta}{4\tau_s}\leq \frac{\|x_{s,\tau_s}-x_{s,0}\|}{\tau_s}\leq\frac{\sum_{k=0}^{\tau_s-1}\|x_{s,k+1}-x_{s,k}\|}{\tau_s}\leq  \sqrt{\frac{\sum_{k=0}^{\tau_s-1}\|x_{s,k+1}-x_{s,k}\|^2}{\tau_s}}\,.
\end{equation} 
Because $\tau_s\leq\tau$ always holds, we have 
\begin{align}
    \label{lm:event-halfway-3}
    \sum_{k=0}^{\tau_s-1}\|x_{s,k+1}-x_{s,k}\|^2  \geq  \sum_{k=0}^{\tau_s-1}\|x_{s,k+1}-x_{s,k}\|^2\geq\frac{\delta^2}{16\tau_s}\geq\frac{\delta^2}{16\tau}\,.
\end{align}
Note that the above inequalities hold w.p. 1 conditioning on $\cA_1(m_1)$. Combined with \eqref{lm:event-halfway-1}, we have 
$$\frac{\Delta_\Psi}{\gamma\eta} \geq\frac{\mu}{2\gamma^2}\cdot \mathrm{Prob}\left(\mathcal{A}_1(m_1)\right)\cdot\frac{m_1\delta^2}{16\tau}.$$
Rearranging the terms completes the proof. 
\end{proof}
By Lemma \ref{lemma:event-halway}, we show that at least $S-m_1$ epochs never stop early with $1-O(1/m_1)$ probability, which indicates that the constraint set $\cX_s$ remains inactive in these epochs. In the next lemma, we show that with high probability, the majority of iterates will not suffer the issue that $\cGshe(\cdot)\neq\cGphe(\cdot)$.   

\begin{lemma}
    \label{lemma:event-grad-map}
    For any $S$ epochs generated by Algorithm \ref{algorithm:proxSARAH}, define the set $\mathcal{I}_2$ and event $\mathcal{A}_2(m_2)$ as
    $$\mathcal{I}_2:=\left\{(s,k) :  \hat{x}_{s,k+1}\in\partial \mathcal{X}_s, 1\leq s\leq S, 0\leq k\leq \tau_s-1\right\}\quad\mbox{and}\quad\mathcal{A}_2(m_2):=\big\{\omega : |\mathcal{I}_2|\geq m_2\big\}$$
    where  $m_2>0$ is an arbitrary positive number. Then it holds that 
    $$\mathrm{Prob}\left(\mathcal{A}_2(m_2)\right)\leq \frac{128\eta\Delta_\Psi}{\gamma\mu\delta^2\cdot m_2}\,.$$
\end{lemma}
\begin{proof}
Similar to the proof of Lemma \ref{lemma:event-halway}, ignoring the Bregman divergence terms in Lemma \ref{lemma:proxSARAH-grad-map-mid} yields   
\begin{eqnarray}
    \label{lm:event-grad-map-1}
    \frac{\Delta_\Psi}{\gamma\eta} & \geq & \mathbb{E}\left[\sum_{s=1}^S\sum_{k=0}^{\tau_s-1}\frac{\mu}{8}\big\|\cGshe(x_{s,k})\big\|^2\right]\nonumber\\
    & = & \frac{\mu}{8\eta^2}\mathbb{E}\left[\sum_{s=1}^S\sum_{k=0}^{\tau_s-1}\|\hat{x}_{s,k+1}-x_{s,k}\|^2\right]\\
    & \geq & \frac{\mu}{8\eta^2}\mathbb{E}\Bigg[\sum_{(s,k)\in\mathcal{I}_2} \|\hat{x}_{s,k+1}-x_{s,k}\|^2\Bigg]\nonumber\\
    & \geq & \frac{\mu}{8\eta^2}\cdot\mathrm{Prob}\left(\mathcal{A}_2(m_2)\right)\cdot\mathbb{E}\Bigg[\sum_{(s,k)\in\mathcal{I}_2}\|\hat{x}_{s,k+1}-x_{s,k}\|^2\,\Big|\,\mathcal{A}_2(m_2)\Bigg]\nonumber.
\end{eqnarray}
Note that for  $\forall(s,k)\in\mathcal{I}_2$, we have $\mathrm{dist}(x_{s,k},\partial\cX_s)\geq\delta/4$ while $\hat{x}_{s,k+1}\in\partial\cX_s$. Consequently, one must have $\|\hat{x}_{s,k+1}-{x}_{s,k}\|\geq 
\frac{\delta}{4}$. Conditioning on the event $\cA_2(m_2)$, we have $|\mathcal{I}_2|\geq m_2$ and 
$$\mathbb{E}\Bigg[\sum_{(s,k)\in\mathcal{I}_2}\|\hat{x}_{s,k+1}-x_{s,k}\|^2\,\Big|\,\mathcal{A}_2(m_2)\Bigg]\geq \frac{m_2\delta^2}{16}.$$
Substitute this bound to \eqref{lm:event-grad-map-1} proves lemma.  
\end{proof}

\noindent Now we present the final result in the following theorem.  

\begin{theorem}
    \label{theorem:proxSARAH-finite}
    For any constant batch size $|\mathcal{B}_{s,k}|=b\in[n]$, let us set the epoch length as $\tau = \lceil n/b\rceil$, step size $\eta = \frac{\sqrt{2\tau}}{\sqrt{7\tau}+\sqrt{2b}}$, interpolation coefficient $\gamma = \frac{\sqrt{b}}{L\kappa_h^\delta\sqrt{\tau}}$, and total epoch number $S = \big\lceil\frac{16\Delta_\Psi}{\tau\gamma\eta\mu\epsilon}\big\rceil$. Suppose the target accuracy satisfies $\epsilon\leq \frac{\delta^2}{16}\cdot\min\big\{\frac{L^2(\kappa_h^\delta)^2}{b\tau},\frac{1}{9\eta^2}\big\} = O(1/n)$ and let $x_{\mathrm{out}}$ be uniformly randomly selected from all iterations, then there is a high probability event $\mathcal{A}$ such that $$\mathbb{E}\Big[\big\|\cGphe(x_{\mathrm{out}})\big\|^2 \,\big|\, \mathcal{A}\Big]\leq 4\epsilon\qquad\mbox{and}\qquad \mathrm{Prob}\left(\mathcal{A}\right)\geq1 - \frac{8\eta\tau b\epsilon}{L^2(\kappa_h^\delta)^2\delta^2}-\frac{4\sqrt{\epsilon}}{\delta}.$$
    In particular, $\mathrm{Prob}\left(\mathcal{A}\right) \geq 1-O(n\epsilon+\sqrt{\epsilon})\to1$ as $\epsilon\to0.$ Suppose we take the batch size $b=O(n^\alpha)$, $\alpha\in[0,1]$, then the total number of samples consumed is $O\big(\epsilon^{-1}\cdot n^{\max\{\alpha,\frac{1}{2}\}}\big)$.
\end{theorem}

Before proceeding to the proof, we would like to give a brief comment on this theorem. First, if we take $b \leq O(\sqrt{n})$, the total sample compelxity reduces to $O(\sqrt{n}\epsilon^{-1})$. However, we should also note that, as a price for imposing $\cX_s$ constraints to activate KC-regularity, this theorem bounds $\mathbb{E}[\|\cGphe(x_{\mathrm{out}})\|^2 \, |\, \mathcal{A}]$  where $\cA$ is a high probability event. Though $\lim_{\epsilon\to0}\mathrm{Prob}(\cA)=1$, this bound is still slightly weaker than the usual in expectation bound on $\mathbb{E}[\|\cGphe(x_{\mathrm{out}})\|^2]$. Therefore, it still remains an interesting question whether one can further improve the analysis technique and obtain the standard in expectation complexity bound.

\begin{proof}
By lemma \ref{lemma:proxSARAH-grad-map-mid}, ignoring the restricted primal gradient mapping terms yields  
\begin{eqnarray} 
\frac{\Delta_\Psi}{\gamma\eta} \geq
\mathbb{E}\bigg[\sum_{s=1}^{S}\sum_{k=0}^{\tau_s-1}D_h(\bar{x}_{s,k+1},x_{s,k})\bigg] 
\geq \frac{\mu}{2\gamma^2}\mathbb{E}\bigg[\sum_{s=1}^{S}\sum_{k=0}^{\tau_s-1}\|x_{s,k+1}-x_{s,k}\|^2\bigg]\,. \nonumber
\end{eqnarray}

By expanding the expectation over all possible $\mathcal{I}_1, \mathcal{I}_2,$ and $\{\tau_s\}$, we have for all $m_1,m_2>0$ that 
\begin{eqnarray} 
\label{thm:proxSARAH-finite-1}
\frac{2\gamma\Delta_\Psi}{\mu\eta} \!\!\!\!& \geq & \!\!\!\!\mathbb{E}\bigg[\sum_{s=1}^{S}\sum_{k=0}^{\tau_s-1}\!\|x_{s,k+1}\!-\!x_{s,k}\|^2\bigg] \\
& \geq &\!\!\!\! \sum_{|\mathcal{I}_1|<m_1}\sum_{|\mathcal{I}_2|< m_2}\sum_{\tau_s\in[\tau-1],s\in\mathcal{I}_1}\!\!\mathrm{Prob}\big(\mathcal{I}_1,\mathcal{I}_2,\{\tau_s\}_{s\in\mathcal{I}_1}\big)\mathbb{E}\bigg[\!\sum_{s=1}^S\!\sum_{k=0}^{\tau_s-1}\!\|x_{s,k+1}-x_{s,k}\|^2 \,\Big|\, \mathcal{I}_1,\mathcal{I}_2,\{\tau_s\}_{s\in\mathcal{I}_1}\!\bigg]\nonumber\\
& \geq &\!\!\!\! \sum_{|\mathcal{I}_1|<m_1}\sum_{|\mathcal{I}_2|< m_2}\sum_{\tau_s\in[\tau-1],s\in\mathcal{I}_1}\!\!\mathrm{Prob}\big(\mathcal{I}_1,\mathcal{I}_2,\{\tau_s\}_{s\in\mathcal{I}_1}\big)\mathbb{E}\bigg[\!\sum_{s\in\mathcal{I}_1}\!\!\sum_{k=0}^{\tau_s-1}\!\|x_{s,k+1}-x_{s,k}\|^2 \,\Big|\, \mathcal{I}_1,\mathcal{I}_2,\{\tau_s\}_{s\in\mathcal{I}_1}\!\bigg]\nonumber 
\end{eqnarray}
Note that for $\forall s\in\mathcal{I}_1$, according to the discussion in Lemma \ref{lemma:event-grad-map}, we have  $\|x_{s,\tau_s}-x_{s,0}\|  \geq \delta/4.$
Repeating the analysis of \eqref{lm:event-halfway-2} and \eqref{lm:event-halfway-3} yields $\sum_{k=0}^{\tau_s-1}\|x_{s,k+1}-x_{s,k}\|^2  \geq \frac{\delta^2}{16\tau_s}$, which always holds true. Then substituting this lower bound to \eqref{thm:proxSARAH-finite-1} gives 
\begin{equation}
    \label{thm:proxSARAH-finite-2}
    \sum_{|\mathcal{I}_1|<m_1}\sum_{|\mathcal{I}_2|< m_2}\sum_{\tau_s\in[\tau-1],s\in\mathcal{I}_1}\mathrm{Prob}\big(\mathcal{I}_1,\mathcal{I}_2,\{\tau_s\}_{s\in\mathcal{I}_1}\big)\cdot\sum_{s\in\mathcal{I}_1}\tau_s^{-1}\leq\frac{32\gamma\Delta_\Psi}{\mu\eta\delta^2}\,.
\end{equation} 
Define the events  
\begin{equation}
    \label{defn:event-A3}
    \mathcal{A}_3:=\big\{\omega: \mbox{the index of } x_{\mathrm{out}}\mbox{ is from }\mathcal{I}_2^c\big\}\quad\mbox{and}\quad \mathcal{A}:=\mathcal{A}_1^c(m_1)\cap\mathcal{A}_2^c(m_2)\cap\mathcal{A}_3.
\end{equation} 
Then we have
\begin{eqnarray}
    \label{thm:proxSARAH-finite-33}
    \!\!\!\!\!\!\!\!\!\!\mathrm{Prob}\left(\mathcal{A}\right)\!\!\! &=&\!\!\! \sum_{|\mathcal{I}_1|<m_1}\sum_{|\mathcal{I}_2|< m_2}\sum_{\tau_s\in[\tau-1],s\in\mathcal{I}_1}\mathrm{Prob}\big(\mathcal{I}_1,\mathcal{I}_2,\{\tau_s\}_{s\in\mathcal{I}_1}\big)\cdot\mathrm{Prob}\big(\cA_3\mid\mathcal{I}_1,\mathcal{I}_2,\{\tau_s\}_{s\in\mathcal{I}_1}\big)\\
    &=&\!\!\! \sum_{|\mathcal{I}_1|<m_1}\sum_{|\mathcal{I}_2|< m_2}\sum_{\tau_s\in[\tau-1],s\in\mathcal{I}_1}\mathrm{Prob}\big(\mathcal{I}_1,\mathcal{I}_2,\{\tau_s\}_{s\in\mathcal{I}_1}\big)\cdot\left(1-\frac{|\mathcal{I}_2|}{(S-|\mathcal{I}_1|)\tau+\sum_{s\in\mathcal{I}_1}\tau_s}\right)\nonumber\\
    & \geq &\!\!\! \sum_{|\mathcal{I}_1|<m_1}\sum_{|\mathcal{I}_2|< m_2}\sum_{\tau_s\in[\tau-1],s\in\mathcal{I}_1}\mathrm{Prob}\big(\mathcal{I}_1,\mathcal{I}_2,\{\tau_s\}_{s\in\mathcal{I}_1}\big)\cdot\left(1-\frac{m_2}{(S-|\mathcal{I}_1|)\tau+\sum_{s\in\mathcal{I}_1}\tau_s}\right).\nonumber
\end{eqnarray}
Note that for any positive numbers $y_1,\cdots,y_S>0$,  the arithmetic-harmonic inequality states that $$\frac{y_1+y_2+\cdots+y_S}{S}\geq\frac{S}{\frac{1}{y_1}+\frac{1}{y_2}+\cdots+\frac{1}{y_S}}\,.$$ 
Applying this inequality gives 
$$\frac{1}{\tau} + \frac{\sum_{s\in\mathcal{I}_1}\tau_s^{-1}}{S}= \frac{(S-|\mathcal{I}|_1)\tau^{-1} + \sum_{s\in\mathcal{I}_1}\tau_s^{-1}}{S}\geq \frac{S}{(S-|\mathcal{I}|_1)\tau+ \sum_{s\in\mathcal{I}_1}\tau_s}\,,$$
which implies that 
$$1-\frac{m_2}{(S-|\mathcal{I}_1|)\tau+\sum_{s\in\mathcal{I}_1}\tau_s}\geq 1-\frac{m_2}{S\tau}-\frac{m_2}{S^2}\sum_{s\in\mathcal{I}_1}\tau_s^{-1}\,.$$
Substituting this bound to \eqref{thm:proxSARAH-finite-33} and setting $m_1^* = S/4$ yields
\begin{eqnarray}
    \label{thm:proxSARAH-finite-44}
    \!\!\!\mathrm{Prob}\left(\mathcal{A}\right) \!\!\!\!& \geq &\!\!\!\! \sum_{|\mathcal{I}_1|<m_1^*}\sum_{|\mathcal{I}_2|< m_2}\sum_{\tau_s\in[\tau-1],s\in\mathcal{I}_1}\mathrm{Prob}\big(\mathcal{I}_1,\mathcal{I}_2,\{\tau_s\}_{s\in\mathcal{I}_1}\big)\cdot\left(1-\frac{m_2}{S\tau}-\frac{m_2}{S^2}\sum_{s\in\mathcal{I}_1}\tau_s^{-1}\right)\nonumber\\
    & \overset{(i)}{\geq} &\!\!\!\! \left(1-\frac{m_2}{S\tau}\right)\mathrm{Prob}\left(\mathcal{A}_1^c(m_1^*)\cap\mathcal{A}_2^c(m_2)\right) - \frac{m_2}{S^2}\cdot\frac{32\gamma\Delta_\Psi}{\mu\eta\delta^2}\\
    & \geq & \!\!\!\! \left(1-\frac{m_2}{S\tau}\right)\big(1-\mathrm{Prob}\left(\mathcal{A}_1(m_1^*)\right)-\mathrm{Prob}\left(\mathcal{A}_2(m_2)\right)\big) - \frac{m_2}{S^2}\cdot\frac{32\gamma\Delta_\Psi}{\mu\eta\delta^2}\nonumber\\
    & \overset{(ii)}{\geq} & \!\!\!\! 1 - \frac{8\eta\tau b\epsilon}{L^2(\kappa_h^\delta)^2\delta^2} - \frac{128\eta\Delta_\Psi}{\gamma\mu\delta^2\cdot m_2} - \frac{m_2}{S\tau} - \frac{m_2}{S^2}\cdot\frac{32\gamma\Delta_\Psi}{\mu\eta\delta^2}\nonumber\\
    & \overset{(iii)}{\geq} & \!\!\!\! 1 - \frac{8\eta\tau b\epsilon}{L^2(\kappa_h^\delta)^2\delta^2} - \frac{128\eta\Delta_\Psi}{m_2\gamma\mu\delta^2} - \frac{9m_2}{8S\tau}\nonumber
\end{eqnarray}
where (i) is due to \eqref{thm:proxSARAH-finite-2} and the fact that 
\begin{equation}
    \label{thm:proxSARAH-finite-55}
    \sum_{|\mathcal{I}_1|<m_1^*}\sum_{|\mathcal{I}_2|< m_2}\sum_{\tau_s\in[\tau-1],s\in\mathcal{I}_1}\mathrm{Prob}\big(\mathcal{I}_1,\mathcal{I}_2,\{\tau_s\}_{s\in\mathcal{I}_1}\big) = \mathrm{Prob}\big(\mathcal{A}_1^c(m_1^*)\cap\mathcal{A}_2^c(m_2)\big)
\end{equation}
(ii) is due to Lemma \ref{lemma:event-halway} with $m_1^*=S/4$ and $S = \big\lceil\frac{16\Delta_\Psi}{\tau\gamma\eta\mu\epsilon}\big\rceil$: 
$$\mathrm{Prob}\big(\mathcal{A}_1(S/4)\big)\leq \frac{32\gamma\tau\Delta_\Psi}{\mu\delta^2S/4} = \frac{16\Delta_\Psi}{S\tau\gamma\eta\mu}\cdot\frac{8\gamma^2\tau^2\eta}{\delta^2} \leq \frac{8\tau b\epsilon}{L^2(\kappa_h^\delta)^2\delta^2},$$
and (iii) is because we require $\epsilon\leq \frac{\delta^2}{16}\cdot\min\left\{\frac{L^2(\kappa_h^\delta)^2}{b\tau},\frac{1}{9\eta^2}\right\} = O(1/n)$ such that 
$$\frac{m_2}{S^2}\cdot\frac{32\gamma\Delta_\Psi}{\mu\eta\delta^2} = \frac{m_2}{S\tau}\cdot\frac{16\Delta_\Psi}{S\tau\gamma\eta\mu}\cdot\frac{2\gamma^2\tau^2}{\delta^2}\leq \frac{m_2}{8S\tau}$$
Therefore, to maximize the above probability, we can choose $m_2^*=\sqrt{\frac{S\tau\eta\Delta_\Psi}{\gamma\mu}}\cdot\frac{32}{3\delta}$ such that
$$\frac{128\eta\Delta_\Psi}{m_2^*\gamma\mu\delta^2} + \frac{9m_2^*}{8S\tau} = 2\sqrt{\frac{128\eta\Delta_\Psi}{m_2^*\gamma\mu\delta^2} \cdot  \frac{9m_2^*}{8S\tau}} = \sqrt{\frac{\eta\Delta_\Psi}{\gamma\mu S\tau}}\cdot\frac{24}{\delta}\leq \frac{6\eta\sqrt{\epsilon}}{\delta}\leq \frac{4\sqrt{\epsilon}}{\delta}\,, $$
where the last inequality is due to the fact that $\eta = \frac{\sqrt{2\tau}}{\sqrt{7\tau}+\sqrt{2b}}\leq\sqrt{2/7}.$ Combining all the above discussion, we can conclude that  
$$\mathrm{Prob}\left(\mathcal{A}\right)\geq1 - \frac{8\eta\tau b\epsilon}{L^2(\kappa_h^\delta)^2\delta^2}-\frac{4\sqrt{\epsilon}}{\delta} = 1-O(n\epsilon+\sqrt{\epsilon})$$
when taking $m_1^* = S/4$ and $m_2^* = \sqrt{\frac{S\tau\eta\Delta_\Psi}{\gamma\mu}}\cdot\frac{32}{3\delta}.$ Note that the requirement on $\epsilon$ further implies that $\mathrm{Prob}(\cA_1(m_1^*))\leq 1/2$ and $\mathrm{Prob}(\cA_2(m_2^*))\leq 1/4$, and hence $\mathrm{Prob}\big(\cA_1^c(m_1^*)\cap\cA_2^c(m_2^*)\big)\geq1/4$. By ignoring the Bregman divergence terms of Lemma \ref{lemma:proxSARAH-grad-map-mid}, we obtain
\begin{eqnarray*}
    \frac{\Delta_\Psi}{\gamma\eta} &\geq& \frac{\mu}{8} \mathbb{E}\bigg[\sum_{s=1}^S\sum_{k=0}^{\tau_s-1}\big\|\cGshe(x_{s,k})\big\|^2\bigg]  \geq \frac{\mu}{32}\mathbb{E}\bigg[\sum_{(s,k)\in\mathcal{I}_2^c}\big\|\cGshe(x_{s,k})\big\|^2\,\Big|\,\cA_1^c(m_1^*)\cap\cA_2^c(m_2^*)\bigg] \nonumber.
\end{eqnarray*}
Conditioning on $\cA_1^c(m_1^*)\cap\cA_2^c(m_2^*)$, regardless of the random sets $\mathcal{I}_1, \mathcal{I}_2$ and $\{\tau_s\}_{s\in\mathcal{I}_1}$, we have  
\begin{eqnarray}
    |\mathcal{I}_2^c| = (S-|\mathcal{I}_1|)\tau+\sum_{s\in\mathcal{I}_1}\tau_s - |\mathcal{I}_2|\geq  \frac{3S\tau}{4} - m_2^*\geq S\tau\Big(\frac{3}{4}-\frac{8\eta\sqrt{\epsilon}}{3\delta}\Big) \geq \frac{S\tau}{2}.\nonumber
\end{eqnarray}
where the second last inequality is because $|\mathcal{I}_2|\leq m_2^*\leq \frac{8S\tau\eta\sqrt{\epsilon}}{3\delta}$, and the last inequality is because  $\frac{8\eta\sqrt{\epsilon}}{3\delta}\leq 2/9< 1/4$ since we require $\epsilon\leq\frac{\delta^2}{16}\cdot\frac{1}{9\eta^2}$. As a result, 
\begin{eqnarray}
\label{thm:proxSARAH-finite-66}
    \frac{\Delta_\Psi}{\gamma\eta} &\geq& \frac{\mu}{32}\cdot\frac{S\tau}{2}\cdot\mathbb{E}\bigg[\frac{|\mathcal{I}_2^c|}{S\tau/2}\cdot\frac{\sum_{(s,k)\in\mathcal{I}_2^c}\big\|\cGshe(x_{s,k})\big\|^2}{|\mathcal{I}_2^c|}\,\Big|\,\cA_1^c(m_1^*)\cap\cA_2^c(m_2^*)\bigg] \\
    & \geq & \frac{\mu S\tau}{64}\mathbb{E}\bigg[\frac{\sum_{(s,k)\in\mathcal{I}_2^c}\big\|\cGshe(x_{s,k})\big\|^2}{|\mathcal{I}_2^c|}\,\Big|\,\cA_1^c(m_1^*)\cap\cA_2^c(m_2^*)\bigg] \nonumber\\
    & = & \frac{\mu S\tau}{64}\mathbb{E}\Big[\big\|\cGphe(x_{\mathrm{out}})\big\|^2\,\big|\,\cA_1^c(m_1^*)\cap\cA_2^c(m_2^*)\cap\cA_3\Big]\,, \nonumber
\end{eqnarray}
where the last equality is due to the definition of $x_{\mathrm{out}}$, $\cA_3$, and the fact that $\cGshe(\cdot)$ coincides with $\cGphe(\cdot)$ in $\mathcal{I}_2^c$. As a result, we have
$$\mathbb{E}\Big[\big\|\cGphe(x_{\mathrm{out}})\big\|^2\,\big|\,\cA\Big]\leq \frac{64\Delta_\Psi}{\gamma\eta\mu\tau S}\leq 4\epsilon.$$
Given the choice of $S$, $\gamma$, $\eta$, and the fact that $\tau = \lceil n/b\rceil$, the total sample complexity will be 
$$S(n+b\tau) = (n+b\tau)\cdot\bigg\lceil\frac{16\Delta_\Psi}{\tau\gamma\eta\mu\epsilon}\bigg\rceil = O\left(\frac{L\kappa_h^\delta\Delta_\Psi\sqrt{n}}{\mu\epsilon}\cdot\Big(1+\sqrt{b/\tau}\Big)\right),$$
which indicates an $O(\sqrt{n}\epsilon^{-1})$ sample complexity for all $b\leq O(\tau)$ (or equivalently, $b\leq O(\lceil\sqrt{n}\rceil)$). If larger batch size $b = \lceil n^\alpha\rceil$ with $\alpha\in(1/2,1]$ is taken, then we obtain an $O(n^\alpha\epsilon^{-1})$ complexity. 
\end{proof}

\subsection{Solving subproblems for $\bar{x}_{s,k+1}$ update}
From the above analysis, we can observe that the key purpose of introducing the epoch-wise constraint $x\in\cX_s$ is to restrict the iterations within a reasonably bounded region where the kernel conditioning regularity (Assumption \ref{assumption:kernel-conditioning}) is activated. 
However, this may also bring difficulties in solving the subproblem. In this subsection, we will discuss a few possible solution approaches to subproblem  \eqref{alg:proxSARAH-update}. 

According to our discussion in Section \ref{subsec:kernel-conditioning}, most of the popular kernels either takes the form of the composition of norm $h(x) := H(\|x\|)$, its block-separable variant $h(x) := \sum_i^mH_i(\|x_i\|)$, or the element-separable case $h(x) := \sum_{i=1}^dh_i(x_i)$. Next, let us discuss them one by one. 

\subsubsection{Element-separable cases}
First, let us briefly discuss the simplest scenario. When $h$ and $\phi$ are element-separable, as shall be discussed in Appendix \ref{section:multi-blocks}, the set $\cX_s$ will take a Cartesian product form of $\cX_s:=\cX_s^1\times\cdots\times\cX_s^d$ where each $\cX_s^i$ is a simple 1-dimensional closed interval and subproblem \eqref{alg:proxSARAH-update} becomes separable. For each element $x_i$, it reduces to solving a problem of form 
\begin{equation}
    \label{eqn:separable}
    \min_{x_i\in\RR}\,\, a_i\cdot x_i + \eta\phi_i(x_i) + h_i(x_i)\quad\mbox{s.t.}\quad x_i\in[b_i,c_i],
\end{equation}
for some constants $a_i,b_i,c_i$. As a 1-dimensional convex problem, it can be easily solved by 
\begin{enumerate}
    \item[(i).] Suppose problem \eqref{eqn:separable} has a closed form solution when removing the constraint $x_i\in[b_i,c_i]$, and we denote this solution by $x_i^*$. If $x_i^*\in[b_i,c_i]$, then it is optimal to problem \eqref{eqn:separable}, otherwise the optimal solution will be the better one between $\{b_i,c_i\}$.
    \item[(ii).] Suppose problem \eqref{eqn:separable} does not have a closed form solution even without constraint $x_i\in[b_i,c_i]$. Then we can apply either golden section search or other one-dimensional optimization method to obtain a solution. When golden section search is applied, at most $O(\ln \epsilon_{\mathrm{tol}}^{-1})$ time is required to obtain a point that is $\epsilon_{\mathrm{tol}}$-close to the optimal solution. 
\end{enumerate}
Therefore, as a large class of kernel functions, the element-separable kernels are always easy to handle, regardless of the availability of a closed form solution. 
\subsubsection{General non-element-separable cases}
In this subsection, we will consider the general block-separable case where each $h_i(x_i)$ is a general kernel. When $x = [x_1,\cdots,x_m]$ is partitioned in to $m$ blocks, similar to the element-separable case, the set $\cX_s$ will take a Cartesian product form of $\cX_s:=\cX_s^1\times\cdots\times\cX_s^m$. Then subproblem \eqref{alg:proxSARAH-update} will reduce to solving 
\begin{equation}
    \label{eqn:separable-block}
    \min_{x_i}\,\, \langle u_i,x_i\rangle + \eta\phi_i(x_i) + h_i(x_i)\quad\mbox{s.t.}\quad x_i\in\cX_s^i,
\end{equation}
for some vector $u_i$, and for each block $x_i$, $1\leq i\leq m$.  

Due to the separability of the subproblems, it is sufficient to restrict the discussion to the single block case where $m=1$. From now on, we will focus on this single block scenario and discuss how it can be solved. First of all, let us bound the number of times that the constraint $x\in\cX_s$ is active. 
\begin{lemma}
    \label{lemma:event-active}
    For any $S$ epochs generated by Algorithm \ref{algorithm:proxSARAH}, define the set $\mathcal{I}_4$ as
    $$\mathcal{I}_4:=\left\{(s,k) :  \mbox{constraint } x\in\cX_s\mbox{ is active at iteration } (s,k)\right\},$$
    then it holds that 
    $\mathbb{E}\big[|\mathcal{I}_4|\big] \leq \frac{32L\kappa_h^\delta\Delta_\Psi}{\mu\delta^2}\cdot\left(1+\sqrt{\frac{7n}{2b^2}}\right) \leq \frac{2\epsilon}{\delta^2}\cdot S\tau$.
\end{lemma}
\begin{proof}
    Ignoring the restricted primal gradient mapping terms in Lemma \ref{lemma:proxSARAH-grad-map-mid} yields   
\begin{eqnarray} 
    \label{haha}
    \frac{\Delta_\Psi}{\gamma\eta}  \geq  \mathbb{E}\left[\sum_{s=1}^S\sum_{k=0}^{\tau_s-1}D_h(\bar{x}_{s,k+1},x_{s,k})\right] 
    \geq  \mathbb{E}\left[\frac{\mu}{2}\sum_{(s,k)\in\mathcal{I}_4} \|\bar{x}_{s,k+1}-x_{s,k}\|^2\right]\geq \frac{\mu\delta^2}{32}\mathbb{E}\big[|\mathcal{I}_4|\big],
\end{eqnarray}
where the last inequality is because Line 9 of Algorithm \ref{algorithm:proxSARAH}, which suggests $\|\bar{x}_{s,k+1}-x_{s,k}\|\geq\delta/4$ when the set constraint $x\in\cX_s$ is active. Dividing both sides by ${\mu\delta^2}/{32}$ and substitute the values of $\eta,\gamma,\tau,b$ and $S$ in Theorem \ref{theorem:proxSARAH-finite} proves the lemma. 
\end{proof}
\noindent It can be observed that the upper bound on $\mathbb{E}[|\mathcal{I}_4|]$ is at most an $O(\epsilon)$-fraction of  the total iteration number. And it decreases as the batch size $b$ increase, when we take a large batch size $b = O(\sqrt{n})$, the factor $\sqrt{7n/2b^2}=O(1)$. Moreover, we should note that in the second inequality of \eqref{haha}, we adopted a very loose bound by omitting all the $(s,k)\notin\mathcal{I}_4$ and using  $\mu$ to lower bound $\mu_h([x_{s,k},\bar{x}_{s,k+1}])$, which is potentially much larger than $\mu$. Therefore, the actual cardinality of $\mathcal{I}_4$ can potentially be much smaller than the bound in Lemma \ref{lemma:event-active}.

Overall, on the average, the constraint $x\in\cX_s$ will become active for at most a constant amount of time. Therefore, a convenient heuristic in this case will be first solving the subproblem \eqref{alg:proxSARAH-update} without the constraint $x\in\cX_{s}$, then in most cases we will obtain a point inside $\cX_{s}$, which will also be optimal to the original subproblem \eqref{alg:proxSARAH-update} with constraint. In these cases, if the kernel $h$ allows a closed form solution for the unconstrained variant of \eqref{alg:proxSARAH-update}, then such a closed form solution can be utilized for most of the iterations. Otherwise, one can use proximal gradient  method \cite{beck2017first} to solve \eqref{alg:proxSARAH-update} without constraint. As the iterates' distance to the optimal solution is non-expansive for proximal gradient method, all iterations will stay in a well bounded area in which the condition number of $h$ is controlled by the kernel conditioning regularity, hence providing an $O(\kappa_h^\delta\ln \epsilon_{\mathrm{tol}}^{-1})$ iterations complexity for any target tolerance $\epsilon_{\mathrm{tol}}$. 
However, if the solving the unconstrained version of \eqref{alg:proxSARAH-update} gives a point outside $\cX_s$, then we will have to consider the original constrained subproblem \eqref{alg:proxSARAH-update}. In this case, if the problem does not have nonsmooth term, i.e. $\phi=0$, then one may use projected gradient method that still has an $O(\kappa_h^\delta\ln \epsilon_{\mathrm{tol}}^{-1})$ iterations complexity. Now, suppose $\phi\neq0$ and the proximal operator of $\mathrm{id}_{\cX_s}+\phi$ is not available, then we can solve the following splitting reformulation
\begin{equation}
    \label{prob:splitting}
    \min_{x,y}\,\, \frac{h(x)+h(y)}{2} + \phi(x) + \mathrm{id}_{\cX_s}(y)\quad\mbox{s.t.}\quad x-y = 0.
\end{equation}
For the linear consensus constraint $x-y=0$, the corresponding coefficient matrix is $[I_{d\times d},-I_{d\times d}]$ and its condition number is 1. Then many primal-dual algorithms can achieve an $O(\kappa_h^\delta\ln \epsilon_{\mathrm{tol}}^{-1})$ complexity for finding an $\epsilon_{\mathrm{tol}}$-optimal solution, see \cite{zhu2023unified}. Due to Lemma \ref{lemma:event-active}, we only need to deal with this scenario for limited times. This is also how we implement the subproblem solvers in the experiments.

\subsubsection{Composition of norm cases}
Finally, we consider a special case of the composition of norm kernels. According to previous discussion, we only need to discuss the single block case $h(x) = H(\|x\|)$, and then the multi-block case will be straightforward due to the separable structure. Following the discussion of Lemma \ref{lemma:event-active}, the constraint $x\in\cX_s$ can be active for only limited iterations. In these cases, iterative methods can be used to solve the subproblem efficiently because the kernel conditioning regularity guarantees a mild condition number. While for the most cases, one can solve an unconstrained version of the subproblem \eqref{alg:proxSARAH-update}, which has the form
\begin{equation}
    \label{eqn:customize-gen}
    \min_{x}\,\, \langle u, x\rangle + \eta\phi(x) + H(\|x\|),
\end{equation}
for some vector $u\in\RR^d$. Suppose $H(\cdot)$ is a strictly convex, monotonically increasing, and nonnegative function, we discuss a few examples where \eqref{eqn:customize-gen} can be efficiently solved. 
\begin{example}
    \label{example:basic}
    Consider a basic scenario where $\phi(x) = 0$, then \eqref{eqn:customize-gen} reduces to $\min_{x}\,\, \langle u, x\rangle + H(\|x\|).$
    This problem can be solved by a 1-dimensional search.
\end{example}

Due to the monotonicity of $H(\cdot)$, one can observe that the optimal solution should take the form $x = -\alpha\cdot{u}/{\|u\|}$ for some scalar $\alpha\geq0$. Then problem \eqref{eqn:customize-gen} is equivalent to a 1-dimension problem 
\begin{equation}
    \label{eqn:customize-gen-1D}
    \min_{\alpha}\,\, H(\alpha)-\|u\|\cdot\alpha\quad\mathrm{s.t.}\quad   \alpha\geq0.
\end{equation}
Computing the objective gradient of \eqref{eqn:customize-gen-1D} gives $H'(\alpha)-\|u\|$. Because $H(\cdot)$ is strictly convex and increasing, we know $H'(\cdot)\geq0$ and $H'(\cdot)$ is increasing on $[0,+\infty)$, then we know $\alpha^* = 0$ if $H'(0) \geq \|u\|$. Otherwise, there is a unique solution $\alpha^*$ s.t. $H'(\alpha^*) = \|u\|$. 
In some cases, the equation $H'(\alpha) = \|u\|$ allows a closed form solution, then we can directly adopt it, see \cite{bolte2018first}. If no closed-form solution is available, then one can apply either Newton's method or a binary search to find the root. Based on this observation, let us consider a few more examples. 
\begin{example}
    \label{example:ell-1-norm}
    Consider $\ell_1$-regularization term $\phi(x) = \beta\|x\|_1$ for some $\beta>0$. Define the index sets $J_1:=\{i\in[d]:|u_i|\leq \beta\eta\}$ and $J_2=[d]\backslash J_1$. Then the optimal solution $x^*$ to problem \eqref{eqn:customize-gen} will satisfy $x^*_{J_1} = 0$ and $x^*_{J_2} = \argmin\,\,\langle u_{J_2}-\eta\beta\cdot\mathrm{sign}(u_{J_2}),x_{J_2} \rangle + H(\|x_{J_2}\|)$.  
\end{example}
\noindent For any $x$, consider any $i\in J_1$, since $|u_i|\leq \beta\eta$, we know $u_ix_i+\beta\eta|x_i|\geq0$. Hence setting $x_i^*=0$ will minimize the objective value w.r.t. $x_i$, regardless of the other elements of $x$. According to the discussion of Example \ref{example:basic} and the definition of the index set $J_2$, we know $$\mathrm{sign}(x^*_{J_2}) = -\mathrm{sign}(u_{J_2}-\eta\beta\cdot\mathrm{sign}(u_{J_2})) = -\mathrm{sign}(u_{J_2}),$$
therefore $\langle u_{J_2}-\eta\beta\cdot\mathrm{sign}(u_{J_2}),x^*_{J_2} \rangle = \langle u_{J_2},x^*_{J_2} \rangle + \eta\beta\|x^*_{J_2}\|_1$ and $x^*_{J_2}$ also solves the original subproblem $\min\langle u_{J_2},x_{J_2} \rangle +\eta\beta\|x_{J_2}\|_1 +  H(\|x_{J_2}\|)$.\vspace{0.2cm} 

In fact, this result can be  generalized to the group $\ell_1$/$\ell_2$ norm that promotes group sparsity.  
\begin{example}
    \label{example:ell-1-ell-2-norm}
    Suppose $x = [x_1,x_2,\cdots,x_m]$ can be separated into $m$ groups, where each $x_i\in\RR^{d_i}$ is a subvector of $x$.   Consider the group $\ell_1/\ell_2$-regularization $\phi(x) = \sum_{i=1}^m\beta\|x_i\|$ for some $\beta>0$. Then problem \eqref{eqn:customize-gen} can be solved by the following procedure:\vspace{0.1cm}\\ 
    \emph{(i)}. For any $i\in[m]$, let $u_i$ be the subvector of $u$ that corresponds to $x_i$.  Let us define the index sets $J_1:=\{i\in[m]:\|u_i\|\leq \beta\eta\}$ and $J_2=[m]\backslash J_1$. Then for any $i\in J_1$, set $x^*_i = 0$ for any $i\in J_1$. \vspace{0.1cm}\\
    \emph{(ii)}. Construct and solve a new problem 
    $\alpha^* = \argmin_{\alpha\in\RR^{|J_2|}} \,\, \sum_{i\in J_2}(\eta\beta-\|{u}_i\|)\cdot\alpha_i + H(\|\alpha\|)$.\vspace{0.1cm}\\
    \emph{(iii)}. The optimal solution to problem \eqref{eqn:customize-gen} is $x_i^*=0$ for $\forall i\in J_1$, and  $x_i^* = -\frac{\alpha_i^*\cdot{u}_i}{\|{u}_i\|}$ for $\forall i\in J_2.$
\end{example}
\noindent Similar to Example \ref{example:ell-1-norm}, for any $i\in J_1$, we must have  $\langle u_i,x_i\rangle+\beta\eta\|x_i\|\geq0$. Hence setting $x_i^*=0$ will minimize the objective value w.r.t. $x_i$, regardless of the other groups of $x$.
Through a similar argument to Example \ref{example:basic}, each $x_i$ with $i\in J_2$ should take the form of $x_i = -\alpha_i\cdot {u}_i/\|{u}_i\|$ when it is optimal. Hence $\|x_i\|=\alpha_i$ and $\|x\| = \|\alpha\|$. Then we can rewrite the above problem as 
$$\min_{\alpha\in\RR^{|J_2|}} \,\, \sum_{i\in J_2}(\eta\beta-\|{u}_i\|)\cdot\alpha_i + H(\|\alpha\|) 
 \quad\mbox{s.t.}\quad \alpha\geq 0.$$
because $\eta\beta-\|{u}_i\|<0$ for all $i\in J_2$, the $\alpha\geq0$ constraint can be relaxed. Then we can apply the approach for Example \ref{example:basic} to obtain the optimal $\alpha$. \vspace{0.2cm}

As a summary for this subsection, Lemma \ref{lemma:event-active} indicates that the constraint $x\in\cX_s$ can be active for at most $O(\epsilon)$-fraction of the total iterations. Therefore, most of the subproblems \eqref{alg:proxSARAH-update} will actually be unconstrained. If this subproblem allows closed-form solution or can be efficiently evaluated without the constraint $x\in\cX_s$, then we can first ignore this constraint and obtain a solution $\tilde{x}$. If $\tilde{x}\in\cX_s$, then it will be solution to the subproblem \eqref{alg:proxSARAH-update}. Otherwise, we solve the original constrained problem with an appropriate iterative algorithm. Because the KC-regularity guarantees a mild condition number for the subproblem, it will take $O(\kappa_h^\delta\ln\epsilon^{-1}_{\mathrm{tol}})$ iterations for any tolerance $\epsilon_{\mathrm{tol}}>0$.

\section{Instance-free complexity under dual gradient mapping}
\label{section:GrdMap-new}
In Section \ref{section:VR}, we have established an improved $O(\sqrt{n}\epsilon^{-1})$ complexity for finding $\epsilon$-small squared primal gradient mapping. By Corollary \ref{corollary:mismatch}, having $\|\cGphl(x_{\text{out}})\|^2\leq \epsilon$ implies $\mathrm{dist}^2(0,\partial\Psi(x_{\text{out}}))\leq 4L_h^2(\cX)\cdot\epsilon$ for some output $x_{\text{out}}$, where $\cX$ denote the convex hull of all the iterations. In particular, the reason why we use the mismatch factor for $\cX$ instead of the output $x_{\text{out}}$ is that, for stochastic algorithms, $x_{\text{out}}$ is often randomly selected among all iterations. It provides a very desirable $O(\sqrt{n}\epsilon^{-1})$ complexity for finding solution with $O(\epsilon)$-small Fr\'echet measure  when $L_h^2(\cX)$ is mild, which clearly depends on the landscape of the input instance. Therefore, such a bound may fail for hard instances like Example \ref{example:counterexample-1} whose mismatch factor is unreasonably large or even unbounded. Therefore, it is also important to obtain a robust and stable instance-free (worst-case) complexity that holds for all problem instances. Note that the KC-regularity guarantees an $O(1)$ constant mismatch for the dual gradient mapping $\cDphl(\cdot)$, in this section, we will illustrate how to obtain the instance-free complexity for bounding the Fr\'echet measure by analyzing the dual gradient mapping.

\subsection{An adaptive step size control for BPG method}
\label{subsection:TBPG}
Because such an instance-free bound is not available for the basic deterministic setting,  let us start the discussion from the basic deterministic case for the ease of understanding. First of all, we propose a simple mechanism to adaptively determine the step sizes: 
\begin{equation}
    \label{defn:TBPG}
    x_{k+1} = \Tphlk\big(x_k,\nabla f(x_k)\big)\qquad\mbox{with}\qquad \lambda_k = \min\left\{\frac{1}{2L},\,\frac{\mu\delta}{3\rho},\,\frac{\mu\delta}{\|\nabla f(x_k)\|+\rho}\right\}
\end{equation}
where $L$ is introduced by Assumption \ref{assumption:L-smad},  $\mu$ is introduced by Assumption \ref{assumption:SC}, $\delta$ chosen so that $\kappa_h^\delta$ is mild, see Assumption \ref{assumption:kernel-conditioning}, and $\rho$ comes from the following bounded subgradient assumption on the non-differentiable term $\phi$, which will be used throughout Section \ref{section:GrdMap-new}. 

\begin{assumption}
\label{assumption:Lipschtiz-h}
There exists a constant $\rho>0$ such that $\sup_{u\in\partial \phi(x)}\|u\|\leq \rho$ for any $x\in\RR^d$.
\end{assumption}
\noindent In particular, if $\phi = 0$, then $\rho = 0$, the step size rule reduces to $\lambda_k = \min\big\{\frac{1}{2L},\,\frac{\mu\delta}{\|\nabla f(x_k)\|}\big\}$. Intuitively, requiring $\lambda_k\leq\frac{1}{2L}$ ensures that \eqref{defn:TBPG} is a descent step, while requiring $\lambda_k\leq \frac{\mu\delta}{\|\nabla f(x_k)\|+\rho}$ ensures that 
\begin{equation}
\label{lm:TBPG-descent-0}
\mu\|x_{k+1}-x_k\| \leq \|\nabla h(x_{k+1}) - \nabla h(x_k)\| \overset{\eqref{lm:grd-vs-grdmap-1}}{\leq} \mu\delta\cdot\frac{\|\nabla f(x_k)+u_{k+1}\|}{\|\nabla f(x_k)\|+\rho}\leq \mu\delta, 
\end{equation}
where $u_{k+1}\in\partial \phi(x_{k+1})$ and it satisfies $\|u_{k+1}\|\leq\rho$. That is, this part of step size design guarantees that $\|x_{k+1}-x_{k}\|\leq \delta$ so that kernel conditioning can be activated and more results can be exploited. Finally, $\lambda_k\leq \frac{\delta\mu}{3\rho}$ is only a technical requirement that simplifies the proof.  
Given this observation, we can apply the kernel condition regularity to obtain the following result for the adaptive scheme \eqref{defn:TBPG}.
 
\begin{lemma}
\label{lemma:TBPG-descent}
Under Assumptions \ref{assumption:L-smad}, \ref{assumption:kernel-conditioning}, \ref{assumption:SC}, and \ref{assumption:Lipschtiz-h}, the update \eqref{defn:TBPG} simultaneously satisfies both  
$$\Psi(x_{k+1})\leq  \Psi(x_k) -\frac{3L\mu_h([x_{k},x_{k+1}])}{2}\cdot\|x_{k+1}-x_{k}\|^2$$
and
$$\Psi(x_{k+1}) \leq  \Psi(x_k) - \frac{\min\left\{\rho\mu\delta,\,\,3\min\big\{ \frac{1}{2L}, \frac{\mu\delta}{3\rho}\big\}\cdot\big\|\cDphlk(x_k)\big\|^2\right\}}{4L_h([x_{k},x_{k+1}])}.$$
Moreover, the two successive iterates are $\delta$-close: $\|x_{k+1}-x_k\|\leq\delta$.  
\end{lemma}
\begin{proof}

First, let us establish the descent results of this lemma. By standard  analysis, we have
\begin{eqnarray}
\label{lm:TBPG-descent-1}
\Psi(x_{k+1}) & \leq & \Psi(x_k) - \Big(\frac{1}{\lambda_k}-L\Big)D_h(x_{k+1},x_k) - \frac{1}{\lambda_k}D_h(x_k,x_{k+1})\nonumber\\
& \leq & \Psi(x_k) - \frac{1}{2\lambda_k}D_h(x_{k+1},x_k) - \frac{1}{\lambda_k}D_h(x_k,x_{k+1})\\
& \leq & \Psi(x_k) -\frac{3L\mu_h([x_{k},x_{k+1}])}{2}\cdot\|x_{k+1}-x_{k}\|^2\nonumber,
\end{eqnarray}
where the second inequality is due to $\lambda_k\leq 1/2L$. By \cite[Theorem 2.1.5, Eq.(2.1.10)]{nesterov2018lectures}, we also have 
\begin{eqnarray}
\label{lm:TBPG-descent-2}
D_h(x_{k+1},x_{k}) \geq  \frac{\|\nabla h(x_{k})-\nabla h(x_{k+1})\|^2}{2L_h([x_{k},x_{k+1}])} = \frac{\lambda_k^2\cdot\|\cDphlk(x_k)\|^2}{2L_h([x_{k},x_{k+1}])}.
\end{eqnarray}
A similar inequality also holds for $D_h(x_k,x_{k+1})$. Suppose $\lambda_k = \frac{\mu\delta}{\|\nabla f(x_k)\|+\rho}$, then this situation may only happen if  $\|\nabla f(x_k)\|\geq2\rho$ such that  $\frac{\mu\delta}{\|\nabla f(x_k)\|+\rho}\leq\frac{\mu\delta}{3\rho}$. In this situation, with $L\leq\frac{1}{2\lambda_k}$, the second row of \eqref{lm:TBPG-descent-1} and \eqref{lm:TBPG-descent-2} indicate that  
\begin{eqnarray} 
\Psi(x_{k+1}) & \leq & \Psi(x_k) - \frac{3\lambda_k\cdot\|\cDphlk(x_k)\|^2}{4L_h([x_{k},x_{k+1}])} = \Psi(x_k) - \frac{3\|\cDphlk(x_k)\|}{4L_h([x_{k},x_{k+1}])}\cdot\frac{\mu\delta\|\nabla f(x_k)+u_{k+1}\|}{\|\nabla f(x_k)\|+\rho},\nonumber
\end{eqnarray}
where the last inequality is due to the fact that $\cDphlk(x_k) = \nabla f(x_k)+u_{k+1}$. Because $\|\nabla f(x_k)\|\geq2\rho$ and $\|u_{k+1}\|\leq\rho$, we have $\|\cDphlk(x_k)\|\geq\rho$ and 
$$\frac{\mu\delta\|\nabla f(x_k)+u_{k+1}\|}{\|\nabla f(x_k)\|+\rho}\geq \frac{\mu\delta(\|\nabla f(x_k)\|-\|u_{k+1}\|)}{\|\nabla f(x_k)\|+\rho}\geq\frac{\mu\delta}{3}.$$
Consequently, we have 
$$\Psi(x_{k+1}) \leq  \Psi(x_k) - \frac{\rho\mu\delta}{4L_h([x_{k},x_{k+1}])}.$$
If $\lambda_k = \min\big\{ \frac{1}{2L}, \frac{\mu\delta}{3\rho}\big\}$, the second row of \eqref{lm:TBPG-descent-1} and \eqref{lm:TBPG-descent-2} indicate that 
\begin{eqnarray} 
\Psi(x_{k+1}) \leq \Psi(x_k) - \frac{3\min\big\{ \frac{1}{2L}, \frac{\mu\delta}{3\rho}\big\}}{4L_h([x_{k},x_{k+1}])}\cdot\big\|\cDphlk(x_k)\big\|^2\nonumber.
\end{eqnarray}
Therefore, no matter which value $\lambda_k$ takes, it will at least achieve the minimum descent among the two cases. Hence we complete proof of the lemma.   
\end{proof}

Combining the above results, we can obtain the following bound on the iterations. 

\begin{lemma}
\label{lemma:TBPG-limit-movement}
Consider the update \eqref{defn:TBPG}, for any target accuracy $\epsilon\leq\max\big\{\rho^2, \frac{2L\rho\mu\delta}{3}\big\}$,  we have
\begin{equation} 
\Psi(x_{k+1}) \leq \Psi(x_k)-\frac{3}{4}\sqrt{\frac{\epsilon}{\kappa_h^\delta}\cdot\min\bigg\{1\,, \frac{2L\mu\delta}{3\rho}\bigg\}}\cdot\|x_{k+1}-x_{k}\|,\nonumber
\end{equation} 
as long as $\big\|\cDphlk(x_k)\big\|^2\geq\epsilon$. Denote $T_\epsilon:=\min\big\{k: \|\cDphlk(x_k)\|^2\leq \epsilon, k\geq 0 \big\}$, then 
    \begin{equation} 
        \max_{0\leq k\leq T_\epsilon} \|x_k-x_0\|\leq R_\epsilon:= \frac{4}{3}\sqrt{\max\bigg\{1\,,\frac{3\rho}{2L\mu\delta}\bigg\}}\cdot\frac{\sqrt{\kappa_h^\delta}\Delta_\Psi}{\sqrt{\epsilon}}.\nonumber
    \end{equation} 
\end{lemma}

\begin{proof}
    First of all, our requirement on the target accuracy indicates that $\rho\mu\delta\geq3\min\big\{ \frac{1}{2L}, \frac{\mu\delta}{3\rho}\big\}\cdot\epsilon$. Then the second inequality of Lemma \ref{lemma:TBPG-descent} indicates that 
    \begin{equation}
    \label{lm:TBPG-limit-movement-1}
          \Psi(x_{k+1}) \leq  \Psi(x_k) - \frac{3\min\big\{ \frac{1}{2L}, \frac{\mu\delta}{3\rho}\big\}}{4L_h([x_{k},x_{k+1}])}\cdot\epsilon
    \end{equation}
    as long as $\|\cDphlk(x_k)\|^2\geq\epsilon$.
    Combined with the first inequality of Lemma \ref{lemma:TBPG-descent}, we have  
    \begin{eqnarray}
        \Psi(x_{k+1}) & \leq &  \Psi(x_k) - \frac{1}{2}\bigg(\frac{3L\mu_h([x_{k},x_{k+1}])}{2}\|x_{k+1}-x_{k}\|^2+\frac{3\min\big\{ \frac{1}{2L}, \frac{\mu\delta}{3\rho}\big\}}{4 L_h([x_{k},x_{k+1}])}\cdot\epsilon\bigg)\nonumber\\
        & \leq & \Psi(x_k)-\sqrt{\frac{9\mu_h([x_{k},x_{k+1}])}{16L_h([x_{k},x_{k+1}])}\cdot 2L\cdot\min\bigg\{ \frac{1}{2L}\,, \frac{\mu\delta}{3\rho}\bigg\}\cdot\epsilon\|x_{k+1}-x_{k}\|^2}\nonumber\\
        & \leq & \Psi(x_k)-\frac{3}{4}\sqrt{\frac{\epsilon}{\kappa_h^\delta}\cdot\min\bigg\{1\,, \frac{2L\mu\delta}{3\rho}\bigg\}}\cdot\|x_{k+1}-x_{k}\|\,,\nonumber
    \end{eqnarray}
    where the last inequality is due to KC-regularity and $\|x_k-x_{k+1}\|\leq \delta$ (Lemma \ref{lemma:TBPG-descent}). This proves the first part of the lemma. Next, we show the bound on the maximum movement before $T_\epsilon$. By the definition of $T_\epsilon$, we have $\|\nabla \cDphlk(x_k)\| ^2 > \epsilon$ for $k \leq T_\epsilon-1$. As a result    $$\sum_{k=0}^{T_\epsilon-1}\|x_{k+1}-x_k\|\leq\frac{4}{3}\sqrt{\max\bigg\{1\,,\frac{3\rho}{2L\mu\delta}\bigg\}}\cdot\frac{\sqrt{\kappa_h^\delta}\Delta_\Psi}{\sqrt{\epsilon}} = R_\epsilon.$$
    Applying triangle inequality to the above bound proves the rest of the lemma.  
\end{proof}

\noindent Denote the instance's level set as $\mathrm{Lev}_0:=\left\{x: \Psi(x)\leq\Psi(x_0)\right\}$, then the following theorem holds. 

\begin{theorem}
    \label{theorem:TBPG-deterministic}
    Under the setting of Lemma \ref{lemma:TBPG-limit-movement} and let $\cX_\epsilon := \mathrm{Lev}_0\cap B(x_0,R_\epsilon)$ be a compact set, then 
    $$ T_\epsilon\leq \max\bigg\{\frac{8L}{3},\frac{4\rho}{\mu\delta}\bigg\}\cdot\frac{L_h(\cX_\epsilon)\Delta_\Psi}{\epsilon}$$
    where $L_h(\cX_\epsilon)$ may depend on $\epsilon$ for hard instances. Moreover, the solution $x_{T_\epsilon}$ and $x_{T_\epsilon+1}$ satisfy
    $$\big\|\mathcal{D}_{\phi,h}^{\lambda_{T_\epsilon}}(x_{T_\epsilon})\big\|^2\leq \epsilon\qquad\mbox{and}\qquad\mathrm{dist}^2\big(0, \partial \Psi(x_{T_\epsilon+1})\big)\leq \Big(1+\frac{\kappa_h^{\delta}}{2}\Big)^{\!2}\!\!\cdot\epsilon\,.$$
    In the special case where $\phi=0$, we have $\|\nabla\Psi(x_{T_\epsilon})\|^2\leq \epsilon.$
\end{theorem}
\noindent As a remark, for the polynomial kernel where $\mu$ and $\delta$ are $O(1)$, the maximal iterations before finding a point with $\epsilon$-small Fr\'echet measure is reduced to $ T_\epsilon\leq O\big(\frac{\max\{L,\rho\}\cdot L_h(\cX_\epsilon)\Delta_\Psi}{\epsilon}\big)$.
 
\begin{proof}
    By Lemma \ref{lemma:TBPG-descent} and \ref{lemma:TBPG-limit-movement}, it is straightforward that  $\big\{x_k : k\leq T_\epsilon\big\}\subseteq \mathrm{Lev}_0\cap B(x_0,R_\epsilon)=\cX_\epsilon$. Hence,   $L_h([x_{k},x_{k+1}])\leq L_h(\cX_\epsilon)<+\infty$ for $k\leq T_\epsilon-1$. Substituting this upper bound to \eqref{lm:TBPG-limit-movement-1} and then summing the resulting inequalities up for $k\leq T_\epsilon-1$ yields 
    \begin{equation}
        \label{thm:TBPG-det-1}
        \Delta_\Psi\geq \sum_{k=0}^{T_\epsilon-1}\frac{3\min\big\{ \frac{1}{2L}, \frac{\mu\delta}{3\rho}\big\}}{4L_h([x_{k},x_{k+1}])}\cdot\epsilon \geq 
        \frac{3\min\big\{ \frac{1}{2L}, \frac{\mu\delta}{3\rho}\big\}}{4L_h(\cX_\epsilon)} \epsilon \cdot T_\epsilon\,,\nonumber
    \end{equation} 
    which proves the first inequality of the theorem. Note that by Lemma \ref{lemma:TBPG-descent}, our adaptive step size control strategy guarantees that $\|x_{T_\epsilon}-x_{T_\epsilon+1}\|\leq \delta$, then the second inequality of the theorem directly follows Lemma \ref{lemma:GradMap-nonsmooth} and the fact that $\lambda_k\leq 1/2L$. For the differentiable case where $\phi = 0$, the result directly follows the definition of $T_\epsilon$ and the fact that $\cDphl(\cdot) = \nabla f(\cdot) = \nabla \Psi(\cdot)$ for any $\lambda>0$.
\end{proof}

As commented in the theorem, the constant $L_h(\cX_\epsilon)$ with $R_\epsilon=O(1/\sqrt{\epsilon})$ potentially depends on $\epsilon$. For example, for a degree-$(\alpha+2)$ polynomial kernel $h(x) =  {\|x\|^2}/{2} +  {\|x\|^{\alpha+2}}/{(\alpha+2)}$, then the worst-case pessimistic estimation gives $L_h(\cX_\epsilon) = O(\epsilon^{-\frac{\alpha}{2}})$, which, by Theorem \ref{theorem:TBPG-deterministic}, suggests an $O(\epsilon^{-\frac{\alpha+2}{2}})$ instance-free complexity for making the Fr\'echet measure $\epsilon$-small, whose tightness is confirmed by the constructing the following worst-case problem instance.

\begin{proposition}
\label{proposition:example}
For any predetermined accuracy $\epsilon>0$ and polynomial kernel $h(x) = \frac{\|x\|^2}{2} + \frac{\|x\|^{2+\alpha}}{2+\alpha}$ with even integer $\alpha\geq2$,  the  instance in Example \ref{example:counterexample-1} satisfies: (i). $f$ is $(\alpha^2+4)$-smooth adaptable to $h$. (ii). For any $\epsilon<8/\alpha^2$ and $x_1\geq1$, any solution $x$ with  $\|\nabla \Psi(x)\|^2\leq\epsilon$ should satisfy $x_1\geq \Omega\big(\frac{1}{\sqrt{\epsilon}\ln^2\epsilon^{-1}}\big)$. (iii). Let $\{(x_1^k,x_2^k)\}$ be generated by the standard BPG \eqref{defn:BPG-standard} with $\lambda<1/L$ or our adaptive  variant \eqref{defn:TBPG}, suppose the initial point is $x^0=(1,0)$, then  $\min\left\{\|\nabla \Psi(x^k)\|^2:k\leq T\right\}\geq \Tilde{\Omega}\big(T^{-\frac{2}{\alpha+2}}\big).$
\end{proposition}

The proof of Proposition \ref{proposition:example} is relegated to Appendix \ref{appendix:example:counter-1}. Through this proposition,  the potential $\epsilon$-dependence in $L_h(\cX_\epsilon)$ provided by Theorem \ref{theorem:TBPG-deterministic} is in fact a tight characterization of the iteration complexity for the smooth-adaptable problem classes associated with polynomial kernels.  In addition, the argument (ii) indicates that the $O(1/\sqrt{\epsilon})$ bound on $R_\epsilon$ is also tight. In fact, 
we can easily generalize this property to any kernel that satisfies Assumption \ref{assumption:kernel-conditioning}, hence proving the tightness of the $R_\epsilon$. 

Finally, it is also worth noting that for the vanilla BPG update \eqref{defn:BPG-standard} with constant step size $\lambda_k = 1/2L$, a slightly different analysis can also provide a similar bound  $R_\epsilon = O(\Delta_\Psi/\sqrt{\epsilon})$ and $T_\epsilon\leq O(L_h(\cX_\epsilon)\Delta_\Psi/\epsilon)$. However, as the step size control plays a significant role in the next section where a SPIDER style variance reduction is introduced, we only discuss \eqref{defn:TBPG} here for succinctness.

\subsection{Adaptive step size control with stochastic variance reduction}
\label{subsec:TBPG-SVR}
Note that the exact gradient norm $\|\nabla f(\cdot)\|$ is required in the adaptive step size control scheme \eqref{defn:TBPG}, which is inaccessible in the stochastic setting. Moreover, as both $R_\epsilon$ and $T_\epsilon$ are random variables, the complex interplay between them makes a sheer in-expectation analysis insufficient to bound the sample complexity. Instead, a high probability bound will be favorable in the following discussion. Basically, we will still adopt the framework of Algorithm \ref{algorithm:proxSARAH}, while removing the bound constraint $\cX_s$ and the early stop mechanism (Line 9) of each epoch, the responsibility to maintain kernel conditioning will be inherited by the step size control policy. In other words, we set $\cX_s=\RR^d$ in Algorithm \ref{algorithm:proxSARAH}. In addition, we modify the update \eqref{alg:proxSARAH-update} with the following update under adaptive step size control:  
\begin{equation}
    \label{defn:TBPG-svr-1}
    \bar{x}_{s,k+1} = \Tshesk\big(x_{s,k},v_{s,k}\big)\qquad\mbox{with}\qquad \eta_{s,k} = \min\left\{\frac{1}{2\kappa_h^\delta L},\,\frac{\mu\delta}{3\rho},\,\frac{\mu\delta}{\|v_{s,k}\|+\rho}\right\}\,,\quad
\end{equation}
\begin{equation}
    \label{defn:TBPG-svr-2}
    x_{s,k+1} = x_{s,k} + \gamma_{s,k}(\bar{x}_{s,k+1}-x_{s,k}) \qquad\mbox{with}\qquad \gamma_{s,k} = \min\left\{1\,,\,\frac{\sqrt{\epsilon}/2L(\kappa_h^\delta)^2}{\|\nabla h(x_{s,k}) -\nabla h(\bar{x}_{s,k+1})\|}\right\}\,.
\end{equation}
By slightly modifying the analysis of \eqref{lm:TBPG-descent-0} and \eqref{lm:descent-proxSARAH-finite-1}, we obtain the following descent result for the update \eqref{defn:TBPG-svr-1} and \eqref{defn:TBPG-svr-2}, whose proof is omitted.
\begin{lemma}
    \label{lemma:TBPG-svr-descent}
    Given Assumption \ref{assumption:kernel-conditioning} and \ref{assumption:L-smad-finite-sum}, the update \eqref{defn:TBPG-svr-1} and \eqref{defn:TBPG-svr-2} satisfy  $\|\bar{x}_{s,k+1}\!-\!x_{s,k}\| \!\leq\! \delta$ and 
    \begin{equation}
        \Psi(x_{s,k+1}) \leq \Psi(x_{s,k}) - \frac{\gamma_{s,k}}{2\eta_{s,k}}\!\cdot\! D_h(\bar{x}_{s,k+1},x_{s,k}) - \frac{\gamma_{s,k}}{\eta_{s,k}}\!\cdot\! D_h(x_{s,k},\bar{x}_{s,k+1}) + \gamma_{s,k}\|\bar{x}_{s,k+1}-x_{s,k}\|\!\cdot\!\|\mathcal{E}_{s,k}\|\nonumber
    \end{equation} 
    where $\mathcal{E}_{s,k}=\nabla f(x_{s,k})-v_{s,k}$ is the gradient estimation error. 
\end{lemma}

To establish the counterpart of Lemma \ref{lemma:TBPG-descent}, a high probability bound on $\mathcal{E}_{s,k}$ is required. However, simply applying the standard Azuma-Hoeffding inequality may incur additional dependence on problem dimension. To avoid such a dependence, we need the following large deviation bound for vector-valued martingale in 2-smooth normed spaces from \cite{juditsky2008large}. 

\begin{definition}
\label{definition:regular-space}
Let $(E,\trinorm{\cdot})$ denote a finite-dimensional space $E$ equipped with some norm $\trinorm{\cdot}$. We say the space $(E,\trinorm{\cdot})$ (and the norm $\trinorm{\cdot}$ on $E$) is $\kappa$-regular for some $\kappa\in[1,+\infty)$, if there exists a constant $\kappa_+\in[1,\kappa]$ and a norm $\trinorm{\cdot}_+$ on $E$ such that the function $p(x):=\trinorm{x}_+^2$ is $\kappa_+$-smooth and $\trinorm{x}_+$ is $\kappa/\kappa_+$-compatible with $\trinorm{\cdot}$. That is, for $\forall x,y\in E$, we have
$$p(x+y)\leq p(x) + \langle\nabla p(x),y\rangle + \kappa p(y)\qquad\mbox{and}\qquad\trinorm{x}^2\leq \trinorm{x}_+^2 \leq \frac{\kappa}{\kappa_+}\trinorm{x}^2.$$ 
\end{definition}
\noindent We should notice that the $\kappa$ and $\kappa_+$ here has nothing to do with the condition numbers that are widely used throughout the paper. 

\begin{theorem}[Theorem 2.1-(ii), \cite{juditsky2008large}]
    \label{theorem:MDS-LargeDev}
    Suppose $(E,\trinorm{\cdot})$ is $\kappa$-regular for some $\kappa\geq 1$ and $\{\zeta_t\}_{t\geq0}$ is an $E$-valued martingale difference sequence w.r.t. the filtration $\{\mathcal{F}_t\}_{t\geq0}$ and default $\zeta_0=0$. Suppose $\{\zeta_t\}_{t\geq0}$ satisfies the following light-tail property: 
    $$\mathbb{E}\bigg[\exp\left\{\frac{\trinorm{\zeta_t}^\alpha}{\sigma_t^\alpha}\right\}\,\big|\,\mathcal{F}_{t-1}\bigg]\leq \exp\{1\}, \quad \forall t\geq1.$$
    When $\alpha = 2$, for any $N,\gamma\geq0$, it holds that $$\mathrm{Prob}\left(\trinorm{\sum_{i=0}^N\zeta_i}\geq(\sqrt{\kappa}+\gamma)\sqrt{2\sum_{i=0}^N\sigma_i^2}\right)\leq \exp\left\{-\frac{\gamma^2}{3}\right\}.$$
\end{theorem}
Consider $(\mathbb{R}^d,\|\cdot\|)$ where $\|\cdot\|$ stands for the standard Euclidean (L-2) norm that we use throughout this paper. Setting $\kappa=\kappa_+=1$ and $\trinorm{\cdot}=\trinorm{\cdot}_+ = \|\cdot\|$ in Definition \ref{definition:regular-space}, then straight computation shows that $(\mathbb{R}^d,\|\cdot\|)$ is $1$-regular. As a result, we have the following bound for $\mathcal{E}_{s,k}.$

\begin{lemma}
    \label{lemma:TBPG-error}
    Suppose Assumptions \ref{assumption:kernel-conditioning} and  \ref{assumption:L-smad-finite-sum} hold. For any epoch $s\geq1$  and let $p_s = \frac{6q}{\pi^2s^2\tau}$ for some $q\in(0,1)$. Suppose we select $|\cB_{s,k}|=b_s$ for all $0\leq k\leq \tau-1$, then 
    \begin{equation*} 
        \|\mathcal{E}_{s,k}\|\leq \left(1+\sqrt{3\ln\left(\frac{1}{p_s}\right)}\right)\cdot\frac{\sqrt{2\tau\epsilon}\cdot L_{\max}}{\sqrt{b_s}\cdot L\kappa_h^\delta}
    \end{equation*}
    with probability at least $1-p_s$.
\end{lemma}
\begin{proof}
    Fix any epoch index $s\geq1$, consider the sequence $\{\zeta_{k,j}^s\}$ defined as  
    $$\zeta_{k,j}^s = \frac{1}{|\mathcal{B}_{s,k}|}\left(\big(\nabla f_{\xi_{k,j}^s}(x_{s,k})- \nabla f_{\xi_{k,j}^s}(x_{s,k-1})\big) - \big(\nabla f(x_{s,k})-\nabla f(x_{s,k-1})\big)\right).$$
    In the above definition, the index $k$ runs through $1\leq k\leq\tau-1$, and the index $j$ can take value from $1\leq j\leq |\mathcal{B}_{s,k}|$ given each $k$. For each $(k,j)$ in our index range, $\xi_{k,j}^s$ stands for the $j$-th sample from the batch $\mathcal{B}_{s,k}\subseteq[n]$. Then by direct computation, we have  
    $\mathcal{E}_{s,k} = \sum_{k'=1}^k\sum_{j=1}^{|\mathcal{B}_{s,k'}|} \zeta_{k',j}^s$ and $\{\zeta_{k,j}^s\}$ forms a martingale difference sequence if the index $(k,j)$ runs in a lexicographical order. Note that 
    \begin{eqnarray}
        \|\zeta_{k,j}^s\| &\leq& \frac{2L_{\max}L_h([x_{s,k-1},x_{s,k}])\|x_{s,k-1}-x_{s,k}\|}{b_s}\nonumber\\
        &\leq& \frac{2L_{\max}L_h([x_{s,k-1},x_{s,k}])}{b_s}\cdot\frac{\sqrt{\epsilon}/2L(\kappa_h^\delta)^2}{\|\nabla h({x}_{s,k-1})-\nabla h(\bar{x}_{s,k})\|}\cdot\|{x}_{s,k-1}-\bar{x}_{s,k}\|\nonumber\\
        &\leq& \frac{2L_{\max}L_h([x_{s,k-1},x_{s,k}])}{b_s}\cdot\frac{\sqrt{\epsilon}/2L(\kappa_h^\delta)^2}{\mu_h([x_{s,k-1},\bar{x}_{s,k}])\|{x}_{s,k-1}-\bar{x}_{s,k}\|}\cdot\|{x}_{s,k-1}-\bar{x}_{s,k}\|\nonumber\\
        &\leq& \frac{L_{\max}\cdot\sqrt{\epsilon}}{L\kappa_h^\delta\cdot b_s}\,,\nonumber
    \end{eqnarray} 
    where the last inequality is because $\|x_{s,k-1}-\bar{x}_{s,k}\|\leq \delta$, Assumption \ref{assumption:kernel-conditioning}, and the fact that 
    $$\gamma_{s,k}\leq 1\quad\Longrightarrow\quad [x_{s,k-1},x_{s,k}]\subseteq[x_{s,k-1},\bar{x}_{s,k}]\quad\Longrightarrow\quad L_h([x_{s,k-1},x_{s,k}])\leq L_h([x_{s,k-1},\bar{x}_{s,k}]).$$
    As this bound holds almost surely, we have $\EE\left[\exp\left\{\frac{\|\zeta_{k,j}^s\|^2}{(L_{\max}\sqrt{\epsilon}/L\kappa_h^\delta b_s)^2}\right\}\right] \leq \exp\{1\}$. Applying Theorem \ref{theorem:MDS-LargeDev} to this martingale difference sequence gives
    \begin{equation*} 
    \mathrm{Prob}\left(\|\mathcal{E}_{s,k}\| \geq \frac{(1+\gamma)\sqrt{2\tau\epsilon}\cdot L_{\max}}{\sqrt{b_s}\cdot L\kappa_h^\delta}\right)\leq \exp\left\{-\frac{\gamma^2}{3}\right\}.
    \end{equation*}
    Finally, setting $\gamma = \sqrt{3\ln (1/p_s)}$ gives $\exp\{-\gamma^2/3\} = p_s$, which proves the lemma.    
\end{proof}
Let us define $\tcDshesk(x_{s,k}):=\frac{\nabla h(x_{s,k})-\nabla h(\bar{x}_{s,k+1})}{\eta_{s,k}}$ as the stochastic surrogate of the the exact dual gradient mapping $\cDshesk(x_{s,k})$, then we have the following lemma. 
\begin{lemma}
\label{lemma:TBPG-svr-descent-final}
    Let us set $\tau = \lceil\sqrt{n}\rceil$, and $|\mathcal{B}_{s,k}|= b_s = 8\lceil\sqrt{n}\rceil(2+6\ln(1/p_s)) L_{\max}^2/L^2$. 
    For $(s,k)$-th iteration of the update \eqref{defn:TBPG-svr-1} and \eqref{defn:TBPG-svr-2}, as long as $\|\tcDshesk(x_{s,k})\|^2\geq\epsilon$, it holds w.p. at least $1-p_s$ that
    $$\Psi(x_{s,k+1}) \leq \Psi(x_{s,k}) - \frac{\sqrt{\epsilon}\|{x}_{s,k+1}-x_{s,k}\|}{4\kappa_h^\delta}.$$
    Suppose the target accuracy satisfies $\epsilon\leq \max\left\{2L\kappa_h^\delta\mu\delta\rho/3,\rho^2\right\}$, then we also have 
    $$\Psi(x_{s,k+1}) \leq \Psi(x_{s,k}) -  \frac{\min\{1/2\kappa_h^\delta L, \frac{\mu\delta}{3\rho}\}\cdot\epsilon}{4(\kappa_h^\delta)^2 L_h([x_{s,k},\bar{x}_{s,k+1}])}.$$
\end{lemma} 
\begin{proof}
First of all, by Lemma \ref{lemma:TBPG-error}, setting $b_s=8\lceil\sqrt{n}\rceil(2+6\ln(1/p_s)) L_{\max}^2/L^2$ gives 
\begin{equation}
    \label{lm:TBPG-svr-descent-final-1} \mathrm{Prob}\left(\|\mathcal{E}_{s,k}\|\leq\frac{\sqrt{\epsilon}}{2\kappa_h^\delta}\right)\geq 1- p_s\,.
\end{equation} 
Given $\|\tcDshesk(x_{s,k})\|\geq\sqrt{\epsilon}$, we have with probability at least $1-p_s$ that 
\begin{eqnarray*}
    &&\frac{\gamma_{s,k}}{\eta_{s,k}}\cdot D_h(\bar{x}_{s,k+1},x_{s,k}) - \gamma_{s,k}\cdot\|\bar{x}_{s,k+1}-x_{s,k}\|\cdot\|\mathcal{E}_{s,k}\|\\
    &\geq&  \frac{\gamma_{s,k}}{\eta_{s,k}}\cdot \frac{\|\nabla h(x_{s,k})-\nabla h(\bar{x}_{s,k+1})\|^2}{2L_h([x_{s,k},\bar{x}_{s,k+1}])}-\gamma_{s,k}\cdot\|\bar{x}_{s,k+1}-x_{s,k}\|\cdot\frac{\sqrt{\epsilon}}{2\kappa_h^\delta}\\
    &\geq&  \frac{\gamma_{s,k}}{\eta_{s,k}}\cdot \frac{\mu_h([x_{s,k},\bar{x}_{s,k+1}])\cdot\|x_{s,k}-\bar{x}_{s,k+1}\|\cdot\|\nabla h(x_{s,k})-\nabla h(\bar{x}_{s,k+1})\|}{2L_h([x_{s,k},\bar{x}_{s,k+1}])}-\gamma_{s,k}\cdot\|\bar{x}_{s,k+1}-x_{s,k}\|\cdot\frac{\sqrt{\epsilon}}{2\kappa_h^\delta}\\
    & \geq & \gamma_{s,k}\|\bar{x}_{s,k+1}-x_{s,k}\|\cdot\left(\frac{\|\tcDshesk(x_{s,k})\|}{2\kappa_h^\delta}-\frac{\sqrt{\epsilon}}{2\kappa_h^\delta}\right)\\
    &\geq& 0\,.
\end{eqnarray*} 
Combining the above inequality with Lemma \ref{lemma:TBPG-svr-descent}, we have 
\begin{eqnarray}
    \label{lm:TBPG-svr-descent-final-2}
    \Psi(x_{s,k+1}) &\leq& \Psi(x_{s,k}) - \frac{\gamma_{s,k}}{2\eta_{s,k}}\cdot D_h(\bar{x}_{s,k+1},x_{s,k})\nonumber\\
    &\leq& \Psi(x_{s,k}) - \gamma_{s,k}\|\bar{x}_{s,k+1}-x_{s,k}\|\cdot\frac{\|\tcDshesk(x_{s,k})\|}{4\kappa_h^\delta}\\
    &=& \Psi(x_{s,k}) - \frac{\|{x}_{s,k+1}-x_{s,k}\|\cdot\|\tcDshesk(x_{s,k})\|}{4\kappa_h^\delta}\nonumber\\
    &\leq& \Psi(x_{s,k}) - \frac{\sqrt{\epsilon}\|{x}_{s,k+1}-x_{s,k}\|}{4\kappa_h^\delta}\,.\nonumber
\end{eqnarray}
This proves the first inequality of the lemma. Next, let us prove the $\Psi(x_{s,k+1}) \leq \Psi(x_{s,k}) - \Omega(\epsilon)$ result by discussing the following cases: \vspace{0.1cm}

\noindent\textbf{case 1.} 
When $\gamma_{s,k} = \frac{\sqrt{\epsilon}/2L(\kappa_h^\delta)^2}{\|\nabla h(x_{s,k}) -\nabla h(\bar{x}_{s,k+1})\|}$, regardless of $\eta_{s,k}$, the second row of \eqref{lm:TBPG-svr-descent-final-2} indicates that  
\begin{eqnarray*}
    \Psi(x_{s,k+1}) &\leq& \Psi(x_{s,k}) - \frac{\sqrt{\epsilon}\|\bar{x}_{s,k+1}-x_{s,k}\|/2L(\kappa_h^\delta)^2}{\|\nabla h(x_{s,k}) -\nabla h(\bar{x}_{s,k+1})\|}\cdot\frac{\|\tcDshesk(x_{s,k})\|}{4\kappa_h^\delta}\\
    &\leq & \Psi(x_{s,k}) - \frac{\sqrt{\epsilon}\cdot\|\tcDshesk(x_{s,k})\|}{8L(\kappa_h^\delta)^3L_h([x_{s,k},\bar{x}_{s,k+1}])}\\
    & \leq & \Psi(x_{s,k}) - \frac{\epsilon}{8L(\kappa_h^\delta)^3L_h([x_{s,k},\bar{x}_{s,k+1}])}.
\end{eqnarray*}

\noindent\textbf{case 2.} If $\gamma_{s,k} = 1$ and $\eta_{s,k} = \frac{\mu\delta}{\|v_{s,k}\|+\rho}$. This case may happen only if $\frac{\mu\delta}{\|v_{s,k}\|+\rho}\leq \frac{\mu\delta}{3\rho}$, namely, only if $\|v_{s,k}\|\geq2\rho.$ Note that $\tcDshesk(x_{s,k}) = v_{s,k}+u_{s,k+1}$ for some $u_{s,k+1}\in\partial \phi(\bar{x}_{s,k+1})$, in this case, we have $\|\tcDshesk(x_{s,k})\|\geq\|v_{s,k}\|-\rho\geq\rho$. Then the second row of \eqref{lm:TBPG-svr-descent-final-2} gives  
\begin{eqnarray}
\label{lm:TBPG-svr-descent-final-3}
    \Psi(x_{s,k+1}) &\leq& \Psi(x_{s,k}) - 1\cdot\|\bar{x}_{s,k+1}-x_{s,k}\|\cdot\frac{\|\tcDshesk(x_{s,k})\|}{4\kappa_h^\delta}\nonumber\\
    &\leq & \Psi(x_{s,k}) - \frac{\|\nabla h(x_{s,k})-\nabla h(\bar{x}_{s,k+1})\|}{L_h([x_{s,k},\bar{x}_{s,k+1}])}\cdot\frac{\|\tcDshesk(x_{s,k})\|}{4\kappa_h^\delta}\nonumber\\
    & = & \Psi(x_{s,k}) - \frac{\|\tcDshesk(x_{s,k})\|}{4\kappa_h^\delta L_h([x_{s,k},\bar{x}_{s,k+1}])}\cdot\eta_{s,k}\|\tcDshesk(x_{s,k})\|\\
    & = & \Psi(x_{s,k}) - \frac{\|\tcDshesk(x_{s,k})\|}{4\kappa_h^\delta L_h([x_{s,k},\bar{x}_{s,k+1}])}\cdot\frac{\mu\delta\|v_{s,k}+u_{s,k+1}\|}{\|v_{s,k}\|+\rho}\nonumber\\
    &\leq & \Psi(x_{s,k}) - \frac{\mu\delta\rho}{12\kappa_h^\delta L_h([x_{s,k},\bar{x}_{s,k+1}])}\,.\nonumber
\end{eqnarray}

\noindent\textbf{case 3.} If $\gamma_{s,k} = 1$ and $\eta_{s,k} = \min\big\{\frac{1}{2\kappa_h^\delta L}, \frac{\mu\delta}{3\rho}\big\}$. In this case, the third row of \eqref{lm:TBPG-svr-descent-final-3} indicates that  
\begin{eqnarray*}
    \Psi(x_{s,k+1}) &\leq& \Psi(x_{s,k}) - \frac{\|\tcDshesk(x_{s,k})\|}{4\kappa_h^\delta L_h([x_{s,k}),\bar{x}_{s,k+1}])}\cdot\eta_{s,k}\|\tcDshesk(x_{s,k})\|\\
    & \leq & \Psi(x_{s,k}) - \frac{\min\Big\{\frac{1}{2\kappa_h^\delta L}, \frac{\mu\delta}{3\rho}\Big\}}{4\kappa_h^\delta L_h([x_{s,k}),\bar{x}_{s,k+1}])}\cdot\epsilon\,.
\end{eqnarray*}
Combining cases 1,2, and 3, we know the least descent among the three cases are guaranteed to be achieved. Note that if $\epsilon\leq \max\left\{2L\kappa_h^\delta\mu\delta\rho/3\,,\rho^2\right\}$, direct computation shows that  
$\frac{\min\{1/2\kappa_h^\delta L\,,\, \frac{\mu\delta}{3\rho}\}\cdot\epsilon}{4(\kappa_h^\delta)^2 L_h([x_{s,k},\bar{x}_{s,k+1}])}$ lower bounds the descents in all three cases, which completes the proof.  
\end{proof}
Consequently, define $(S_\epsilon,K_\epsilon):=\argmin_{s,k}\big\{(s-1)\tau+k: \|\tcDshesk(x_{s,k})\|^2\leq \epsilon\big\}$ as the first time that we find a point $\|\tcDshesk(x_{s,k})\|^2\leq \epsilon$, and set 
$R_\epsilon : = \max_{s,k}\big\{\|x_{s,k}-x_{1,0}\|:(s-1)\tau+k\leq (S_\epsilon-1)\tau+K_\epsilon\big\}$ as the maximum traveling distance until finding such a point. Then the following theorem holds while the proof is omitted.
\begin{theorem}
    \label{theorem:TBPG-svr}
    Let us set the parameters according to Lemma  \ref{lemma:TBPG-error} and \ref{lemma:TBPG-svr-descent-final}. Consider the target accuracy $\epsilon\leq \max\left\{2L\kappa_h^\delta\mu\delta\rho/3,\rho^2\right\}$, let $\cX_\epsilon := \mathrm{Lev}_0\cap B(x_{1,0},R_\epsilon)$ be a compact set, then with probability at least $1-q$, all the following arguments hold:    $$R_\epsilon\leq\frac{4\kappa_h^\delta\Delta_\Psi}{\sqrt{\epsilon}}\qquad\mbox{and}\qquad S_\epsilon \leq \frac{\max\{8\kappa_h^\delta L,\frac{12\rho}{\mu\delta}\}\cdot (\kappa_h^\delta)^2 L_h(\cX_s)\cdot\Delta_\Psi}{\epsilon \tau}+1.$$
    Moreover, we have $\|\mathcal{D}_{\phi,h}^{\eta_{S_\epsilon,K_\epsilon}} (x_{S_\epsilon,K_\epsilon})\|^2\leq 2.5\epsilon$. The total samples consumed is $\sum_{s=1}^{S_\epsilon}n+\tau b_s = \widetilde{O}(\sqrt{n}/\epsilon)$, where $\widetilde{O}(\cdot)$ hides the poly-logarithmic terms in $q$ and $\epsilon$.
\end{theorem}
\begin{proof}
    The bounds on $R_\epsilon$ and $S_\epsilon$ are straightforward consequence of Lemma \ref{lemma:TBPG-svr-descent-final}. We only need to show the bound of exact Bregman proximal gradient mapping. For notational simplicity, let us denote $(s,k)=S_\epsilon,K_\epsilon$. Then by definition, we have $\|\tcDshesk(x_{s,k})\|^2\leq \epsilon$. Let $\hat{x}_{s,k+1}:=\Tshesk(x_{s,k},\nabla f(x_{s,k}))$ be the ideal intermediate update point that uses the exact gradient $\nabla f(x_{s,k})$, hence the exact gradient mapping will be $\cDshesk(x_{s,k}) = \frac{\nabla h(x_{s,k})-\nabla h(\hat{x}_{s,k+1})}{\eta_{s,k}}$. By the proof of Lemma \ref{lemma:restricted-grd-mapping} and \eqref{lm:TBPG-svr-descent-final-1}, we have 
    $$\|\hat{x}_{s,k+1}-\bar{x}_{s,k+1}\|\leq\frac{\eta_{s,k}\|\mathcal{E}_{s,k}\|}{\mu_h(B(x_{s,k},\delta))}\leq \frac{\eta_{s,k}\sqrt{\epsilon}}{2\kappa_h^\delta\cdot\mu_h(B(x_{s,k},\delta))}.$$ Consequently, with $\tcDshesk(x_{s,k}) = \frac{\nabla h(x_{s,k})-\nabla h(\bar{x}_{s,k+1})}{\eta_{s,k}}$,  we further obtain that 
    \begin{eqnarray}
        \big\|\cDshesk(x_{s,k})-\tcDshesk(x_{s,k})\big\|   =   \Big\|\frac{\nabla h(x_{s,k})-\nabla h(\bar{x}_{s,k+1})}{\eta_{s,k}} -\frac{\nabla h(x_{s,k})-\nabla h(\hat{x}_{s,k+1})}{\eta_{s,k}} \Big\|\leq \sqrt{\epsilon}/2.\nonumber
    \end{eqnarray} 
    Using the fact that $\|a+b\|^2\leq 2(\|a\|^2+\|b\|^2)$, we finish the proof of $\|\cDshesk(x_{s,k})\|^2\leq2.5\epsilon$.
\end{proof}

\section{Numerical experiments}
In this section, we present some preliminary experiment on the (sparse) quadratic inverse problem studied in \cite{bolte2018first}. In particular, we consider the phase retrieval setting where we try to recover an unknown signal vector $x_\mathrm{true}$ from a bunch of quadratic measurements. Given a group of sampling vector $a_i\in\mathbb{R}^d$, we can take a noisy quadratic measurement and observe a scalar $b_i^2$ such that $|\langle a_i,x_{\mathrm{ture}} \rangle|^2 \approx b_i^2$. Suppose the noise is additive and Gaussian, then we can consider the following formulation:
\begin{equation}
    \label{prob:quad-inv}
    \min_{x\in\mathbb{R}^d}\,\, \Psi(x):=\frac{1}{N}\sum_{i=1}^N\left(|\langle a_i,x\rangle|^2-b_i^2\right)^2 + \sigma\|x\|_1.
\end{equation}
By \cite[Lemma 5.1]{bolte2018first}, the differentiable component of $\Psi(x)$ is $L$-smooth adaptable to the quartic polynomial kernel $h(x) = \frac{1}{2}\|x\|^2 + \frac{1}{4}\|x\|^4$, with the constant $L = \frac{1}{N}\sum_{i=1}^N(3\|a_i\|^4+b_i^2\|a_i\|^2)$.  \vspace{0.2cm}

\textbf{Dataset preparation}.\,\, In the experiments, we considered two datasets. 
The first is a set of popular signal processing test images, including Lena, Barbara, Peppers, and Baboon. We unify these images to $64\times64$ pixels and reshape them to vectors in $\RR^{d}$ with $d = 4096$. 
We normalize elements of $x_{\mathrm{true}}$ to $[0,1]$ by setting $x_{\mathrm{true}}\leftarrow \frac{x_{\mathrm{true}}}{\|x_{\mathrm{true}}\|_\infty}$.
Each sampling vector $a_i\in\mathbb{R}^{4096}$ are taken as Gaussian random vectors with each element generated from $\mathcal{N}(0,1)$, and an additive Gaussian noise from $\mathcal{N}(0,0.05)$ is added to each measurement. To achieve robust recovery of the signal, we set the total number of sampling vectors to be $N = 4 d$. In particular, because these test images are not sparse, we will set $\sigma = 0$ and measure convergence by $\|\nabla\Psi(\cdot)\|^2$. In order to test the nonsmooth case, we pick a few $28\times28$  images from the MNIST dataset, and pad their margin with zeros to make them of size $36\times36$. Then similar to the first dataset, we vectorize the images and take measure with Gaussian random vector from $\mathcal{N}(0,1)$. For these $x_{\mathrm{true}}\in\RR^{d}$ with $d = 1296$, the number of nonzero pixels $k$ are typically between 100 and 200. Therefore, for robust recovery, the total number of sampling is set to $N = \lceil4 k\ln d\rceil$ with $k = 200$. For each measurement, we still put an additive noise sampled from $\mathcal{N}(0,0.05)$. We set the  regularization coefficient $\sigma = 0.001$ and measure the convergence by $\mathrm{dist}^2(0,\partial\Psi(\cdot))$ and function value gap versus the total number of samples consumed. \vspace{0.2cm}

\textbf{Algorithmic setup.}\,\, In the experiments, we will test Algorithm \ref{algorithm:proxSARAH} abbreviated as SVRBPG-EB, where ``SVR'' stands for ``stochastic variance reduced'' and ``EB'' stands for ``epoch bounds''. For its adaptive step size variant described by \eqref{defn:TBPG-svr-1} and \eqref{defn:TBPG-svr-2}, we denote it as SVRBPG-AS where ``AS'' stands for ``adaptive step''. For both algorithms, we set the batch size to be $b = 100$ and the epoch length to be  $\tau = \lceil2N/b\rceil$. For SVRBPG-EB, the parameter $\gamma$ and $\eta$ are chosen according to Theorem \ref{theorem:proxSARAH-finite}. The parameter $\kappa_h^\delta$ and $\delta$ are chosen according to Proposition \ref{proposition: Poly-kernel-condition}. In particular, this proposition allows us to adaptively set $\cX_s = B(x_{s,0},\max\left\{ {1}/{4},{\|x_{s,0}\|}/{5}\right\})$, see Line 4 of Algorithm \ref{algorithm:proxSARAH}. For SVRBPG-AS, the parameter $\eta_{s,k}$ and $\gamma_{s,k}$ are chosen according to \eqref{defn:TBPG-svr-1} and \eqref{defn:TBPG-svr-2}. In particular, we utilize the special property of polynomial kernel in Proposition \ref{proposition: Poly-kernel-condition} and set $\delta = \max\left\{ {1}/{4},{\|x_{s,0}\|}/{5}\right\}$ and $\mu = \mu_h(B(x_{s,0},\delta))$ for each epoch $s$ as the adaptive step sizes provably restrict the iterates within this region. Finally, for the constant $L$, we notice that the $L$ estimate provided by \cite[Lemma 5.1]{bolte2018first} is way too conservative. For example, for the Lena data, the estimate of \cite[Lemma 5.1]{bolte2018first} gives $L \approx 3.9\times 10^8$. Therefore, for both SVRBPG-EB and SVRBPG-AS, the constant $L$ remains a tuning parameter, and from $L=\{10^0,10^1,\cdots,10^8\}$ we find $L=10$ works best.  For comparison, we will test the stochastic BPG (SBPG) method studied in \cite{davis2018stochastic,zhang2018convergence,ding2023nonconvex}, and the momentum stochastic BPG with  (MSBPG) studied in \cite{ding2023nonconvex}. For both SBPG and MSBPG, we still choose the batch size to be $b = 100$. For the step size, we slightly change the suggestion of \cite{ding2023nonconvex} from $\eta_t = \max\left\{10^{-4},\frac{a}{\sqrt{t+1}}\right\}$ to $\eta_t = \max\left\{10^{-4},\frac{1}{a+b\sqrt{t}}\right\}$ as the original step size rule does not work very well in our setting even after parameter tuning. For SBPG and MSBPG, we tune the step size by choosing $a,b\in\{10^0,10^1,\cdots,10^4\}$ and we find $\eta_t = \max\left\{10^{-4},\frac{1}{10^3+10\sqrt{t}}\right\}$ works best for the first dataset, and $\eta_t = \max\left\{10^{-4},\frac{1}{10^2+10\sqrt{t}}\right\}$ works best for the second dataset. For MSBPG, the momentum parameter is tuned from $\{0.05,0.1\}$, and we find $\beta = 0.05$ works best for MSBPG. Finally, to illustrate the general advantage of Bregman-type methods to automatically adjust to a problem's local geometry, we also add  SARAH \cite{pham2020proxsarah} and STORM \cite{cutkosky2019momentum}, two variance reduced non-Bregman first-order methods, to our benchmarks. Both of them achieve the state-of-the-art complexity under the classic L-smooth setting. For SARAH, the batch sizes remain the same $b=100$, while its stepsize is set to $1/L$ with $L$ being a tuning parameter. STORM is a momentum-type variance reduced gradient method with an Adam-style adaptive stepsize. It has two related parameters $L$ and $G$. As we are considering a quartic polynomial problem, we set $G=L^{1.5}$ while let $L$ to be tuned. For both SARAH and STORM, parameter $L$ is tuned from $\{10^0,10^1,\cdots,10^8\}$. However, we do not find a uniformly best $L$ for all instances, and thus different tuned parameters are used for each task, respectively.\vspace{0.2cm}

\textbf{Experimental results.}\,\, Following the above parameter selection, we present the preliminary numerical results for the test instances in Figure \ref{fig:quad-inv} and Figure \ref{fig:quad-inv-sparse}. 
\begin{figure}[H]
    \centering
    \includegraphics[width=0.24\linewidth]{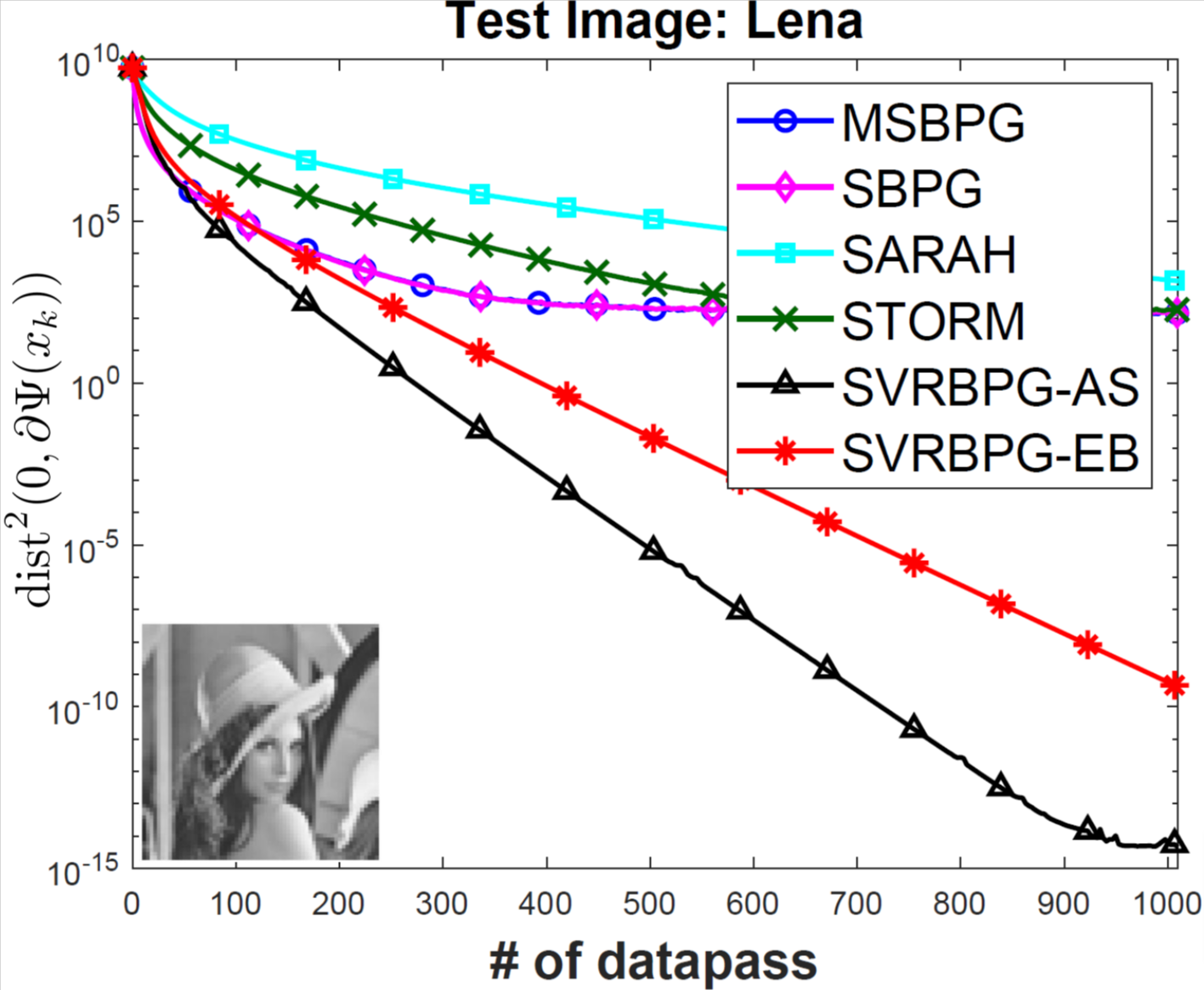}
    \includegraphics[width=0.24\linewidth]{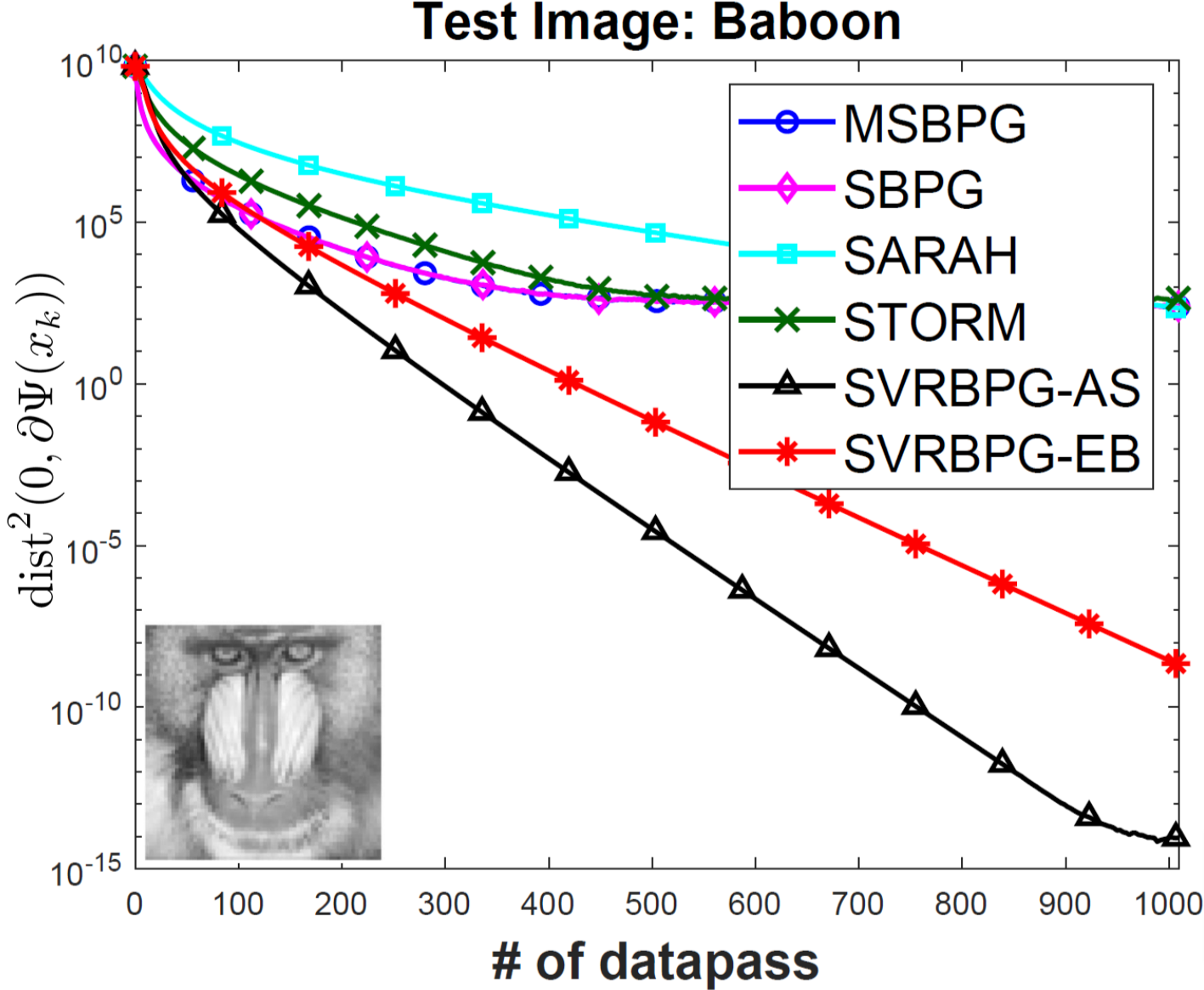}
    \includegraphics[width=0.24\linewidth]{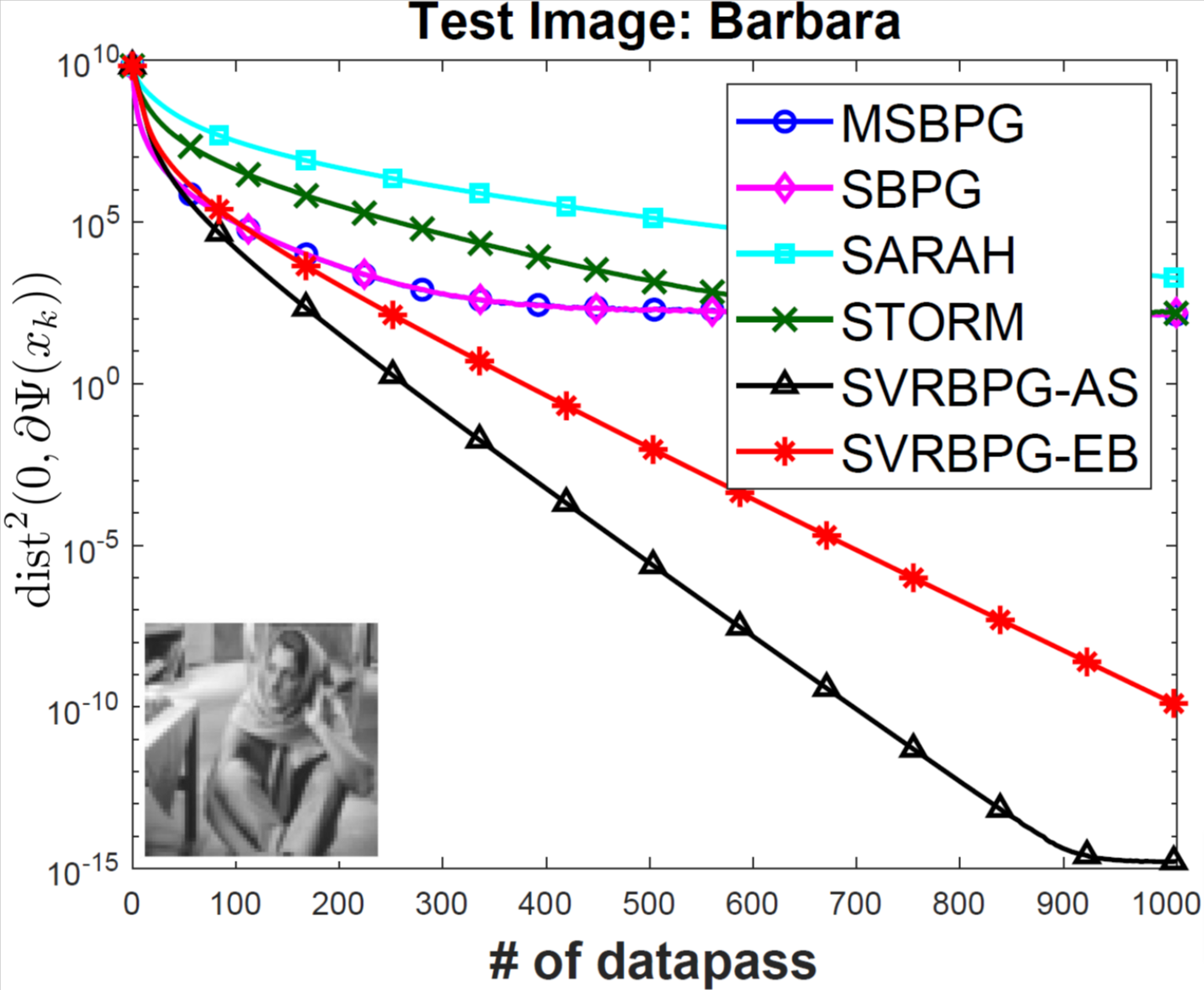}
    \includegraphics[width=0.24\linewidth]{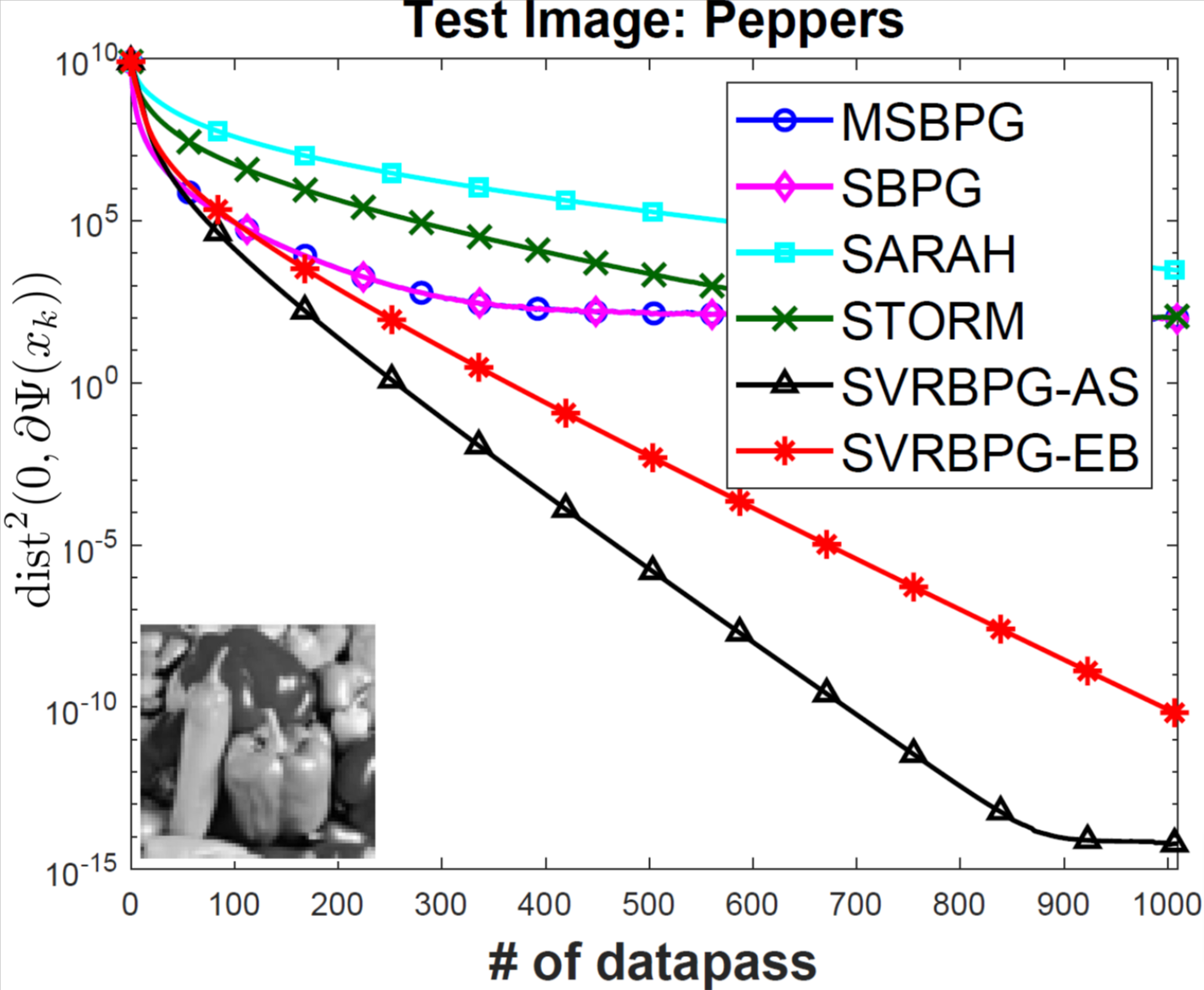}
    \vspace{0.2cm}\\
    \includegraphics[width=0.24\linewidth]{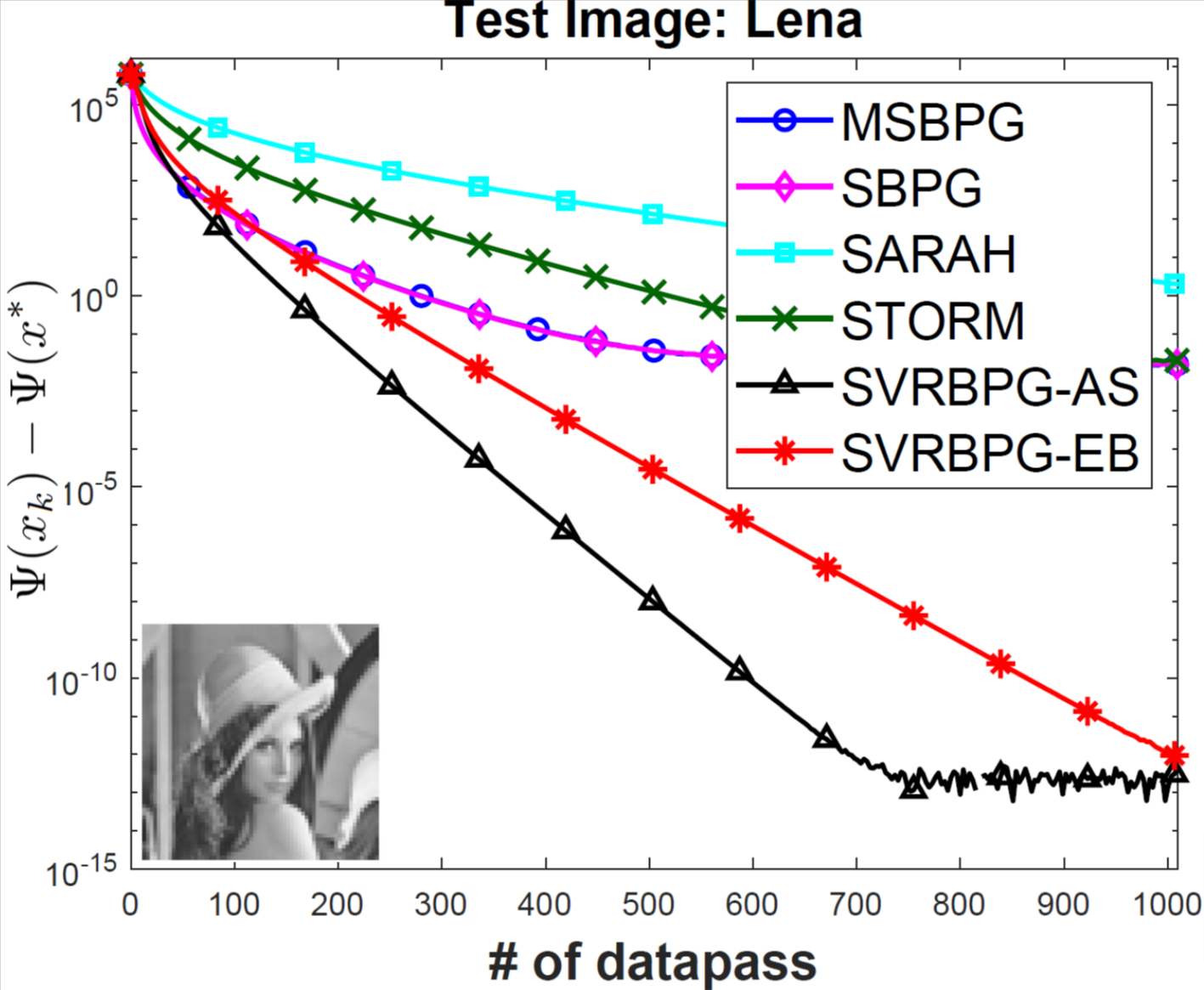}
    \includegraphics[width=0.24\linewidth]{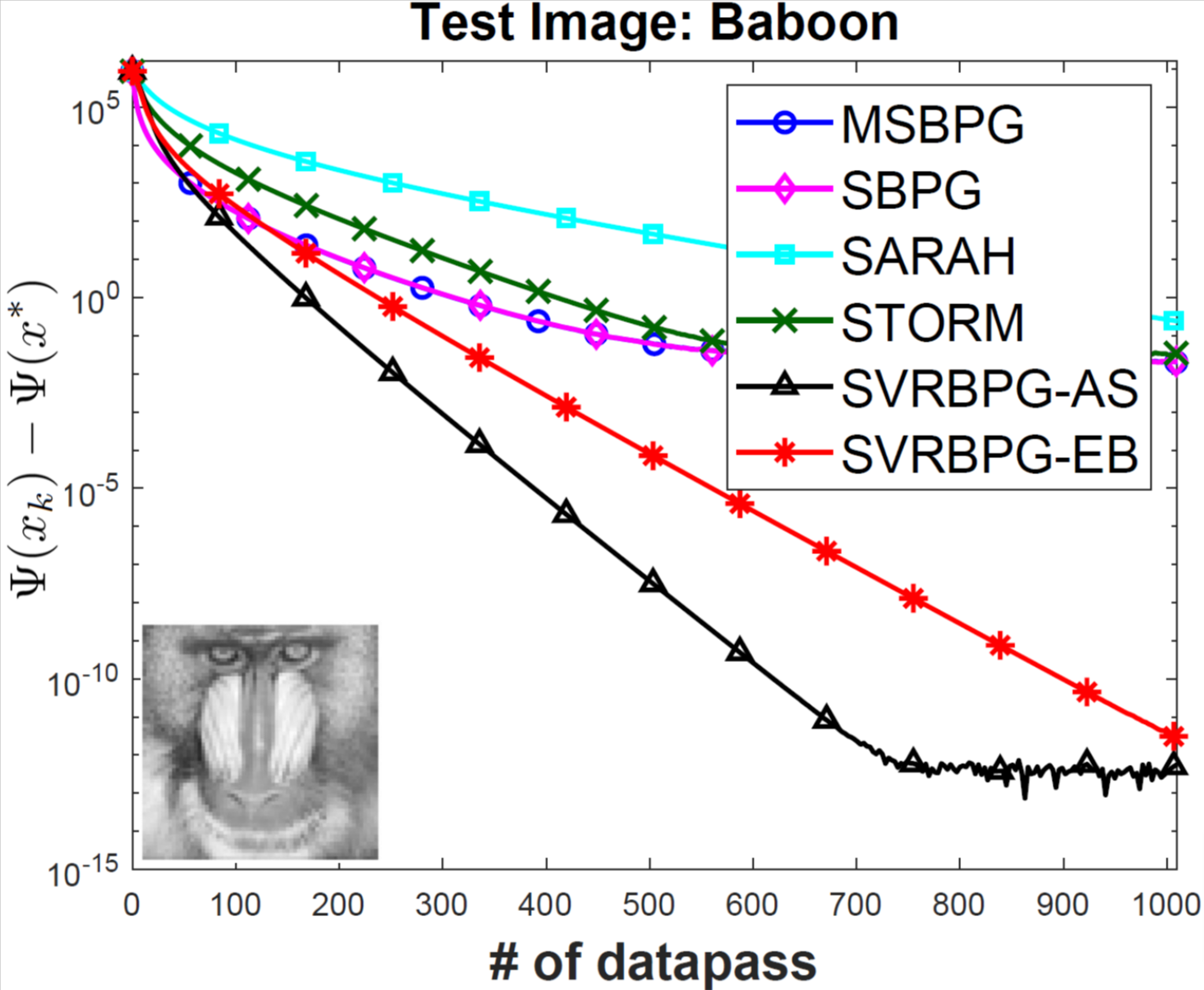}
    \includegraphics[width=0.24\linewidth]{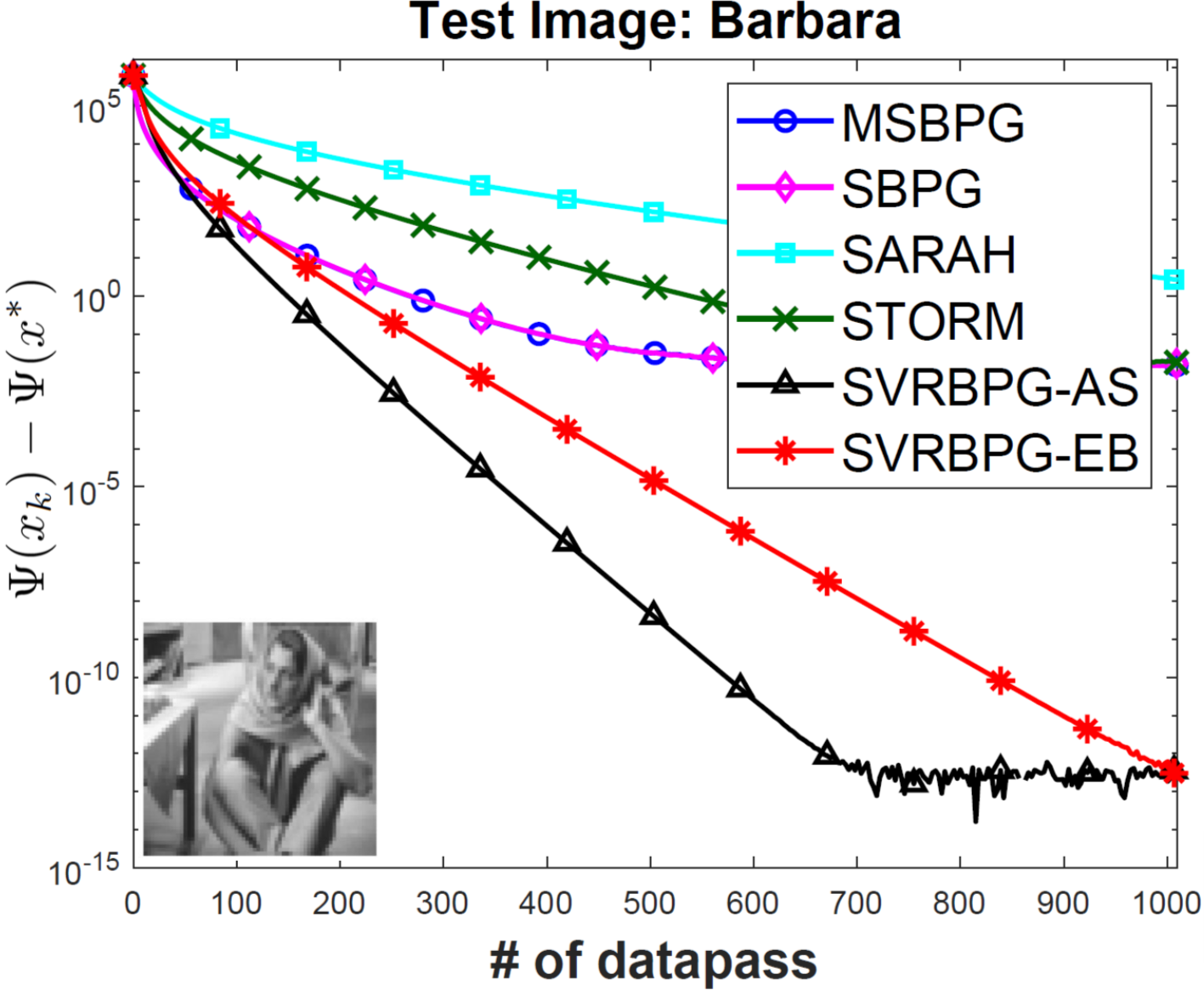}
    \includegraphics[width=0.24\linewidth]{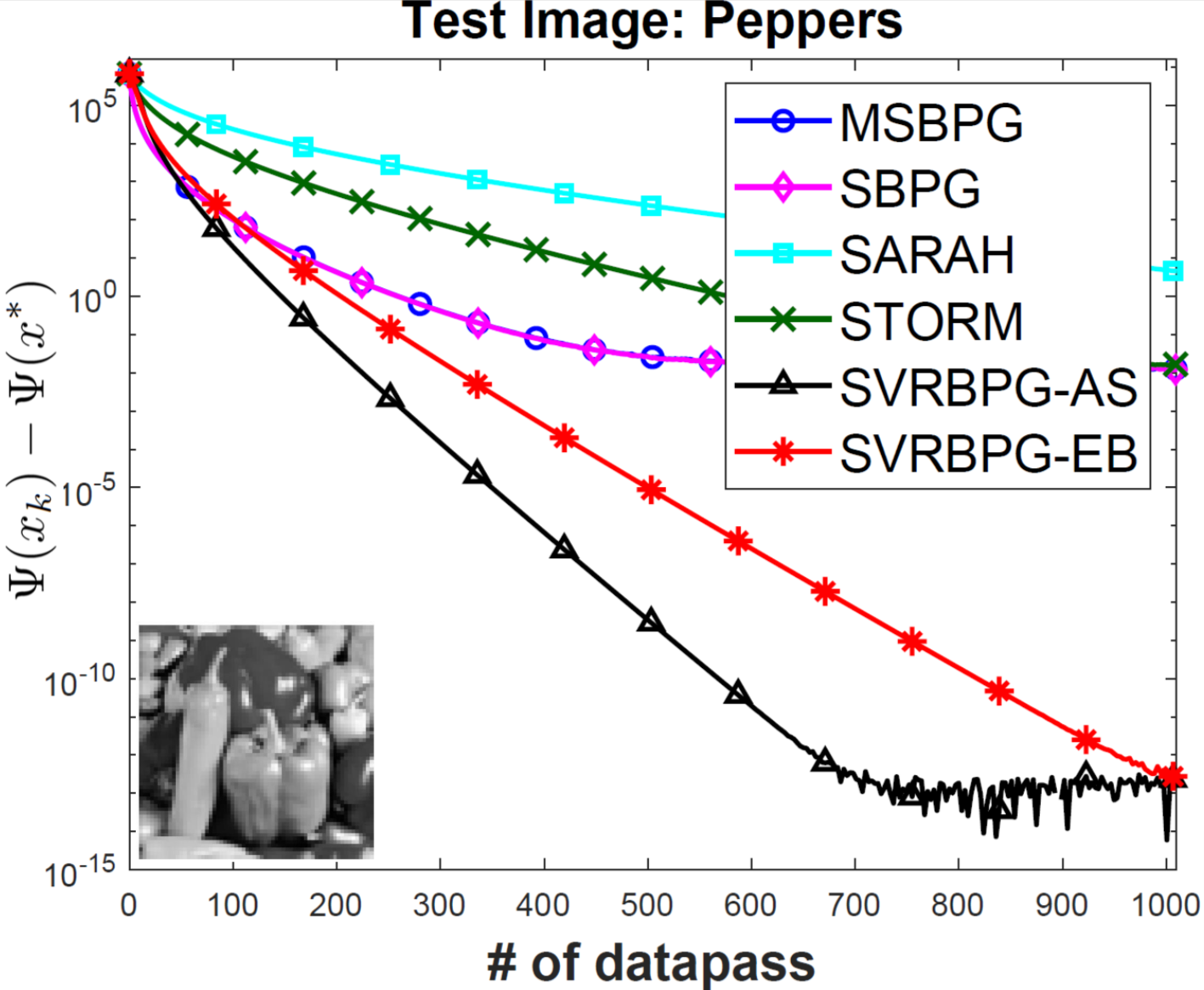}
    \caption{Experimental results for quadratic reverse problem. The raw signal vector $x_{\mathrm{true}}$ is attached at the bottom left corner of each subfigure. The ``\# of datapass'' in x-axis refers to $\frac{\#\mbox{samples  consumed}}{\mbox{full batch size } N}$.  }
    \label{fig:quad-inv}
\end{figure}

In Figure \ref{fig:quad-inv}, we present the differentiable case of problem \eqref{prob:quad-inv}. In this setting, the MSBPG and SBPG behaves very similarly and both of them are slower than the proposed two variants of stochastic variance reduced BPG.  In particular, for the subproblem of SVRBPG-EB, we adopt the heuristic that first ignores the constraint $x\in\cX_s$, if the resulting solution exits $\cX_s$, we project it to $\cX_s$ and use it as a warm start and run projected gradient method (PGM) for 25 iterations. By our record,  only 1.19\%, 1.46\%, 1.19\%, and 1.19\% iterations requires running an extra PGM for Lena, Peppers, Barbara, and Baboon, respectively. Moreover, all such cases happens in the first 3 epochs, which is very intuitive and as only early stages of the algorithm allows aggressive steps. Moreover, the early stop epoch (Line 9 of Algorithm \ref{algorithm:proxSARAH}) does not happen for all 4 cases. For the adaptive step size variant SVRBPG-AS, it shares a comparable performance of SVRBPG-EB while having easier subproblems, which is a desirable feature. Compared to the above Bregman-based first-order algorithms, the non-Bregman variance-reduced methods SARAH and STORM are not behaving very well. Possibly due to their inability to adapt to the varying local landscape of the tested instances, they behave slower than both MSBPG and the vanilla SBPG methods. 

\begin{figure}[H]
    \centering 
    \includegraphics[width=0.24\linewidth]{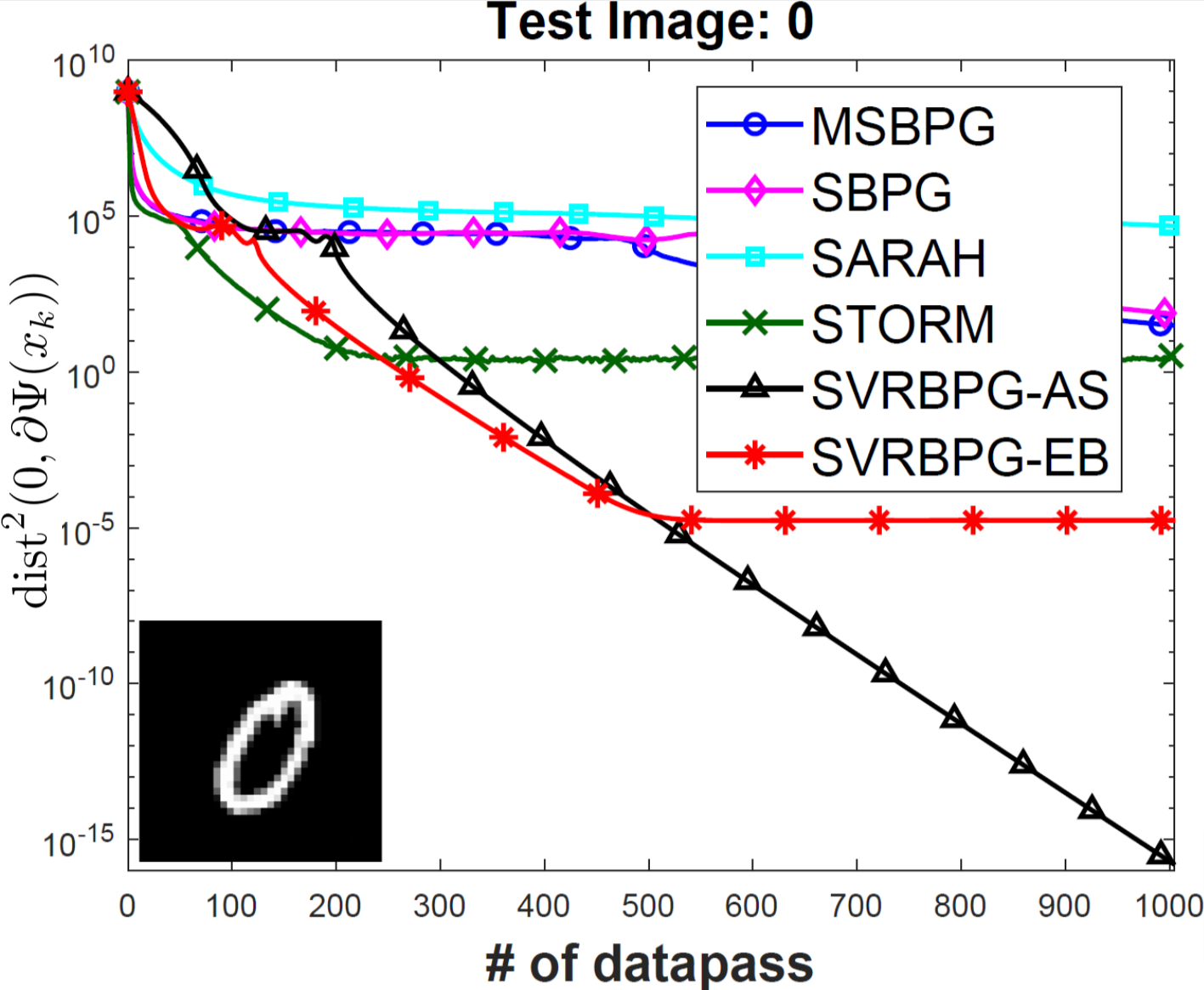}
    \includegraphics[width=0.24\linewidth]{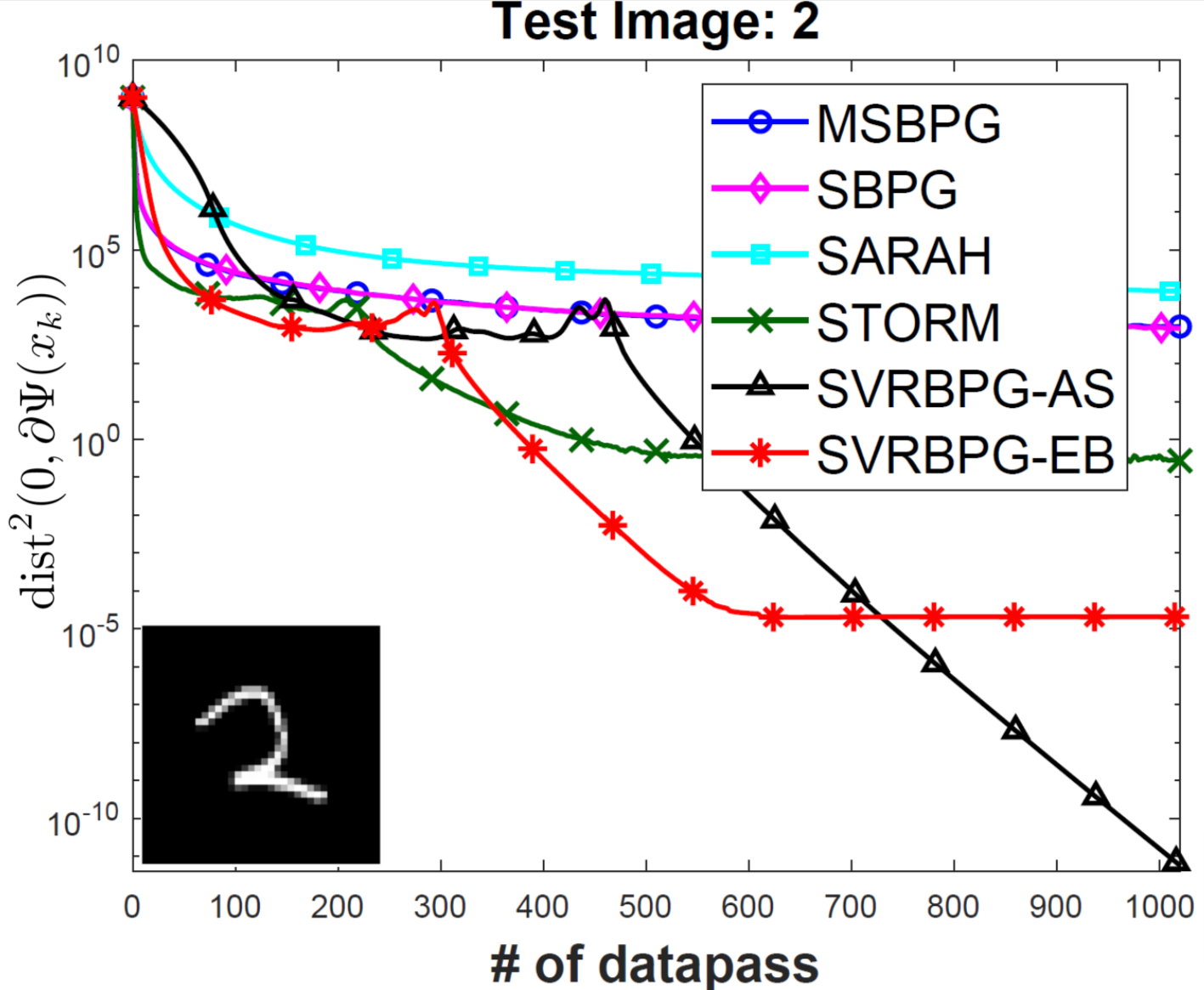}
    \includegraphics[width=0.24\linewidth]{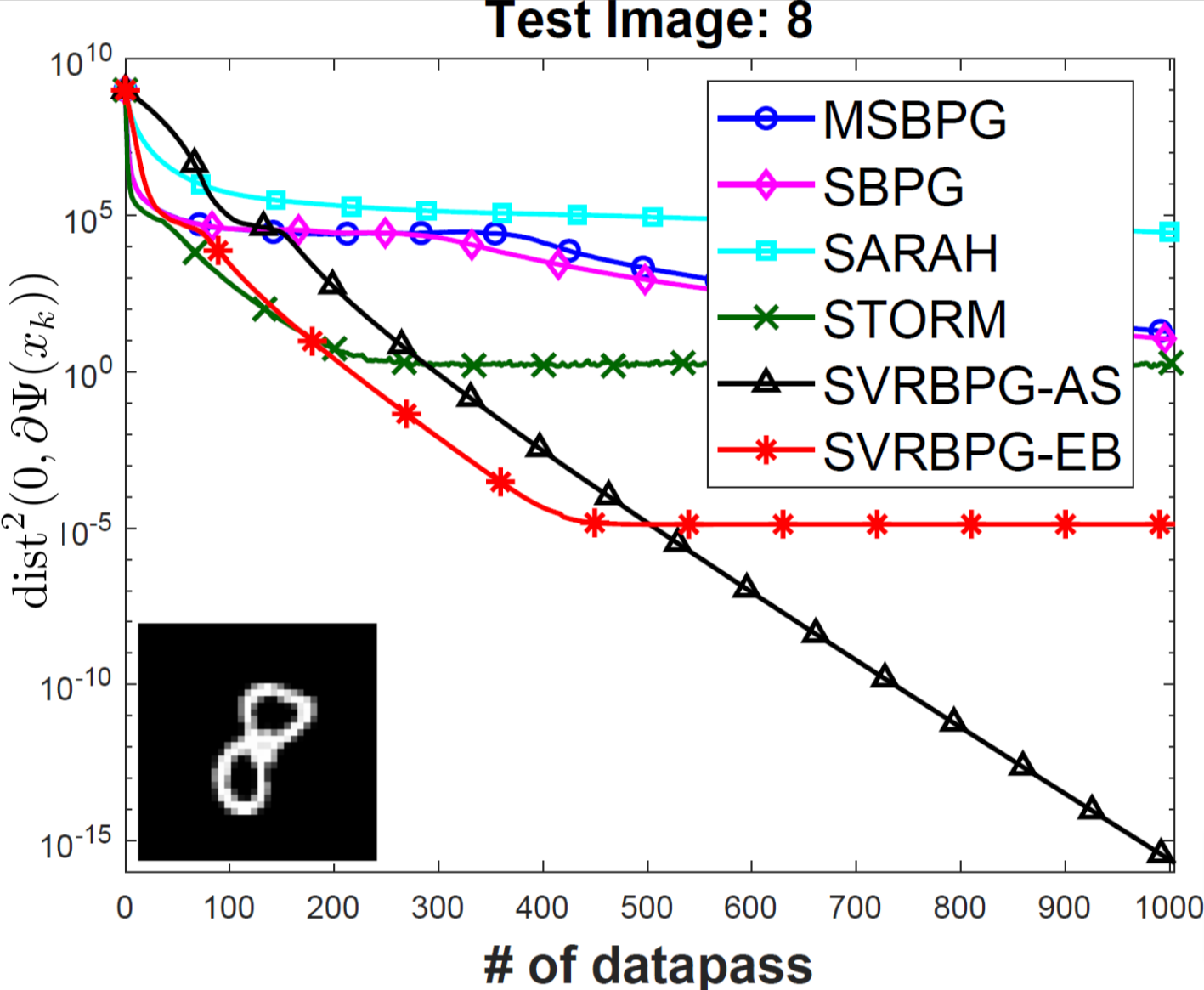}
    \includegraphics[width=0.24\linewidth]{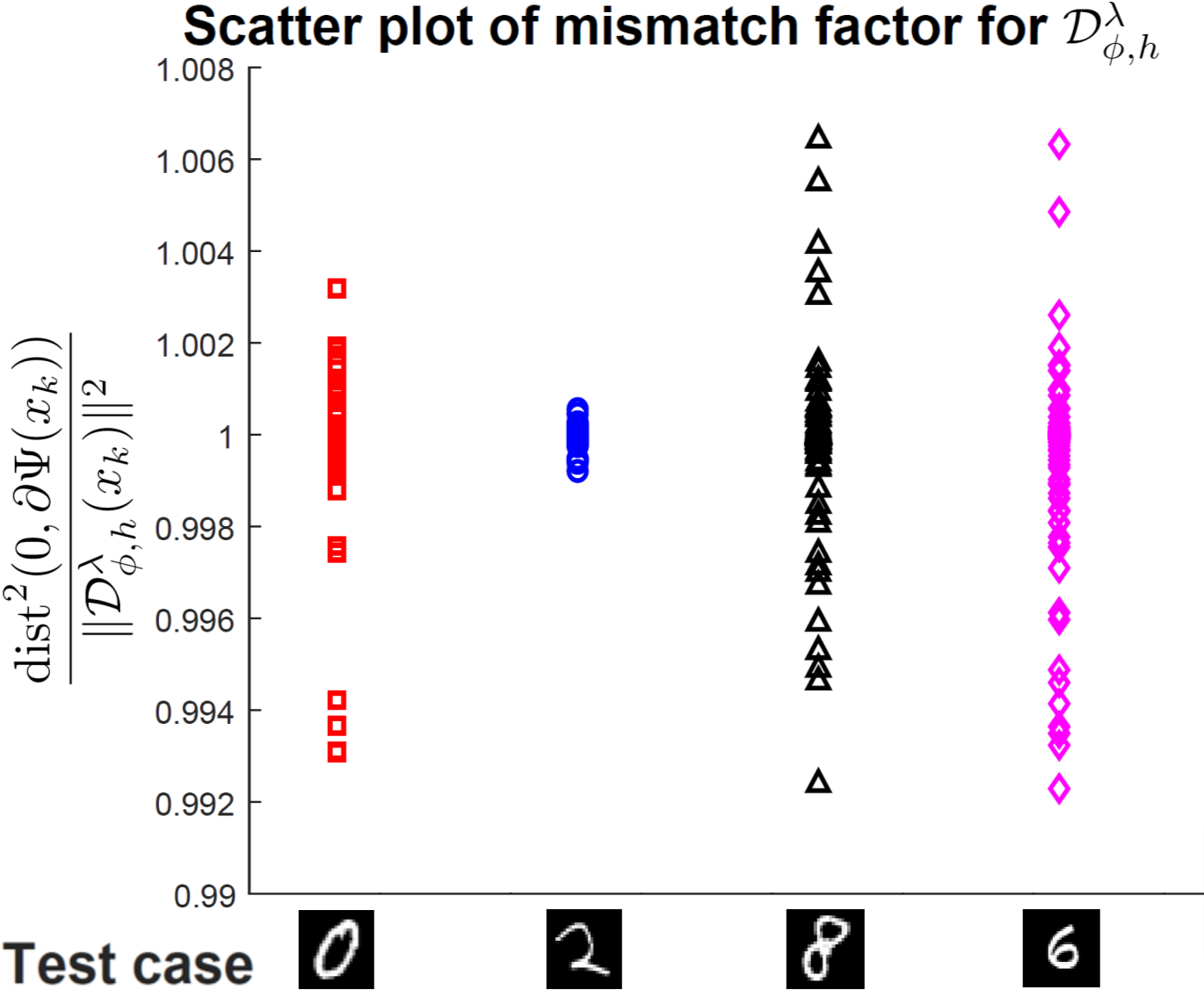}
    \vspace{0.2cm}\\
    \includegraphics[width=0.24\linewidth]{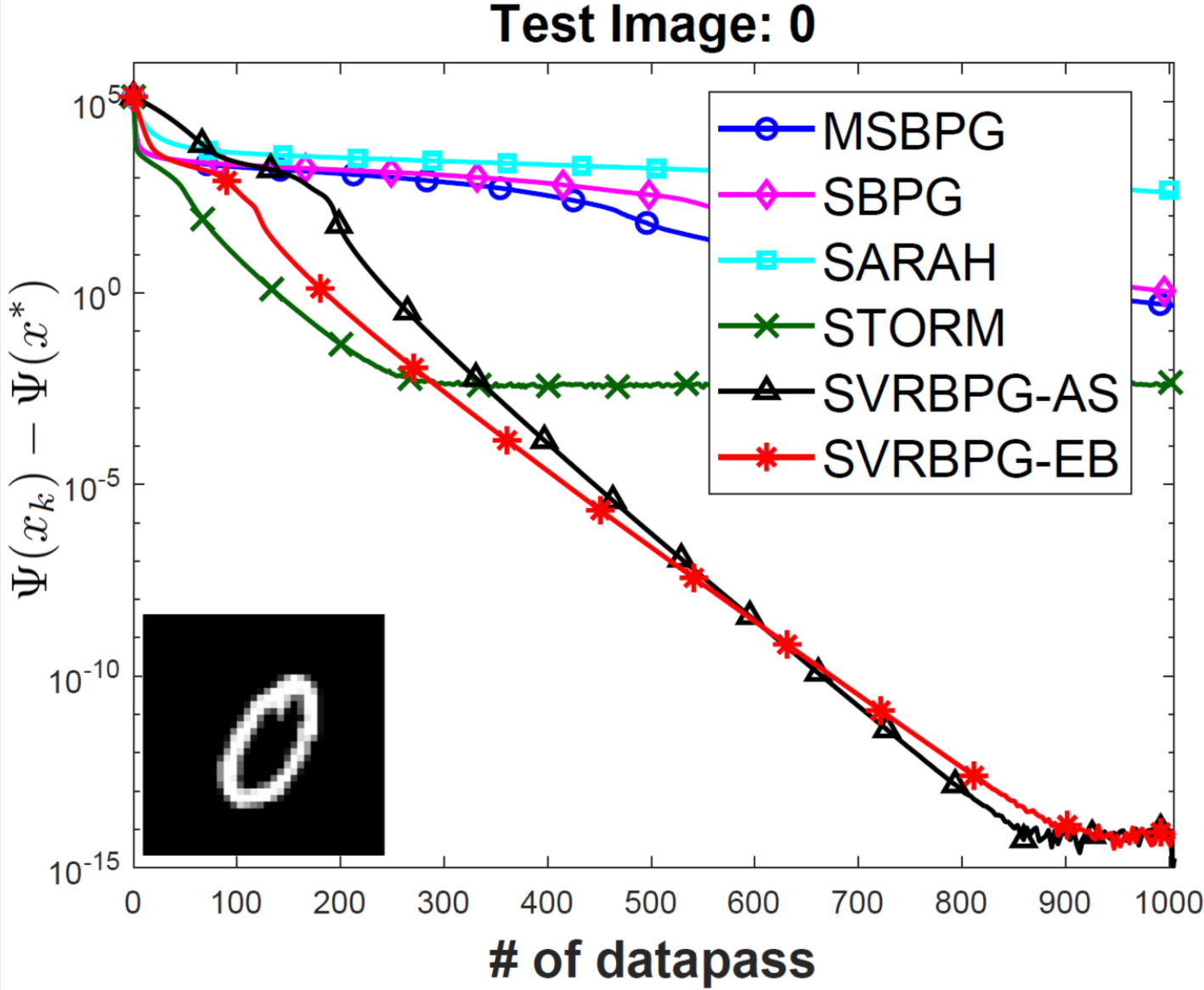}
    \includegraphics[width=0.24\linewidth]{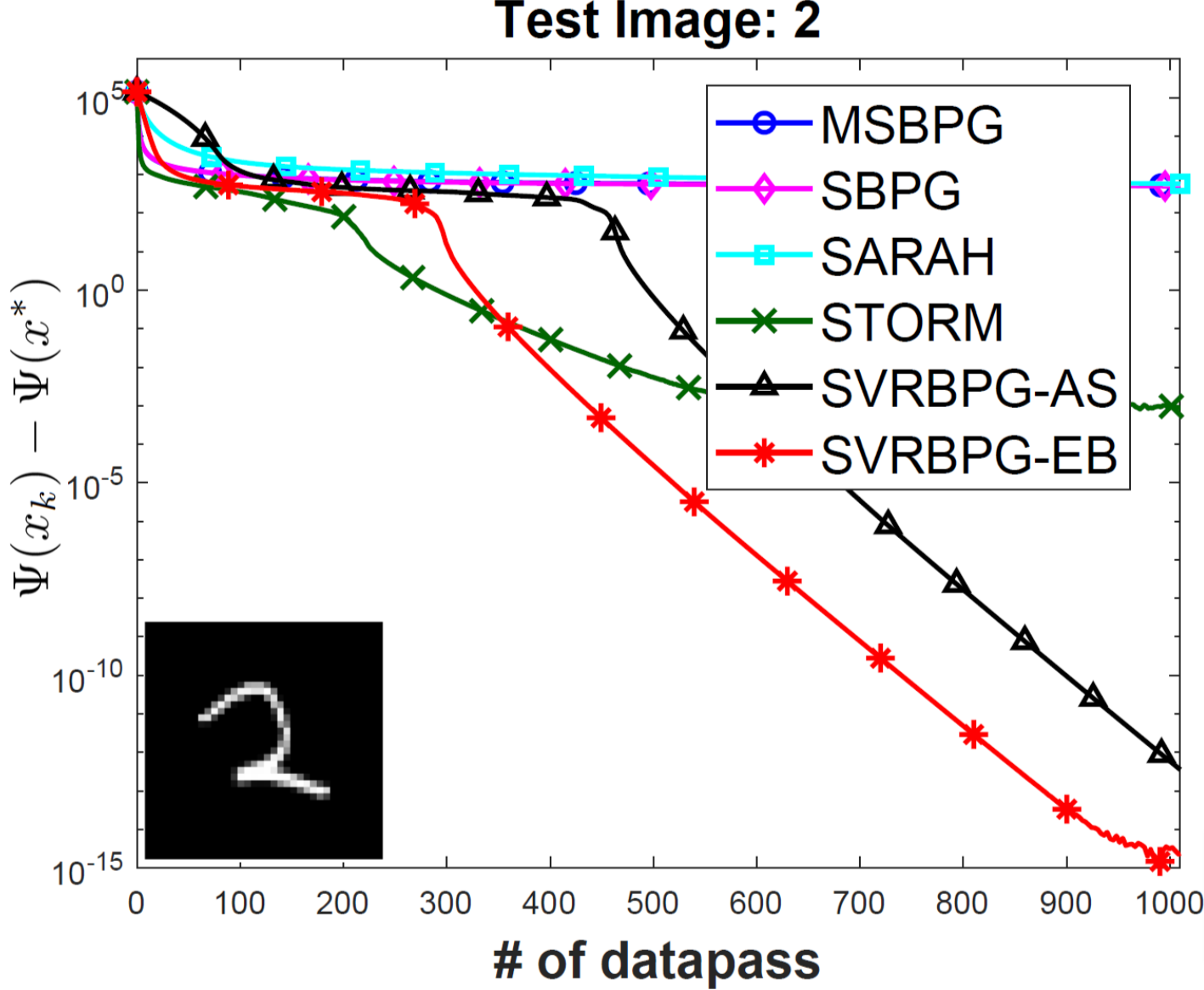}
    \includegraphics[width=0.24\linewidth]{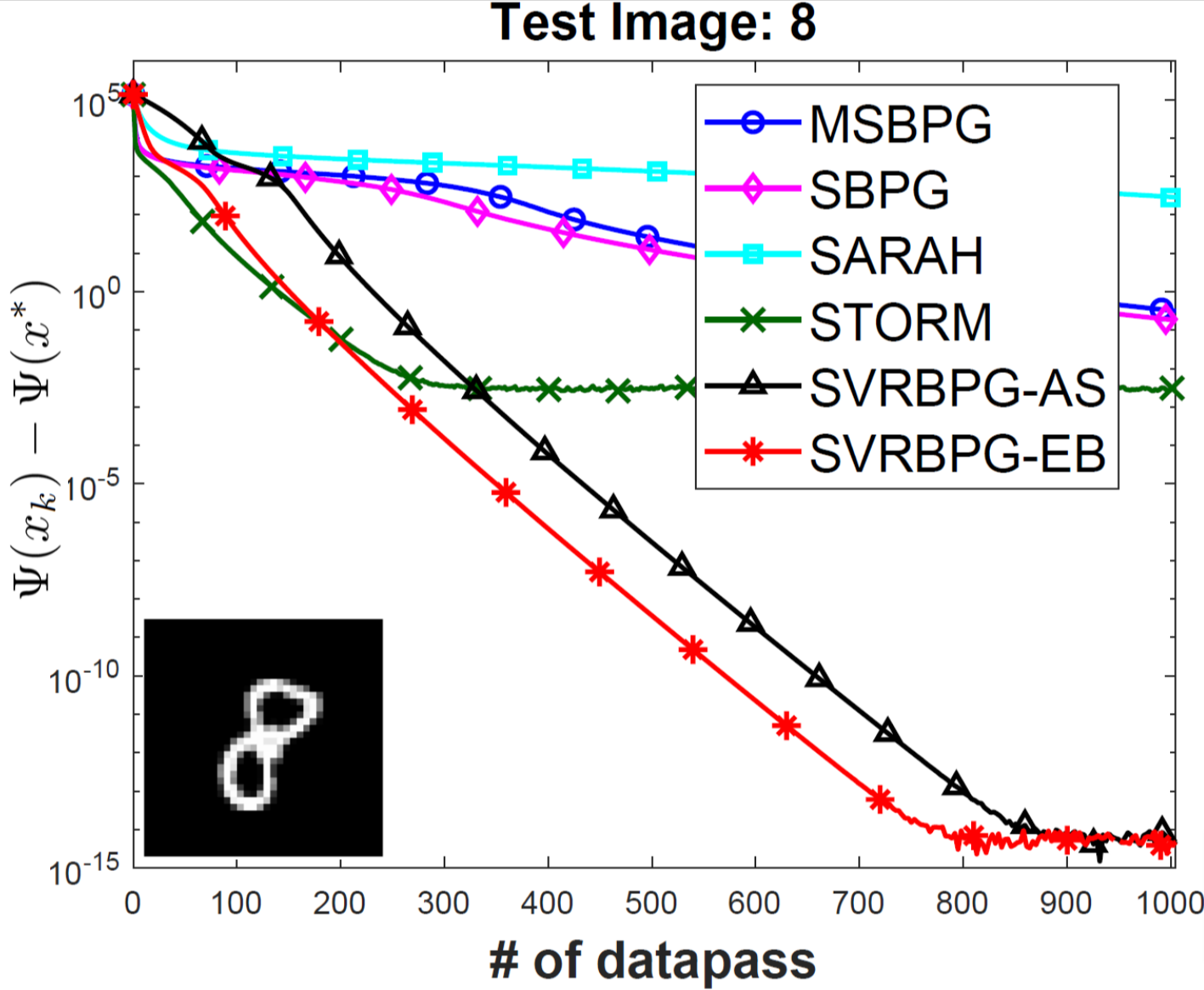}
    \includegraphics[width=0.24\linewidth]{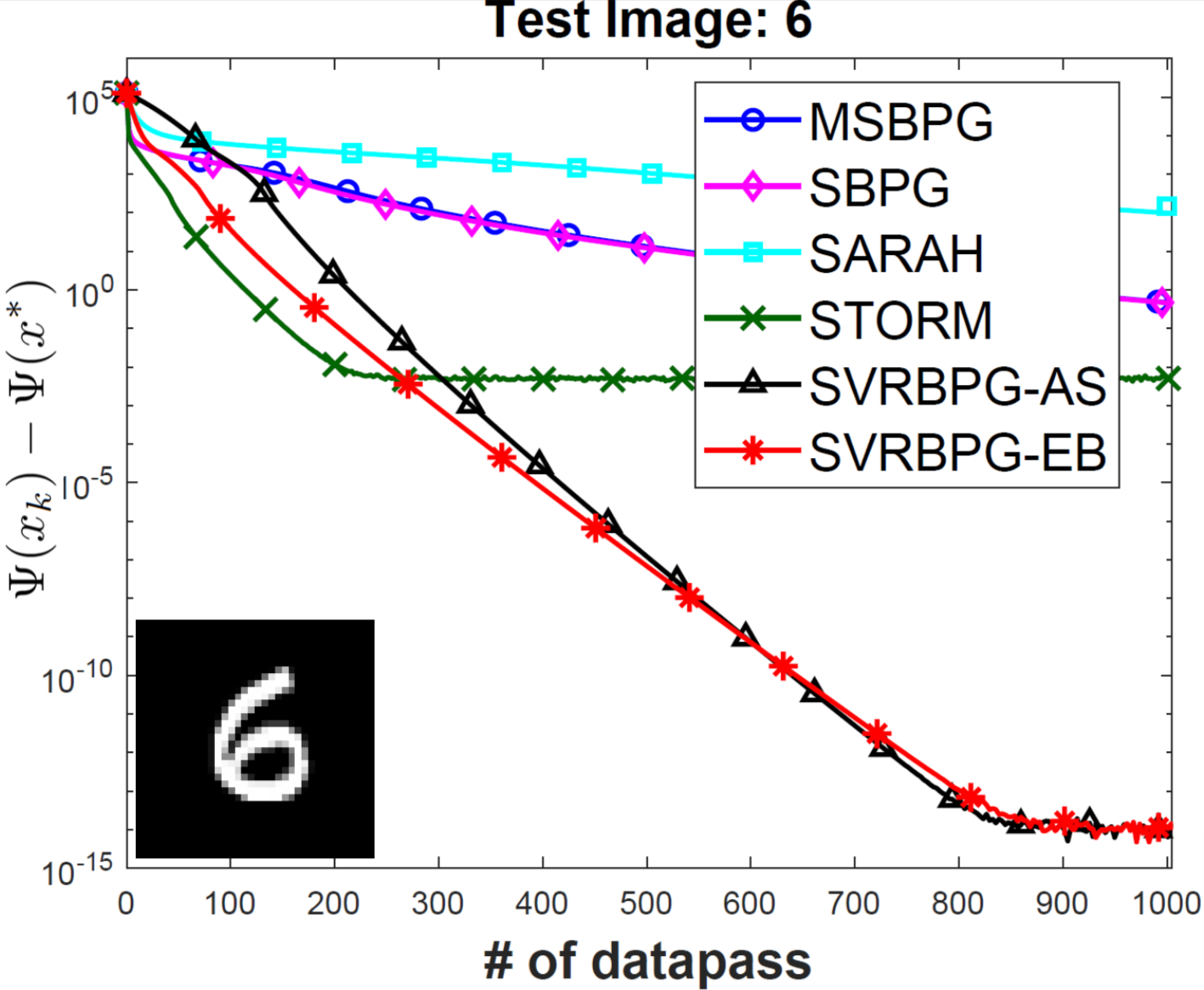}
    \caption{Experiments for $\ell_1$-regularized quadratic inverse problem. We omit the minimal Fr\'echet subdifferential plot for image 6 and replace it with the mismatch factor of new gradient mapping. } 
    \label{fig:quad-inv-sparse}
\end{figure}

In Figure \ref{fig:quad-inv-sparse}, we present the numerical results for the quadratic inverse problem \eqref{prob:quad-inv} with a nonzero sparse $\ell_1$ regularizer. For this case, the MSBPG and SBPG shares a similar performance and both of them are outperformed by the SVRBPG-EB and SVRBPG-AS. For the subproblems of SVRBPG-EB, if the trial solution without $x\in\cX_s$ constraint violates this constraint, then we project this point to $\cX_s$ and run 25 primal-dual iterations for formulation \eqref{prob:splitting}. According to the record, only 2\%, 2.2\%, and 2.2\% iterations need extra primal-dual iterations for the three test cases respectively, and they are only required for the first 8 epochs. In particular, unlike the first dataset where our new gradient mapping $\cDphl(\cdot)\equiv\nabla\Psi(\cdot)$ always hold. In this nonsmooth setting, we present the scatter plot of the mismatch factor ${\mathrm{dist}^2(0,\partial\Psi(\cdot))}/{\|\cDphl(\cdot)\|^2}$. To save computation, we only compute this factor at the first iteration of each epoch. Although $\cDphl(\cdot)$ no longer exactly recover $\partial\Psi(\cdot)$ due to the nonsmoothness of the $\ell_1$ regularizer, $\|\cDphl(\cdot)\|$ still approximates $\mathrm{dist}(0,\partial\Psi(\cdot))$ very well, which also justifies the use of the our newly defined gradient mapping. Similar to the differentiable instances, the non-Bregman variance-reduced method (prox-)SARAH still optimizes slower than SBPG and MSBPG. For STORM, though variance reduction together with adaptive stepsizes facilitate a faster convergence than SBPG and MSBPG, it is not as fast as the proposed methods.

\section{Conclusion and future work}
\label{section:comparison}
\noindent\textbf{Conclusion.}\, In this paper, we point out two important gaps in the sample complexity research of (unconstrained) stochastic BPG method: the absence of an instance-free (worst-case)  complexity result and the inability to get improved complexity by the popular acceleration techniques for SA. We resolve these issues by introducing the KC-regularity concept, under which our newly proposed dual gradient mapping possesses an instance-free constant mismatch against the Fr\'echet measure, and a Lipschitz-like bound for gradient differences is derived. With these tools, we study both the instance-dependent and instance-free complexities for the finite-sum nonconvex smooth-adaptable problem class. Under the most popular squared primal gradient mapping measure (instance-dependent), we improve the existing $O(\epsilon^{-2})$ sample complexity to $O(\sqrt{n}\epsilon^{-1})$. Under the squared dual gradient mapping measure and the standard Fr\'echet measure, we derive an instance-free $O(\sqrt{n}L_h(\cX_\epsilon)\epsilon^{-1})$ sample complexity, where $L_h(\cX_\epsilon)$ contains potential $\epsilon$-dependence for the worst-case hard instances. To our best knowledge, such improvement has yet been achieved by the existing nonconvex stochastic BPG methods. \vspace{0.2cm}

\noindent\textbf{Future works.}   Note that this paper mainly considers the unconstrained BPG method over $\mathbb{R}^d$, where the kernel conditioning is naturally defined for set with bounded Euclidean $\ell_2$-norm radius. However, there are also many problems with constraints, e.g., the optimization problems over $\mathbb{R}_+^d$ paired with regularized Burg's entropy kernel \cite{bauschke2017descent}. Therefore, it remains an interesting question to properly extend kernel conditioning to general kernels that are essentially smooth over a nontrivial subset of $\RR^d$.

\paragraph{Acknowledgment.}  This research is fully supported by the Singapore Ministry of Education (MOE) AcRF Grant, under the WBS number A-0009530-04-00. We also acknowledge the anonymous reviewers for their insightful comments and suggestions that help us improve the result of the paper.

\appendix
\section{Supporting Lemmas}
\label{appendix:SupLem}
\begin{lemma}[Three-Point Property of Tseng \cite{tseng2008accelerated}]
    \label{lemma:Tseng-3point}
    Let $\phi(x)$ be a convex function, and let $D_h(\cdot,\cdot)$ be the Bregman distance for $h(\cdot)$. For a given vector $z$, let
    $z_+:=\argmin_{x\in Q} \phi(x) + D_h(x,z).$
    Then 
    $$\phi(x)+D_h(x,z)\geq\phi(z_+) + D_h(z_+,z) + D_h(x,z_+), \quad\forall x\in Q.$$
\end{lemma}

\begin{lemma}[Lemma 2 finite-sum case of \cite{pham2020proxsarah}]
\label{lemma:proxSARAH-finite-sum}
    Let $v_{s,k}$ be generated by \eqref{alg:sarah-small}, suppose $|\mathcal{B}_{s,k}| = b$ and the sampled index are uniformly randomly picked from $[n]$ with replacement, then  
    \begin{eqnarray}
        \mathbb{E}\Big[\|\nabla f(x_{s,k})-v_{s,k}\|^2\,\big|\,x_{s,0}\Big] \leq \sum_{j=0}^{k-1} \mathbb{E}\Big[\|v_{s,j+1}-v_{s,j}\|^2-\|\nabla f(x_{s,j+1})-\nabla f(x_{s,j})\|^2\,\big|\, x_{s,0}\Big], \nonumber 
    \end{eqnarray}
    where the expectation term of $\|v_{s,j+1}-v_{s,j}\|^2$ satisfies
    \begin{eqnarray}
        \mathbb{E}\Big[\|v_{s,j+1}-v_{s,j}\|^2\,\big|\,x_{s,0}\Big] \leq  \mathbb{E}\bigg[\|\nabla f(x_{s,j+1})-\nabla f(x_{s,j})\|^2 + \frac{1}{bn}\sum_{i=1}^n\|\nabla f_i(x_{s,j+1})-\nabla f_i(x_{s,j})\|^2\,\big|\,x_{s,0}\bigg].  \nonumber
    \end{eqnarray}  
    In particular, we have slightly modified the second inequality to suit our analysis. 
\end{lemma}

\section{Proof of Section \ref{section:preliminary}}
\subsection{Proof of Proposition \ref{proposition: Poly-kernel-condition}}
\label{appdx:proposition: Poly-kernel-condition}
\begin{proof}
    First, direct computation gives 
    $\nabla^2 h(x) = (\|x\|^r+\alpha)\cdot I + r\|x\|^{r-2}\cdot xx^\top$. For $\forall x\in\RR^d$, we have $\lambda_{\max}(\nabla^2h(x))=(r+1)\|x\|^{r} + \alpha$ and $\lambda_{\min}(\nabla^2h(x))=\|x\|^{r} + \alpha$. Then for any compact set $\cX\subseteq\RR^n$ with diameter denoted by $\mathrm{diam}(\cX)=d_\cX$, let $y \in \mathop{\mathrm{argmax}}_{u\in\cX} \|u\|$ and $x \in \mathop{\mathrm{argmin}}_{u\in\cX} \|u\|$, then 
    \begin{eqnarray} 
    \label{prop:poly-KC-1}
        \kappa_h(\cX) = \frac{\lambda_{\max}(\nabla^2h(y))}{\lambda_{\min}(\nabla^2h(x))} = \frac{(r+1)\|y\|^r+\alpha}{\|x\|^r+\alpha}\leq \frac{(r+1)(\|x\|+d_\cX)^r+\alpha}{\|x\|^r+\alpha}.
    \end{eqnarray} 
    When $d_\cX \leq \|x\|/r$, \eqref{prop:poly-KC-1} and the fact that $(1+1/r)^r\leq e < 3, \forall r>0$ indicate 
    \begin{equation}
        \label{prop:poly-KC-1.5}
        \kappa_h(\cX) \leq (r+1) \left(1+\frac{d_\cX}{\|x\|}\right)^r+\frac{\alpha} {\|x\|^r+\alpha}\leq (r+1) \left(1+\frac{1}{r}\right)^r+1\leq 3r+4, \quad  \forall r\geq0.
    \end{equation} 
    This proves half of (i), the other half when $\mathrm{diam}(\cX)\leq\alpha^\frac{1}{r}/r$ is indicated by (ii). Thus we will then directly proceed with the proof of (ii).  For any $\cX$ with $\mathrm{diam}(\cX) = d_\cX \leq \delta$, \eqref{prop:poly-KC-1} indicates that   
    \begin{eqnarray}
    \label{prop:poly-KC-2}
        \kappa_h^\delta \leq (r+1)\cdot\sup\left\{ \frac{(t+\delta)^r}{t^r+\alpha}  : t\geq0\right\} + 1.
    \end{eqnarray}
    By direct computation, the function $\rho(t):=\frac{(t+\delta)^r}{t^r+\alpha}$ attains its maximal value at 
    $$t^* = \begin{cases}
        \alpha^{\frac{1}{r^2-r}}\delta^{-\frac{1}{r-1}}, & \mbox{ if }r>1\\
        0, & \mbox{ if }r\leq1, \delta^r\geq\alpha\\
        +\infty, & \mbox{ if }r\leq1, \delta^r<\alpha
    \end{cases}\qquad\mbox{with}\qquad \rho(t^*) = \begin{cases}
        \big(1+\delta^\frac{r}{r-1}\alpha^{-\frac{1}{r-1}}\big)^{r-1}, & \mbox{ if }r>1\\
        \delta^r/\alpha, & \mbox{ if }r\leq1, \delta^r\geq\alpha\\
        1, & \mbox{ if } r\leq1, \delta^r<\alpha
    \end{cases}$$
    The above result only requires elementary computation of critical points, which will be omitted for simplicity. 
    Substituting the above bounds to \eqref{prop:poly-KC-2} proves that 
    $$\kappa_h^\delta\leq \begin{cases}
        (r+1)\max\big\{1,\frac{\delta^r}{\alpha}\big\}+1,&\mbox{ if } r\leq 1\\
        (r+1)\big(1+\big(\frac{\delta^r}{\alpha}\big)^{\frac{1}{r-1}}\big)^{r-1} + 1, &\mbox{ if } r>1
    \end{cases}$$
    In particular, when $r>1$ and $\delta\leq \alpha^\frac{1}{r}/r \leq \alpha^\frac{1}{r}/r^{\frac{r-1}{r}}$, we have $\big(\frac{\delta^r}{\alpha}\big)^{\frac{1}{r-1}}\leq \frac{1}{r-1}$. Following the same logic of the last inequality in \eqref{prop:poly-KC-1.5}, we obtain $\kappa_h^\delta\leq 3r+4$ in this case. This completes the proof. 
\end{proof}

\subsection{Proof of Proposition \ref{proposition:Lipschitz}}
\label{appdx:proposition:Lipschitz}
Before proving the proposition, let us introduce a simple but not straightforward linear algebra result.   
\begin{lemma}
    \label{lemma:B.1}
    Let $A$ be a symmetric matrix, and let $B\succeq0$ be positive semidefinite matrix. Then $-B\preceq A\preceq B$ indicates that $\|A\|\leq \|B\|$.
\end{lemma} 
\begin{proof}
    For any symmetric but not semidefinite matrix $A$, it is easy to verify that 
    \begin{equation}
        \label{lm:B.1-1}
        \|A\| = \max\big\{|\lambda_{\max}(A)|,|\lambda_{\min}(A)|\big\}.
    \end{equation}
    Let $u\neq0$ be the eigenvector of $A$ associated with the maximum eigenvalue.  Then by Rayleigh’s principle for maximum eigenvalue, $B\succeq A$ indicates that $B-A\succeq0$ and hence 
    \begin{equation}
        \label{lm:B.1-2}
        0\leq \frac{u^\top(B-A)u}{u^\top u} \leq \max_{x\neq0}\left\{\frac{x^\top Bx}{x^\top x}\right\} - \frac{u^\top(B-A)u}{u^\top u} = \lambda_{\max}(B) - \lambda_{\max}(A). 
    \end{equation}
    Similarly, $A\succeq-B$ indicates that $B-(-A)\succeq0$, then \eqref{lm:B.1-2} immediately gives
    \begin{equation}
        \label{lm:B.1-3}
        0\leq  \lambda_{\max}(B) - \lambda_{\max}(-A) = \lambda_{\max}(B) + \lambda_{\min}(A).
    \end{equation}
    Combining \eqref{lm:B.1-1}-\eqref{lm:B.1-3} and the fact that $\|B\| = \lambda_{\max}(B)$ for p.s.d. matrix proves the lemma.
\end{proof}
\noindent Given the above technical lemma, the proof of Proposition \ref{proposition:Lipschitz} becomes straightforward. 
\begin{proof}
     By Lemma \ref{lemma:B.1} and the assumption that $f$ is $L$-smooth adaptable to $h$, we know 
     $$\max\left\{\|\nabla^2f(x)\|: x\in\cX\right\} \leq \max\left\{L\|\nabla^2h(x)\|: x\in\cX\right\} = L\cdot L_h(\cX).$$
     As $\cX$ is convex, the line segment $[x,y]\subseteq\cX$ and hence 
     $$\|\nabla f(z_\theta)-\nabla f(y)\|^2\leq L^2L_h^2(\cX)\|z_\theta-y\|^2 = L^2\theta^2 L_h^2(\cX)\|x-y\|^2.$$
     Combined with    the fact that $$D_h(x,y)\geq\frac{\mu_h([x,y])}{2}\|x-y\|^2\geq  \frac{\mu_h(\cX)}{2}\|x-y\|^2,$$
     we obtain
     $$\frac{\|\nabla f(z_\theta)-\nabla f(y)\|^2}{2L^2\mu_h(\cX)} \leq \theta^2\cdot\frac{L_h^2(\cX)}{\mu_h^2(\cX)}\cdot\frac{\mu_h(\cX)}{2}\|x-y\|^2\leq \theta^2\kappa_h^2(\cX)D_h(x,y),$$
     where KC-regularity guarantees that $\kappa_h(\cX)$ is always upper bounded by $\kappa_h^\delta$.
\end{proof}

\section{Proof of Section \ref{section:VR}}
\label{appendix:sec:VR}
\subsection{Proof of Lemma \ref{lemma:restricted-grd-mapping}}
\begin{proof}
    Denote $\widetilde{\mathcal{G}}=\frac{x_{s,k}-\bar{x}_{s,k+1}}{\eta}$, then $\|\cGshe(x_{s,k})-\widetilde{\mathcal{G}}\| = \frac{1}{\eta}\|\bar{x}_{s,k+1}-\hat{x}_{s,k+1}\|$ holds by definition. 
    By the optimality of $\hat{x}_{s,k+1}$ and $\bar{x}_{s,k+1}$ for the corresponding subproblems, Tseng's three point property (Lemma \ref{lemma:Tseng-3point}) indicates that
    \begin{eqnarray*}
        &&\langle v_{s,k}, \bar{x}_{s,k+1}\rangle + \phi(\bar{x}_{s,k+1}) + \frac{D_h(\bar{x}_{s,k+1},x_{s,k})}{\eta} + \frac{D_h(\hat{x}_{s,k+1},\bar{x}_{s,k+1})}{\eta}\qquad\,\,\,\\
        &\leq& \langle v_{s,k}, \hat{x}_{s,k+1}\rangle + \phi(\hat{x}_{s,k+1})+ \frac{D_h(\hat{x}_{s,k+1},x_{s,k})}{\eta} 
    \end{eqnarray*} 
    and 
    \begin{eqnarray*}
        && \langle \nabla f(x_{s,k}), \hat{x}_{s,k+1}\rangle + \phi(\hat{x}_{s,k+1})+\frac{D_h(\hat{x}_{s,k+1},x_{s,k})}{\eta} + \frac{D_h(\bar{x}_{s,k+1},\hat{x}_{s,k+1})}{\eta}\\
        &\leq& \langle \nabla f(x_{s,k}), \bar{x}_{s,k+1}\rangle + \phi(\bar{x}_{s,k+1})+\frac{D_h(\bar{x}_{s,k+1},x_{s,k})}{\eta}
    \end{eqnarray*} 
    Summing up the two inequalities and applying Lemma \ref{lemma:grd-vs-grdmap} gives 
    \begin{align} &\frac{\mu_h([\bar{x}_{s,k+1},\hat{x}_{s,k+1}])}{\eta}\cdot\|\bar{x}_{s,k+1}-\hat{x}_{s,k+1}\|^2 \leq \frac{D_h(\bar{x}_{s,k+1},\hat{x}_{s,k+1})}{\eta} +  \frac{D_h(\hat{x}_{s,k+1},\bar{x}_{s,k+1})}{\eta}\nonumber\\
    &\qquad\qquad\qquad\qquad\qquad\,\,\,\, \leq  \langle v_{s,k}-\nabla f(x_{s,k}), \hat{x}_{s,k+1}-\bar{x}_{s,k+1}\rangle \leq  \|v_{s,k}-\nabla f(x_{s,k})\|\cdot\|\hat{x}_{s,k+1}-\bar{x}_{s,k+1}\|\nonumber.
    \end{align} 
    Hence $\|\hat{x}_{s,k+1}-\bar{x}_{s,k+1}\|\leq \frac{\eta\|v_{s,k}-\nabla f(x_{s,k})\|}{\mu_h([\bar{x}_{s,k+1},\hat{x}_{s,k+1}])}$, and $\big\|\cGshe(x_{s,k})-\widetilde{\mathcal{G}}\big\|\leq\frac{\|v_{s,k}-\nabla f(x_{s,k})\|}{\mu_h([\bar{x}_{s,k+1},\hat{x}_{s,k+1}])}$. As a result, 
    \begin{eqnarray*}
        \big\|\cGshe(x_{s,k})\big\|^2 \leq \left(\big\|\widetilde{\mathcal{G}}\big\|+\big\|\cGshe(x_{s,k})-\widetilde{\mathcal{G}}\big\|\right)^2\leq  \frac{2\|x_{s,k}-\bar{x}_{s,k+1}\|^2}{\eta^2} + \frac{2\|v_{s,k}-\nabla f(x_{s,k})\|^2}{\mu_h^2(\mathcal{X}_s)},
    \end{eqnarray*} 
    where the last inequality is because $(a+b)^2\leq 2a^2+2b^2$ and $\mu_h^2\big([\bar{x}_{s,k+1},\hat{x}_{s,k+1}]\big)\geq\mu_h^2(\cX_s)$. 
\end{proof}

\subsection{Proof of Lemma \ref{lemma:descent-proxSARAH-finite}}
\begin{proof} 
    First of all, by the update rule of $x_{s,k+1}$, we have   
    \begin{eqnarray}
    \label{lm:descent-proxSARAH-finite-1}
        \Psi(x_{s,k+1}) 
        \!\!\!& = & \!\!\!f\big(x_{s,k} + \gamma (\bar{x}_{s,k+1}-x_{s,k})\big) + \phi\big((1-\gamma)x_{s,k} + \gamma\bar{x}_{s,k+1}\big) \\
        &\overset{(i)}{\leq}& \!\!\!f(x_{s,k}) + \gamma\langle\nabla f(x_{s,k}),\bar{x}_{s,k+1}\!-x_{s,k}\rangle \!+\! LD_h\left(x_{s,k+1},x_{s,k}\right) \!+\! (1-\gamma)\phi(x_{s,k})+\gamma\phi(\bar{x}_{s,k+1})\nonumber\\
        & \overset{(ii)}{\leq} & \!\!\!\Psi(x_{s,k}) + L\kappa_h^\delta\gamma^2D_h(\bar{x}_{s,k+1},x_{s,k}) + \gamma\langle\mathcal{E}_{s,k}+v_{s,k},\bar{x}_{s,k+1}-x_{s,k}\rangle  + \gamma\left(\phi(\bar{x}_{s,k+1})-\phi(x_{s,k})\right)\nonumber\\
        &\overset{(iii)}{\leq}&\!\!\! \Psi(x_{s,k}) + L\kappa_h^\delta\gamma^2D_h(\bar{x}_{s,k+1},x_{s,k}) + \frac{\gamma\eta\|\mathcal{E}_{s,k}\|^2}{\mu_h(\mathcal{X}_s)} + \frac{\gamma\mu_h(\mathcal{X}_s)}{4\eta} \|\bar{x}_{s,k+1}-x_{s,k}\|^2\nonumber\\
        & & \!\!\!-\frac{\gamma}{\eta}D_h(\bar{x}_{s,k+1},x_{s,k}) - \frac{\gamma}{\eta}D_h(x_{s,k},\bar{x}_{s,k+1})\nonumber\\
        &\leq & \!\!\!\Psi(x_{s,k}) - \left(\frac{\gamma}{\eta}- L\kappa_h^\delta\gamma^2\right)D_h(\bar{x}_{s,k+1},x_{s,k}) -\frac{\gamma}{2\eta}D_h(x_{s,k},\bar{x}_{s,k+1})+ \frac{\gamma\eta\|\mathcal{E}_{s,k}\|^2}{\mu_h(\mathcal{X}_s)} \,,\nonumber
    \end{eqnarray}
    where (i) is due to Assumption \ref{assumption:L-smad-finite-sum}, Lemma \ref{lemma:descent-lemma}, and the convexity of $\phi$, (ii) is due to the definition of $\mathcal{E}_{s,k}$ in Lemma \ref{lemma:restricted-grd-mapping} and the following scaling property
    \begin{eqnarray}
        D_h\left(x_{s,k+1},x_{s,k}\right)  \leq \frac{L_h(\mathcal{X}_s)}{2}\|x_{s,k+1}-x_{s,k}\|^2 = \frac{\gamma^2L_h(\mathcal{X}_s)}{2}\|\bar{x}_{s,k+1}-x_{s,k}\|^2 \leq \gamma^2\kappa_h^\delta D_h(\bar{x}_{s,k+1},x_{s,k})\,,\nonumber
    \end{eqnarray}
    and (iii) is due to the optimality of  $\bar{x}_{s,k+1}$ to the corresponding subproblem and Tseng's three point property (Lemma \ref{lemma:Tseng-3point}). Finally, by Lemma \ref{lemma:restricted-grd-mapping}, we also have
    $$\big\|\cGshe(x_{s,k})\big\|^2 \leq \frac{2\|x_{s,k}-\bar{x}_{s,k+1}\|^2}{\eta^2} + \frac{2\|\mathcal{E}_{s,k}\|^2}{\mu_h^2(\mathcal{X}_s)}\,.$$
    Multiplying both sides of the above inequality by $\frac{\gamma\eta\mu_h(\mathcal{X}_s)}{8}$ and add it to \eqref{lm:descent-proxSARAH-finite-1} proves the lemma. 
\end{proof}

\subsection{Proof of Lemma \ref{lemma:proxSARAH-grad-map-mid}}
\begin{proof}
    First, substituting the gradient estimation bound in Lemma \ref{lemma:err-sarah} to Lemma \ref{lemma:descent-proxSARAH-finite}, we have the following descent result throughout the $s$-th epoch

\begin{align}
    \label{lm:proxSARAH-grad-map-mid-1}
    &\mathbb{E}\Big[\Psi(x_{s,\tau_s}) \,\big|\, x_{s,0}\Big] \!\leq \Psi(x_{s,0}) \!-\! \mathbb{E}\bigg[\frac{\gamma\eta\mu_h(\mathcal{X}_s)}{8}\!\sum_{k=0}^{\tau_s-1}\!\big\|\cGshe(x_{s,k})\big\|^2 \!+\!\bigg(\!\frac{\gamma}{\eta}- L\kappa_h^\delta\gamma^2\!\bigg)D_h(\bar{x}_{s,k+1},x_{s,k}) \,\Big|\, x_{s,0}\bigg]\nonumber\\
    &\qquad\qquad\qquad\qquad\,\, +\frac{5\gamma\eta}{4\mu_h(\mathcal{X}_s)}\cdot\frac{2\gamma^2L^2\kappa_h^\delta L_h(\mathcal{X}_s)}{b}\mathbb{E}\bigg[\sum_{k=0}^{\tau_s-1}\sum_{j=0}^{k-1} D_h(\bar{x}_{s,j+1},x_{s,j}) \,\Big|\, x_{s,0}\bigg]\\
    &\quad\,\leq \Psi(x_{s,0}) \!-\! \mathbb{E}\bigg[\frac{\gamma\eta\mu_h(\mathcal{X}_s)}{8}\!\sum_{k=0}^{\tau_s-1}\big\|\mathcal{G}_{s,k}^\eta(x_{s,k})\big\|^2 \!+\bigg(\frac{\gamma}{\eta}- L\kappa_h^\delta\gamma^2-\frac{5\tau\gamma^3\eta L^2(\kappa_h^\delta)^2}{2b}\bigg)D_h(\bar{x}_{s,k+1},x_{s,k}) \,\Big|\, x_{s,0}\bigg]\nonumber
\end{align}

\noindent Suppose we choose $\eta  = \frac{\sqrt{2\tau}}{\sqrt{7\tau}+\sqrt{2b}}$ and we choose $\gamma = \frac{\sqrt{b}}{L\kappa_h^\delta\sqrt{\tau}}$. Then we have 
\begin{eqnarray}
    \label{lm:proxSARAH-grad-map-mid-2}
    &&\frac{\gamma}{\eta} - L\kappa_h^\delta\gamma^2 - \frac{5\tau\gamma^3\eta L^2(\kappa_h^\delta)^2}{2b}\nonumber\\
    & = & \gamma\eta\bigg(\frac{1}{\eta^2} - \frac{L\kappa_h^\delta\gamma}{\eta} - \frac{5\tau\gamma^2 L^2(\kappa_h^\delta)^2}{2b}\bigg)\\
    &=& \gamma\eta\bigg(1+\frac{\sqrt{14} -\sqrt{3.5}}{\sqrt{\tau/b}}\bigg)\nonumber \\
    &\geq& \gamma\eta\,.\nonumber
\end{eqnarray} 
Substitute the bound to the previous inequality, summing up over all epochs, and taking the expectation over all randomness proves the lemma. 
\end{proof}

\section{Proof of Proposition \ref{proposition:example}}
\label{appendix:example:counter-1}
\begin{proof}

By straight computation, we know $\|\nabla^2 f(x)\| \leq 2+(\alpha^2+4)\|x\|^{\alpha}$, hence \eqref{proposition: Poly-kernel-condition} indicates that $f$ is $(\alpha^2+4)$-smooth adaptable to the kernel $h(x) = \frac{\|x\|^2}{2}+\frac{\|x\|^{2+\alpha}}{2+\alpha}$, which proves the argument (i) of the proposition.  Next, for argument (ii), note that 
$$\frac{\partial \Psi(x)}{\partial x_1} = \frac{-1}{(\sqrt{2}+\ln(1+x_1^2))^2}\cdot\frac{2 x_1}{1+x_1^2} + \alpha x_1^{\alpha-1}x_2^2\quad\mbox{and}\quad\frac{\partial \Psi(x)}{\partial x_2} =  2x_1^{\alpha}x_2.$$
By the symmetry of the objective function, let us assume $x_1,x_2>0$ in the following discussion. By $\|\nabla \Psi(x)\|^2\leq\epsilon$, we must have $\|\frac{\partial\Psi(x)}{\partial x_2}\|^2\leq \epsilon$ and hence $x_2\leq \sqrt{\epsilon}/2x_1^\alpha.$ Together with $|x_1|\geq1$, the second term of $\frac{\partial\Psi(x)}{\partial x_1}$ satisfies 
$\alpha x_1^{\alpha-1}x_2^2\leq \frac{\alpha\epsilon}{4x_1^{\alpha+1}}\leq \frac{\alpha\epsilon}{4}.$ As $\alpha\epsilon/4\leq\sqrt{\epsilon}$ when $\epsilon\leq8/\alpha^2$, we have
$$\frac{1}{(\sqrt{2}+\ln(2x_1^2))^2}\cdot\frac{1}{x_1}\leq\frac{1}{(\sqrt{2}+\ln(1+x_1^2))^2}\cdot\frac{2 x_1}{1+x_1^2}\leq \sqrt{\epsilon}+\frac{\alpha\epsilon}{4}\leq 2\sqrt{\epsilon},$$
which indicates that $x_1\geq\Omega\big(\frac{1}{\sqrt{\epsilon}\ln^2\epsilon^{-1}}\big).$ This proves the argument (ii).

Finally, to prove argument (iii), let us prove by induction that $x_2^t=0$ and $x_1^t\geq1$ for all $t\geq0$. By initialization, $x_2^0=0$ and $x_1^0=1$.  Suppose $x_2^k=0$ and $x_1^k\geq1$, then $\frac{\partial \Psi(x^k)}{\partial x_2}=\frac{\partial h(x^k)}{\partial x_2}=0$. Then substituting these derivatives to the BPG subproblem yields
\begin{equation}
\label{appendix:example:counter-2-pf-1}    (x_1^{k+1},x_2^{k+1})=\argmin_{x_1,x_2}\,\, \left(\frac{\partial \Psi(x^k)}{\partial x_1}-\frac{1}{\lambda_k}\cdot\frac{\partial h(x^k)}{\partial x_1}\right)\cdot x_1 + \frac{1}{\lambda_k}\cdot \left(\frac{\|x\|^2}{2}+\frac{\|x\|^{2+\alpha}}{2+\alpha}  \right).
\end{equation} 
Note that $\frac{\partial h(x^k)}{\partial x_1}=x_1^k + (x_1^k)^{\alpha+1}$ when $x_2^k=0, x_1^k\geq1$ and $\nabla h(x^{k+1}) = x^{k+1}+\|x^{k+1}\|^\alpha\cdot x^{k+1}$, we can write the KKT condition of the convex problem \eqref{appendix:example:counter-2-pf-1} as 
\begin{eqnarray*}
    \lambda_k\cdot\frac{\partial \Psi(x^k)}{\partial x_1} - x_1^k - (x_1^k)^{\alpha+1} + x^{k+1}_1+\|x^{k+1}\|^\alpha\cdot x^{k+1}_1 & = & 0,\\
    x^{k+1}_2+\|x^{k+1}\|^\alpha\cdot x^{k+1}_2& = & 0.
\end{eqnarray*}
The second equation indicates that  $x_2^{k+1}=0$.  Because $\frac{\partial \Psi(x_1^k,0)}{\partial x_1}<0$ when $x_1^k\geq 1$, the first equation of KKT condition implies that $x_1^{k+1}>x_1^k\geq1$, and it can be further simplified to 
\begin{equation}
    \label{appendix:example:counter-2-pf-2}
     (x_1^{k+1})^{\alpha+1}+x_1^{k+1} =  (x_1^{k})^{\alpha+1}+x_1^{k} - \lambda_k\cdot\frac{\partial \Psi(x^k)}{\partial x_1}.
\end{equation}
By induction, we know $x_2^k=0$ and $x_1^k\geq1$ for $\forall k\geq0$. In addition, the above analysis also indicates that $x_1^k$ is monotonically increasing. 
 
Therefore, substituting the value of $\frac{\partial \Psi(x^k)}{\partial x_1}$ to \eqref{appendix:example:counter-2-pf-2} gives 
$$(x_1^{k+1})^{\alpha+1} \leq  (x_1^{k})^{\alpha+1} +\frac{2\lambda_k x_1^k}{(\sqrt{2}+\ln(1+(x_1^k)^2))^2(1+(x_1^k)^2)}\leq (x_1^{k})^{\alpha+1} + \frac{\lambda_k}{x_1^k}\leq (x_1^{k})^{\alpha+1}\left(1+\frac{\lambda_k}{(x_1^k)^{\alpha+2}}\right)$$
Take the $(\alpha+1)$-th root and apply the inequality that $(1+u)^\alpha\leq 1+\alpha u, \forall u\geq0, \forall \alpha\in[0,1]$, we obtain 
\begin{eqnarray}
    \label{appendix:example:counter-2-pf-3}
    x_1^{k+1} \leq  x_1^{k}+\frac{\lambda_k}{(\alpha+1)(x_1^k)^{\alpha+1}} \leq x_1^{k}+\frac{1}{(\alpha+1)(x_1^k)^{\alpha+1}},
\end{eqnarray} 
where the last inequality is because $\lambda_k\leq 1$ for both the update scheme \eqref{defn:BPG-standard} and  \eqref{defn:TBPG}.  
Define $k_t:=\inf\{k:x_1^{k}\geq t\}$, for $t=1,2,3,\cdots$, where $k_1=0$. Then we know $x_1^{k_t-1}<t$. On the other hand, \eqref{appendix:example:counter-2-pf-3} indicates that $x_1^{k_t}\leq t+\frac{1}{\alpha+1}$. Therefore, we have 
$$t+1\leq x_1^{k_{t+1}} \leq x_1^{k_{t+1}-1} + \frac{1}{(\alpha+1)t^{\alpha+1}}\leq\cdots\leq x_1^{k_t}+\frac{k_{t+1}-k_t}{(\alpha+1)t^{\alpha+1}}\leq t+\frac{1}{\alpha+1}+\frac{k_{t+1}-k_t}{(\alpha+1)t^{\alpha+1}}.$$
That is, $k_{t+1}-k_t\geq \alpha t^{\alpha+1}$. Notice that if we want $x_1^k\in[t,t+1)$, then we will need
$$k\geq k_t\geq \alpha\cdot\sum_{\tau=1}^{t-1} \tau^{\alpha+1}\geq\frac{\alpha}{\alpha+2}\cdot(t-1)^{\alpha+2}\geq \frac{\alpha}{\alpha+2}\cdot(x_1^k-2)^{\alpha+2}.$$
As a result, we have $x_1^k\leq \left(\frac{\alpha+2}{\alpha}\cdot k\right)^{\frac{1}{\alpha+2}}+2 = O(k^{\frac{1}{\alpha+2}}).$ Substituting this bound and $x_2^k=0$ to the gradient $\nabla \Psi(x^k)$ yields
$$\|\nabla \Psi(x^k)\| = \frac{2 x_1^k}{(\sqrt{2}+\ln(1+(x_1^k)^2))^2(1+(x_1^k)^2)}=\Tilde{\Omega}\left(k^{-\frac{1}{\alpha+2}}\right).$$
Then squaring both sides and using the fact that $x_1^k$ is monotonically increasing proves the argument (iii) of the proposition. 
\end{proof}

\section{Extension to multi-block kernel conditioning}
\label{section:multi-blocks}

Finally, we will roughly discuss how one can extend KC-regularity to block separable kernels, which is considered in \cite{bauschke2017descent,gao2020randomized,gao2021convergence,hanzely2021fastest}. In particular, we will mostly focus on the extension of Section \ref{section:VR}, while the extension of Section \ref{section:GrdMap-new} can be done similarly. To differentiate the meaning of the subscripts, in this specific section, we will use direct subscript to denote the iteration counters and use the subscript of $[\cdot]$ to denote the block index. For example, for epoch-wise algorithm such as Algorithm \ref{algorithm:proxSARAH}, $[x_{s,k}]_i$ denotes the $i$-th block of the $(s,k)$-th iteration $x_{s,k}$.   
In this case, suppose the decision variable $x\in\RR^d$ is separated into $m$ blocks, the kernel will take the form 
$h(x) = \sum_{i=1}^mh_i([x]_i),$ and the KC-regularity (Assumption \ref{assumption:kernel-conditioning}) will naturally be extended a block separable version. 
\begin{assumption}[Block KC-regularity]
    \label{assumption:kernel-conditioning-block}
    We say a block separable kernel $h(x) = \sum_{i=1}^mh_i([x]_i)$ satisfies the block KC-regularity if the component function $h_i$ is KC-regular for each $i\in[m]$.
\end{assumption}
A clear motivation for this modification is that when the different variable blocks $[x]_i$ have significant differences in magnitude, it is not realistic to require a bounded kernel condition number for the whole function. One may consider a bivariate case where $h_1(y)=h_2(y) = y^2/2 + y^4/4$. Letting $[x]_1 = 0$ and $[x]_2\to\infty$ will cause $\kappa_h(\cX)\to\infty$ even for singleton set $\cX = \{x=([x]_1,[x]_2)\}$. However, if we separately consider each $h_i$, then Proposition \ref{proposition: Poly-kernel-condition} implies the validity of Assumption \ref{assumption:kernel-conditioning-block}. In the special case where $h$ is element separable \cite{bauschke2017descent,hanzely2021fastest}, Assumption \ref{assumption:kernel-conditioning-block} is in fact very easy to satisfy. Even if $h_i$ has exponentially fast growth, e.g., $h_i(y) = y^2/2 + \exp\{y/R_i\}$, it satisfies kernel conditioning regularity with $\kappa_{h_i}^{\delta_i} = \exp\{\delta_i/R_i\}$, for  $\forall \delta_i>0$.

Now consider Algorithm \ref{algorithm:proxSARAH}, under Assumption \ref{assumption:kernel-conditioning-block} and suppose $\phi$ takes a separable structure $\phi(x) = \sum_{i=1}^m\phi_i([x]_i)$, then it is very natural to modify the Line 3 of Algorithm \ref{algorithm:proxSARAH} to 
``construct a convex set $\cX_s = \cX_s^1\times\cdots\times\cX_s^m$ such that $B([x_{s,0}]_i,\delta_i/2)\subseteq\cX_s^i$ and $\kappa_{h_i}(\cX_s^i)\leq \kappa_{h_i}^{\delta_i}$,'' with properly selected $\delta_i$.  If we closely inspect the analysis of Lemma \ref{lemma:descent-proxSARAH-finite}, \ref{lemma:err-sarah} and \ref{lemma:proxSARAH-grad-map-mid}, we can find that the key is to cancel out the error term $\langle \mathcal{E}_{s,k},\bar{x}_{s,k+1}-x_{s,k}\rangle$ in \eqref{lm:descent-proxSARAH-finite-1} by the Bregman divergence descent terms. Due to the block-wise kernel conditioning regularity (Assumption \ref{assumption:kernel-conditioning-block}), it is natural to upper bound it by 
\begin{equation}
    \langle \mathcal{E}_{s,k},\bar{x}_{s,k+1}-x_{s,k}\rangle \leq \sum_{i=1}^m\left(\frac{\mu_{h_i}(\cX_s^i)}{2}\big\|[\bar{x}_{s,k+1}]_i-[x_{s,k}]_i\big\|^2 + \frac{\|[\mathcal{E}_{s,k}]_i\|^2}{2\mu_{h_i}(\cX_s^i)}\right).
\end{equation}
Assuming $h_i$ to be globally $\mu_i$-strongly convex as a counterpart of Assumption \ref{assumption:SC}, we 
 obtain the following  counterpart of Lemma \ref{lemma:descent-proxSARAH-finite}:
\begin{eqnarray}
    \Psi(x_{s,k+1})\leq \Psi(x_{s,k})-\frac{\gamma\eta\mu_{\min}}{8}\big\|\cGshe(x_{s,k})\big\|^2\!-\!\bigg(\frac{\gamma}{\eta}- L\kappa_{\max}\gamma^2\!\bigg)\!D_h(\bar{x}_{s,k+1},x_{s,k}) \!+\!\sum_{i=1}^m\frac{5\gamma\eta\|[\mathcal{E}_{s,k}]_i\|^2}{4\mu_h(\mathcal{X}_s^i)},\,\nonumber
    \end{eqnarray}
    where $\mu_{\min} = \min_i \mu_i$ and $\kappa_{\max} = \max_i\kappa_{h_i}^{\delta_i}$. Then it remains to bound $\mathbb{E}\Big[\frac{\|[\mathcal{E}_{s,k}]_i\|^2}{\mu_h(\mathcal{X}_s^i)}\mid x_{s,0}\Big]$ for each $i$. 
    To tightly bound this term, we need to introduce the following lemma. 
    \begin{lemma}
    \label{lemma:block-func-diff}
        Under Assumption \ref{assumption:L-smad-finite-sum} and \ref{assumption:kernel-conditioning-block}, for any block index $i\in[m]$, denote $\mathcal{N}_i:=[m]\backslash \{i\}$. For any component function $f_s$ with $s\in[n]$, we have 
        $$\|\nabla_i f_s(x) - \nabla_i f_s(y)\|^2\leq 8L_s^2L_{h_i}\big(\big[[x]_i,[y]_i\big]\big)\cdot\sum_{j=1}^mL_{h_j}\big(\big[[x]_j,[y]_j\big]\big)\|[x]_j-[y]_j\|^2,$$
        where $\nabla_if_s(\cdot)$ denote the partial derivatives of $f_s$ w.r.t. the $i$-th variable block. 
    \end{lemma}
    We place the proof of this lemma at the end of this section to keep focus on the current discussion. Based on this result,   Lemma \ref{lemma:err-sarah} can be extended to  
    $$\sum_{i=1}^m\mathbb{E}\left[\frac{\|[\mathcal{E}_{s,k}]_i\|^2}{\mu_h(\mathcal{X}_s^i)}\,\Big|\, x_{s,0}\right] \leq  \frac{2\gamma^2(2\sqrt{2m}L)^2  \kappa_{\max}^2}{b}\cdot\mathbb{E}\bigg[\sum_{j=0}^{k-1} D_h(\bar{x}_{s,j+1},x_{s,j})\,\big|\,x_{s,0}\bigg].$$
    Therefore, if we do the following replacement in the analysis of Section \ref{section:VR}:
    $$\mu\leftarrow\mu_{\min}, \quad \kappa_{h}^\delta\leftarrow\kappa_{\max}, \quad \delta\leftarrow\delta_{\min}:=\min_{i}\delta_i, \quad L\leftarrow 2\sqrt{2m}L,$$
    then all the proof will remain valid and Theorem \ref{theorem:proxSARAH-finite} will still hold under such replacement.

    \begin{corollary} 
    \label{corollary:proxSARAH-finite-var} Suppose Assumption \ref{assumption:L-smad-finite-sum} and \ref{assumption:kernel-conditioning-block} hold.
    For any constant batch size $|\mathcal{B}_{s,k}|=b\in[n]$, let us set $\tau = \lceil n/b\rceil$,  $\eta = \frac{\sqrt{2\tau}}{\sqrt{7\tau}+\sqrt{2b}}$,  $\gamma = \frac{\sqrt{b}}{2\sqrt{2m}L\kappa_{\max}\sqrt{\tau}}$, and $S = \big\lceil\frac{16\Delta_\Psi}{\tau\gamma\eta\mu_{\min}\epsilon}\big\rceil$. Suppose the target accuracy satisfies $\epsilon\leq \frac{\delta^2_{\min}}{16}\cdot\min\big\{\frac{8mL^2\kappa_{\max}^2}{b\tau},\frac{1}{9\eta^2}\big\}$ and let $x_{\mathrm{out}}$ be uniformly randomly selected from all iterations, then there is a high probability event $\mathcal{A}$ such that $$\mathbb{E}\Big[\big\|\cGphe(x_{\mathrm{out}})\big\|^2 \,\big|\, \mathcal{A}\Big]\leq 4\epsilon\qquad\mbox{and}\qquad \mathrm{Prob}\left(\mathcal{A}\right)\geq1 - \frac{\eta\tau b\cdot\epsilon}{mL^2 \kappa_{\max}^2\delta^2_{\min}}-\frac{4\sqrt{\epsilon}}{\delta_{\min}},$$
    where $\mathrm{Prob}\left(\mathcal{A}\right) \geq 1-O(n\epsilon/m+\sqrt{\epsilon})\to1$ as $\epsilon\to0.$ Suppose we take the batch size $b=O(n^\alpha)$, $\alpha\in[0,1/2]$, then the total number of samples consumed is $O\big(\sqrt{mn}/\epsilon\big)$.
\end{corollary}  
\noindent Therefore, as long as the number of blocks is not too large, the complexity and convergence result is almost the same as the single block situation (Theorem \ref{theorem:proxSARAH-finite}). For example $m=2$ for the two-layer neural network considered in \cite{ding2023nonconvex}, for the multi-layer extension of \cite{ding2023nonconvex}, $m$ naturally be the number of layers of the network, which will be very mild.

\subsection{Proof of Lemma \ref{lemma:block-func-diff}}
\begin{proof}
    Before proving Lemma \ref{lemma:block-func-diff}, we need to establish a linear algebra result first. For any matrix $A,D\succ0$, and any matrix $B$, we have the following argument
    $$\begin{bmatrix}
        A & B\\
        B^\top & D
    \end{bmatrix}\succeq0 \Leftrightarrow \begin{bmatrix}
        I & -BD^{-1}\\
        0 & I
    \end{bmatrix}\begin{bmatrix}
        A & B\\
        B^\top & D
    \end{bmatrix}\begin{bmatrix}
        I & 0\\
        -D^{-1}B^\top & I
    \end{bmatrix}\succeq0\Leftrightarrow\begin{bmatrix}
        A - BD^{-1}B^\top & 0\\
        0 & D
    \end{bmatrix}\succeq0.$$
    That is, the first matrix in the above inequality being p.s.d. indicates that 
    $$A - BB^\top/\|D\|\succeq A - BD^{-1}B^\top\succeq0.$$ 
    Consequently, we have $\|D\|\cdot A-BB^\top\succeq0$. 

    If $A,D\succeq0$ are possibly singular, then one can repeat the above argument with $A_t := A+tI, D_t := D+tI, t>0$ to obtain $\|D_t\|\cdot A_t-BB^\top\succeq0$. Letting $t\to0$ and using the fact that the norm and the minimum eigenvalue of a matrix are continuous functions of its elements, we know $\|D\|\cdot A-BB^\top\succeq0$.

    Overall, if $\begin{bmatrix}
        A & B\\
        B^\top & D
    \end{bmatrix}\succeq0$ and $A,D\succeq0$, then we must have $\|D\|\cdot A-BB^\top\succeq0$.

    Next, without loss of generality, suppose the block index $i = m$, then $\mathcal{N}_m = [m-1]$, and we can apply this result to the matrix $L_s\cdot\nabla^2h(x)-\nabla^2f_s(x)$ with block division
    \begin{eqnarray*}
        A(x)  & = & L_s\cdot\nabla^2_{\mathcal{N}_m,\mathcal{N}_m}h(x)-\nabla^2_{\mathcal{N}_m,\mathcal{N}_m}f_s(x)\\
        & := & L_s\cdot\mathrm{Diag}\Big(\{\nabla^2h_i([x]_i)\}_{i=1}^{m-1}\Big)-\begin{bmatrix}
        \nabla^2_{1,1}f_s(x) & \cdots & \nabla^2_{1,m-1}f_s(x) \\
        \vdots & \ddots & \vdots\\
        \nabla^2_{m-1,1}f_s(x)& \cdots & \nabla^2_{m-1,m-1}f_s(x)  
    \end{bmatrix}\\
    B(x)&=& \nabla^2_{\mathcal{N}_m,m}f_s(x) := \begin{bmatrix}
        \nabla^2_{m,1}f_s(x) & \cdots & \nabla^2_{m,m-1}f_s(x)
    \end{bmatrix}^\top \\
    D(x) & = & L_s\cdot\nabla^2h_m([x]_m) - \nabla^2_{m,m}f(x)
    \end{eqnarray*}   
    Then we have 
    $$B(x)B(x)^\top\preceq \|D(x)\|\cdot A(x)\overset{(i)}{\preceq} 4L_s^2\|\nabla^2h_m([x]_m)\|\cdot \mathrm{Diag}\Big(\{\nabla^2h_i([x]_i)\}_{i=1}^{m-1}\Big)$$
    where (i) is because Assumption \ref{assumption:L-smad-finite-sum} guarantees that $L_s\cdot\nabla^2_{\mathcal{N}_m,\mathcal{N}_m}h(x)\pm\nabla^2_{\mathcal{N}_m,\mathcal{N}_m}f_s(x)\succeq0$ and $L_s\nabla^2h_m([x]_m)\pm\nabla^2_{m,m}f(x)\succeq0$. As a result, denoting $w = y-x$ and $\cX^j:=\big[[x]_j,[y]_j\big]$ for each $j\in[m]$, we can start bounding the squared difference of gradients as 
    \begin{eqnarray*}
    \|\nabla_mf(y)-\nabla_mf(x)\|^2 & = & \left\|\int_{0}^1\nabla^2_{m,[m]}f_s(x+tw)w \mathrm{d}t\right\|^2 \\
    & \leq & \int_{0}^1\|\nabla^2_{m,[m]}f_s(x+tw)w\|^2 \mathrm{d}t\\
    & \leq  & 2\int_{0}^1\left(\left\|\nabla^2_{m,m}f_s(x+tw)[w]_m\right\|^2 + \left\|B(x+tw)^\top w_{\mathcal{N}_m}\right\|^2\right) \mathrm{d}t.
    \end{eqnarray*}
    Note that 
    $$\left\|\nabla^2_{m,m}f_s(x+tw)[w]_m\right\|^2\leq L_s^2L_{h_m}^2\big(\cX^m\big)\|[w]_m\|^2\qquad\mbox{for}\qquad \forall t\in[0,1]$$
    and 
    \begin{eqnarray*}
        \left\|B(x+tw)^\top w_{\mathcal{N}_m}\right\|^2 & = & w_{\mathcal{N}_m}^\top B(x+tw)B(x+tw)^\top w_{\mathcal{N}_m} \\
        & \leq & 4L_s^2\|\nabla^2h_m([x+tw]_m)\|\cdot w_{\mathcal{N}_m}^\top \mathrm{Diag}\Big(\{\nabla^2h_i([x+tw]_j)\}_{j=1}^{m-1}\Big)w_{\mathcal{N}_m}\\
        & \leq & 4L_s^2L_{h_m}\big(\cX^m\big)\|\cdot \sum_{j\neq m}L_{h_j}\big(\cX^j\big)\|[w]_j\|^2\qquad\qquad\mbox{for}\qquad\qquad \forall t\in[0,1].
    \end{eqnarray*}  
    Then, combining the above inequalities proves the lemma. 
\end{proof}

\bibliographystyle{plain}
\bibliography{ref}

\begin{thebibliography}{10}

\bibitem{agrawal2012analysis}
Shipra Agrawal and Navin Goyal.
\newblock Analysis of thompson sampling for the multi-armed bandit problem.
\newblock In {\em Conference on learning theory}, pages 39--1. JMLR Workshop
  and Conference Proceedings, 2012.

\bibitem{arjevani2023lower}
Yossi Arjevani, Yair Carmon, John~C Duchi, Dylan~J Foster, Nathan Srebro, and
  Blake Woodworth.
\newblock Lower bounds for non-convex stochastic optimization.
\newblock {\em Mathematical Programming}, 199(1-2):165--214, 2023.

\bibitem{bauschke2017descent}
Heinz~H Bauschke, J{\'e}r{\^o}me Bolte, and Marc Teboulle.
\newblock A descent lemma beyond lipschitz gradient continuity: first-order
  methods revisited and applications.
\newblock {\em Mathematics of Operations Research}, 42(2):330--348, 2017.

\bibitem{beck2017first}
Amir Beck.
\newblock {\em First-order methods in optimization}.
\newblock SIAM, 2017.

\bibitem{bolte2018first}
J{\'e}r{\^o}me Bolte, Shoham Sabach, Marc Teboulle, and Yakov Vaisbourd.
\newblock First order methods beyond convexity and lipschitz gradient
  continuity with applications to quadratic inverse problems.
\newblock {\em SIAM Journal on Optimization}, 28(3):2131--2151, 2018.

\bibitem{bubeck2013prior}
S{\'e}bastien Bubeck and Che-Yu Liu.
\newblock Prior-free and prior-dependent regret bounds for thompson sampling.
\newblock {\em Advances in neural information processing systems}, 26, 2013.

\bibitem{cutkosky2019momentum}
Ashok Cutkosky and Francesco Orabona.
\newblock Momentum-based variance reduction in non-convex sgd.
\newblock {\em Advances in neural information processing systems}, 32, 2019.

\bibitem{davis2018stochastic}
Damek Davis, Dmitriy Drusvyatskiy, and Kellie~J MacPhee.
\newblock Stochastic model-based minimization under high-order growth.
\newblock {\em arXiv preprint arXiv:1807.00255}, 2018.

\bibitem{ding2023nonconvex}
Kuangyu Ding, Jingyang Li, and Kim-Chuan Toh.
\newblock Nonconvex stochastic bregman proximal gradient method with
  application to deep learning.
\newblock {\em arXiv preprint arXiv:2306.14522}, 2023.

\bibitem{dou2022gap}
Zehao Dou, Zhuoran Yang, Zhaoran Wang, and Simon Du.
\newblock Gap-dependent bounds for two-player markov games.
\newblock In {\em International Conference on Artificial Intelligence and
  Statistics}, pages 432--455. PMLR, 2022.

\bibitem{dragomir2021fast}
Radu~Alexandru Dragomir, Mathieu Even, and Hadrien Hendrikx.
\newblock Fast stochastic bregman gradient methods: Sharp analysis and variance
  reduction.
\newblock In {\em International Conference on Machine Learning}, pages
  2815--2825. PMLR, 2021.

\bibitem{drusvyatskiy2018error}
Dmitriy Drusvyatskiy and Adrian~S Lewis.
\newblock Error bounds, quadratic growth, and linear convergence of proximal
  methods.
\newblock {\em Mathematics of Operations Research}, 43(3):919--948, 2018.

\bibitem{fang2018spider}
Cong Fang, Chris~Junchi Li, Zhouchen Lin, and Tong Zhang.
\newblock Spider: Near-optimal non-convex optimization via stochastic
  path-integrated differential estimator.
\newblock {\em Advances in neural information processing systems}, 31, 2018.

\bibitem{fatkhullin2024taming}
Ilyas Fatkhullin and Niao He.
\newblock Taming nonconvex stochastic mirror descent with general bregman
  divergence.
\newblock In {\em International Conference on Artificial Intelligence and
  Statistics}, pages 3493--3501. PMLR, 2024.

\bibitem{gao2020randomized}
Tianxiang Gao, Songtao Lu, Jia Liu, and Chris Chu.
\newblock Randomized bregman coordinate descent methods for non-lipschitz
  optimization.
\newblock {\em arXiv preprint arXiv:2001.05202}, 2020.

\bibitem{gao2021convergence}
Tianxiang Gao, Songtao Lu, Jia Liu, and Chris Chu.
\newblock On the convergence of randomized bregman coordinate descent for
  non-lipschitz composite problems.
\newblock In {\em ICASSP 2021-2021 IEEE International Conference on Acoustics,
  Speech and Signal Processing (ICASSP)}, pages 5549--5553. IEEE, 2021.

\bibitem{ghadimi2016mini}
Saeed Ghadimi, Guanghui Lan, and Hongchao Zhang.
\newblock Mini-batch stochastic approximation methods for nonconvex stochastic
  composite optimization.
\newblock {\em Mathematical Programming}, 155(1):267--305, 2016.

\bibitem{grapiglia2020tensor}
Geovani~Nunes Grapiglia and Yu~Nesterov.
\newblock Tensor methods for minimizing convex functions with h\"{o}lder
  continuous higher-order derivatives.
\newblock {\em SIAM Journal on Optimization}, 30(4):2750--2779, 2020.

\bibitem{hanzely2021fastest}
Filip Hanzely and Peter Richt{\'a}rik.
\newblock Fastest rates for stochastic mirror descent methods.
\newblock {\em Computational Optimization and Applications}, 79:717--766, 2021.

\bibitem{hanzely2021accelerated}
Filip Hanzely, Peter Richtarik, and Lin Xiao.
\newblock Accelerated bregman proximal gradient methods for relatively smooth
  convex optimization.
\newblock {\em Computational Optimization and Applications}, 79:405--440, 2021.

\bibitem{johnson2013accelerating}
Rie Johnson and Tong Zhang.
\newblock Accelerating stochastic gradient descent using predictive variance
  reduction.
\newblock {\em Advances in neural information processing systems}, 26, 2013.

\bibitem{juditsky2008large}
Anatoli Juditsky and Arkadii~S Nemirovski.
\newblock Large deviations of vector-valued martingales in 2-smooth normed
  spaces.
\newblock {\em arXiv preprint arXiv:0809.0813}, 2008.

\bibitem{kaufmann2012thompson}
Emilie Kaufmann, Nathaniel Korda, and R{\'e}mi Munos.
\newblock Thompson sampling: An asymptotically optimal finite-time analysis.
\newblock In {\em International conference on algorithmic learning theory},
  pages 199--213. Springer, 2012.

\bibitem{Kruger2003}
A.~Ya. Kruger.
\newblock On fr\'{e}chet subdifferentials.
\newblock {\em Journal of Mathematical Sciences}, 116:3325–3358, 2003.

\bibitem{latafat2022bregman}
Puya Latafat, Andreas Themelis, Masoud Ahookhosh, and Panagiotis Patrinos.
\newblock Bregman finito/miso for nonconvex regularized finite sum minimization
  without lipschitz gradient continuity.
\newblock {\em SIAM Journal on Optimization}, 32(3):2230--2262, 2022.

\bibitem{leinertial}
Khanh~Hien Le~Thi, Nicolas Gillis, and Panagiotis Patrinos.
\newblock Inertial block mirror descent method for non-convex non-smooth
  optimization.

\bibitem{li2019provable}
Qiuwei Li, Zhihui Zhu, Gongguo Tang, and Michael~B Wakin.
\newblock Provable bregman-divergence based methods for nonconvex and
  non-lipschitz problems.
\newblock {\em arXiv preprint arXiv:1904.09712}, 2019.

\bibitem{liu2020improved}
Yanli Liu, Yuan Gao, and Wotao Yin.
\newblock An improved analysis of stochastic gradient descent with momentum.
\newblock {\em Advances in Neural Information Processing Systems},
  33:18261--18271, 2020.

\bibitem{lu2019relative}
Haihao Lu.
\newblock “relative continuity” for non-lipschitz nonsmooth convex
  optimization using stochastic (or deterministic) mirror descent.
\newblock {\em INFORMS Journal on Optimization}, 1(4):288--303, 2019.

\bibitem{lu2018relatively}
Haihao Lu, Robert~M Freund, and Yurii Nesterov.
\newblock Relatively smooth convex optimization by first-order methods, and
  applications.
\newblock {\em SIAM Journal on Optimization}, 28(1):333--354, 2018.

\bibitem{mishchenko2020random}
Konstantin Mishchenko, Ahmed Khaled, and Peter Richt{\'a}rik.
\newblock Random reshuffling: Simple analysis with vast improvements.
\newblock {\em Advances in Neural Information Processing Systems},
  33:17309--17320, 2020.

\bibitem{mishchenko2022proximal}
Konstantin Mishchenko, Ahmed Khaled, and Peter Richt{\'a}rik.
\newblock Proximal and federated random reshuffling.
\newblock In {\em International Conference on Machine Learning}, pages
  15718--15749. PMLR, 2022.

\bibitem{mukkamala2020convex}
Mahesh~Chandra Mukkamala, Peter Ochs, Thomas Pock, and Shoham Sabach.
\newblock Convex-concave backtracking for inertial bregman proximal gradient
  algorithms in nonconvex optimization.
\newblock {\em SIAM Journal on Mathematics of Data Science}, 2(3):658--682,
  2020.

\bibitem{mukkamala2019bregman}
Mahesh~Chandra Mukkamala, Felix Westerkamp, Emanuel Laude, Daniel Cremers, and
  Peter Ochs.
\newblock Bregman proximal framework for deep linear neural networks.
\newblock {\em arXiv preprint arXiv:1910.03638}, 2019.

\bibitem{nemirovskij1983problem}
Arkadij~Semenovi{\v{c}} Nemirovskij and David~Borisovich Yudin.
\newblock Problem complexity and method efficiency in optimization.
\newblock 1983.

\bibitem{nesterov1983method}
Yurii Nesterov.
\newblock A method for solving the convex programming problem with convergence
  rate o (1/k2).
\newblock In {\em Dokl akad nauk Sssr}, volume 269, page 543, 1983.

\bibitem{nesterov2018lectures}
Yurii Nesterov et~al.
\newblock {\em Lectures on convex optimization}, volume 137.
\newblock Springer, 2018.

\bibitem{nguyen2019optimal}
Lam~M Nguyen, Marten van Dijk, Dzung~T Phan, Phuong~Ha Nguyen, Tsui-Wei Weng,
  and Jayant~R Kalagnanam.
\newblock Optimal finite-sum smooth non-convex optimization with sarah.
\newblock {\em arXiv preprint arXiv:1901.07648}, 2019.

\bibitem{pham2020proxsarah}
Nhan~H Pham, Lam~M Nguyen, Dzung~T Phan, and Quoc Tran-Dinh.
\newblock Proxsarah: An efficient algorithmic framework for stochastic
  composite nonconvex optimization.
\newblock {\em The Journal of Machine Learning Research}, 21(1):4455--4502,
  2020.

\bibitem{teboulle2018simplified}
Marc Teboulle.
\newblock A simplified view of first order methods for optimization.
\newblock {\em Mathematical Programming}, 170(1):67--96, 2018.

\bibitem{tseng2008accelerated}
Paul Tseng.
\newblock On accelerated proximal gradient methods for convex-concave
  optimization.
\newblock {\em submitted to SIAM Journal on Optimization}, 2(3), 2008.

\bibitem{wang2024bregman}
Qingsong Wang, Zehui Liu, Chunfeng Cui, and Deren Han.
\newblock A bregman proximal stochastic gradient method with extrapolation for
  nonconvex nonsmooth problems.
\newblock In {\em Proceedings of the AAAI Conference on Artificial
  Intelligence}, volume~38, pages 15580--15588, 2024.

\bibitem{wang2022gap}
Xinqi Wang, Qiwen Cui, and Simon~S Du.
\newblock On gap-dependent bounds for offline reinforcement learning.
\newblock {\em Advances in Neural Information Processing Systems},
  35:14865--14877, 2022.

\bibitem{wang2018spiderboost}
Zhe Wang, Kaiyi Ji, Yi~Zhou, Yingbin Liang, and Vahid Tarokh.
\newblock Spiderboost: A class of faster variance-reduced algorithms for
  nonconvex optimization.
\newblock {\em arXiv}, 2018, 2018.

\bibitem{xiao2021unified}
Xiantao Xiao.
\newblock A unified convergence analysis of stochastic bregman proximal
  gradient and extragradient methods.
\newblock {\em Journal of optimization theory and applications},
  188(3):605--627, 2021.

\bibitem{xie2020learning}
Qiaomin Xie, Yudong Chen, Zhaoran Wang, and Zhuoran Yang.
\newblock Learning zero-sum simultaneous-move markov games using function
  approximation and correlated equilibrium.
\newblock In {\em Conference on learning theory}, pages 3674--3682. PMLR, 2020.

\bibitem{zhang2018convergence}
Siqi Zhang and Niao He.
\newblock On the convergence rate of stochastic mirror descent for nonsmooth
  nonconvex optimization.
\newblock {\em arXiv preprint arXiv:1806.04781}, 2018.

\bibitem{zhou2019lower}
Dongruo Zhou and Quanquan Gu.
\newblock Lower bounds for smooth nonconvex finite-sum optimization.
\newblock In {\em International Conference on Machine Learning}, pages
  7574--7583. PMLR, 2019.

\bibitem{zhu2023unified}
Zhenyuan Zhu, Fan Chen, Junyu Zhang, and Zaiwen Wen.
\newblock A unified primal-dual algorithm framework for inequality constrained
  problems.
\newblock {\em Journal of Scientific Computing}, 97(2):39, 2023.

\end{thebibliography}
\end{document}